\documentclass[11pt,reqno]{amsart}
\usepackage{fullpage}
\usepackage{mathrsfs,amssymb,graphicx,verbatim,amsmath,amsfonts,amsbsy}
\usepackage{paralist}
\usepackage[breaklinks,pdfstartview=FitH]{hyperref}
\usepackage{upgreek}
\usepackage[mathscr]{euscript}
\newcommand{\cc}{\mathsf{c}}
\renewcommand{\le}{\leqslant}
\renewcommand{\ge}{\geqslant}

\renewcommand{\setminus}{\smallsetminus}
\newcommand{\M}{\mathsf{M}}
\newcommand{\Q}{\mathbb{Q}}
\newcommand{\ud}[0]{\,\mathrm{d}}
\usepackage{esint}
\usepackage[autostyle]{csquotes}

\usepackage{dutchcal}
\usepackage[T1]{fontenc}

\newcommand{\gp}{\gamma_{\!\pmb{+}}}
\newcommand{\spn}{\mathrm{span}}
\renewcommand{\pi}{\uppi}
\newcommand{\NN}{\mathcal{N}}

\newcommand{\ee}{\mathsf{e}}
\renewcommand{\phi}{\upphi}

\newcommand{\jj}{\mathcal{j}}

\newcommand{\sign}{\mathrm{sign}}

\newcommand{\n}{\{1,\ldots,n\}}

\renewcommand{\k}{\{1,\ldots,k\}}
\newcommand{\nsimplex}{\bigtriangleup^{\!n-1}}

\newcommand{\f}{\upphi}

\newcommand{\s}{\upsigma}

\renewcommand{\d}{\updelta}

\newcommand{\e}{\varepsilon}
\newcommand{\R}{\mathbb R}
\newcommand{\1}{\mathbf 1}

\newcommand{\Id}{\mathsf{Id}}
\newtheorem{theorem}{Theorem}
\newtheorem{lemma}[theorem]{Lemma}
\newtheorem{proposition}[theorem]{Proposition}

\newtheorem{corollary}[theorem]{Corollary}

\theoremstyle{remark}
\newtheorem{remark}[theorem]{Remark}

\newtheorem{conjecture}[theorem]{Conjecture}
\newtheorem{question}[theorem]{Question}
\newtheorem{example}[theorem]{Example}
\renewcommand{\tau}{\uptau}
\newcommand{\cM}{\mathcal{M}}
\newcommand{\MM}{\mathcal{M}}

\renewcommand{\xi}{\upxi}
\renewcommand{\rho}{\uprho}

\newcommand{\C}{\mathbb C}

\DeclareMathOperator{\supp}{supp}

\newcommand{\I}{\mathsf{I}}

\newcommand{\bI}{\mathbf{I}}

\newcommand{\N}{\mathbb N}
\newcommand{\Z}{\mathbb Z}

\newcommand{\eqdef}{\stackrel{\mathrm{def}}{=}}

\newcommand{\Lip}{\mathrm{Lip}}

\renewcommand{\theta}{\uptheta}
\renewcommand{\lambda}{\uplambda}

\DeclareMathOperator{\diam}{diam}

\newcommand{\A}{\mathsf{A}}

\newcommand{\G}{\mathsf{G}}
\renewcommand{\gamma}{\upgamma}
\renewcommand{\beta}{\upbeta}
\renewcommand{\alpha}{\upalpha}
\renewcommand{\kappa}{\upkappa}
\renewcommand{\psi}{\uppsi}
\renewcommand{\rho}{\uprho}
\renewcommand{\delta}{\updelta}
\renewcommand{\pi}{\uppi}
\renewcommand{\omega}{\upomega}
\renewcommand{\sigma}{\upsigma}
\renewcommand{\A}{\mathsf{A}}
\renewcommand{\eta}{\upeta}

\renewcommand{\kappa}{\upkappa}
\renewcommand{\mu}{\upmu}
\renewcommand{\nu}{\upnu}
\renewcommand{\pi}{\uppi}
\renewcommand{\zeta}{\upzeta}

\newcommand\smallmath[2]{#1{\raisebox{\dimexpr \fontdimen 22 \textfont 2
      - \fontdimen 22 \scriptscriptfont 2 \relax}{$\scriptscriptstyle #2$}}}

\newcommand\smallotimes{\smallmath\mathbin\otimes}

\begin{document}


\title{An average John theorem}
\dedicatory{Dedicated to the memory of Eli Stein}

\author{Assaf Naor}
\address{Mathematics Department\\ Princeton University\\ Fine Hall, Washington Road, Princeton, NJ 08544-1000, USA}
\email{naor@math.princeton.edu}

\thanks{Supported by the Packard Foundation and the Simons Foundation.  The research that is presented here was conducted under the auspices of the Simons Algorithms and Geometry (A\&G) Think Tank. An extended abstract~\cite{Nao17} announcing parts of this work appeared in the 33rd {I}nternational {S}ymposium on {C}omputational {G}eometry.}


\maketitle

\vspace{-0.29in}

\begin{abstract}
We prove that the  $\frac12$-snowflake of any finite-dimensional normed space $X$ embeds into a Hilbert space with quadratic average distortion $$O\Big(\sqrt{\log \dim(X)}\Big).$$ We deduce from this (optimal) statement that if an $n$-vertex expander  embeds with average distortion $D\ge 1$ into $X$, then necessarily $\dim(X)\ge n^{\Omega(1/D)}$, which is sharp  by the work of Johnson, Lindenstrauss and Schechtman (1987). This improves over the previously best-known bound $\dim(X)\gtrsim (\log n)^2/D^2$ of Linial, London and Rabinovich (1995),  strengthens a theorem of Matou\v{s}ek (1996) which resolved questions of Johnson and Lindenstrauss (1982), Bourgain (1985) and Arias-de-Reyna and Rodr{\'{\i}}guez-Piazza (1992), and answers negatively a question that was posed (for algorithmic purposes) by Andoni, Nguyen, Nikolov, Razenshteyn and Waingarten (2016).
\end{abstract}

\section{Introduction}

Given $D\ge 1$, we say that an infinite metric space $(\MM,d_\MM)$ embeds into a  normed space $(Z,\|\cdot\|_Z)$ with quadratic average distortion $D$ if for {\em every} Borel probability measure $\mu$ on $\MM$ there exists a $D$-Lipschitz mapping $f=f_\mu:\MM\to Z$ that satisfies
\begin{equation}\label{eq:def quadratic dist}
\iint_{\cM\times \cM} \|f(x)-f(y)\|_{\!Z}^{2\phantom{p}}\!\ud\mu(x)\ud\mu(y)\ge \iint_{\cM\times \cM} d_\cM(x,y)^2\ud \mu(x)\ud\mu(y).
\end{equation}
In comparison, the  requirement that $(\MM,d_\MM)$ embeds with bi-Lipschitz distortion $D$ into $(Z,\|\cdot\|_Z)$ means that there exists a $D$-Lipschitz mapping $f:\MM\to Z$ that satisfies
\begin{equation}\label{eq:worst case lower}
\forall\, x,y\in \MM,\qquad \|f(x)-f(y)\|_{\!Z}^{\phantom{p}}\!\ge d_\MM(x,y).
\end{equation}
Thus~\eqref{eq:def quadratic dist} is a natural average-case counterpart to the worst-case condition~\eqref{eq:worst case lower} where, in lieu of a canonical  probability measure on $\MM$, one demands that the notion of ``average'' is with respect to any Borel probability measure on $\MM$ while allowing the embedding to depend on the given measure.

The following theorem is (a special case of) our main result. Its statement uses the terminology (e.g~\cite{DS97}) that for $\omega\in (0,1]$, the $\omega$-snowflake of a metric space $(\MM,d_\MM)$ is the metric space $(\MM,d_\MM^\omega)$.

\begin{theorem}\label{thm:average john} For every integer $k\ge 2$, the $\frac12$-snowflake of any $k$-dimensional normed space embeds into a Hilbert space with quadratic average distortion $\mathsf{C}\sqrt{\log k}$, where $\mathsf{C}>0$ is a universal constant.
\end{theorem}

Compare Theorem~\ref{thm:average john} with John's classical theorem~\cite{Joh48} that $X$ embeds  into a Hilbert space with bi-Lipschitz\footnote{John's theorem is often stated in the literature with the embedding being a linear transformation, but this is equivalent to the way we stated it by passing to a derivative of the embedding, which exists almost everywhere~\cite{Gel38,BL00}.} distortion $\sqrt{k}$. This is sharp, as exhibited by $X=\ell_\infty^k$ or $X=\ell_1^k$. Power-type behavior is necessary also for bi-Lipschitz embeddings of snowflaked norms, as shown by the following lemma.

\begin{lemma}\label{lem:snowflake john} Fix $\omega\in (0,1]$ and $k\in \N$. The $\omega$-snowflake of any  $k$-dimensional normed space embeds with bi-Lipschitz distortion $k^{\frac{\omega}{2}}$ into a Hilbert space. Conversely, any embedding of the $\omega$-snowflake of $\ell_\infty^k$ into a Hilbert space incurs bi-Lipschitz distortion at least a universal constant multiple of $k^{\frac{\omega}{2}}$.
\end{lemma}

In contrast to Lemma~\ref{lem:snowflake john}, in Theorem~\ref{thm:average john} we establish that if one wishes to obtain an embedding into a Hilbert space which is $\frac12$-H\"older  and preserves the $\frac12$-snowflaked distances only on average rather than the worst-case pairwise distance preservation requirement  of John's theorem (or its version for snowflakes that appears in Lemma~\ref{lem:snowflake john}, which shows  that for $\frac12$-snowflakes the best bi-Lipschitz distortion that one could hope for is of order $\sqrt[4]{k}$), then  the distortion can be improved dramatically to a universal constant multiple of $\sqrt{\log k}$.  Importantly, the notion of ``average'' here can be taken to be with respect to {any} Borel probability measure on $X$ whatsoever.

Theorem~\ref{thm:average john} is sharp in two ways. Firstly, we will see that its $\mathsf{C}\sqrt{\log k}$ bound is sharp (this occurs when $X=\ell_\infty^k$ and the probability measure is uniform over an isometrically embedded $k$-vertex expander). Secondly, one cannot perform a lesser amount of snowflaking of the norm while still obtaining average distortion  $k^{o(1)}$. Namely, we will see that if $\e\in (0,\frac12]$ and one aims to embed  the $\left(\frac12+\e\right)$-snowflake of every such $X$ into a Hilbert space with quadratic average distortion $D$, then necessarily $D\ge k^\e$ (this occurs when $X=\ell_1^k$ and the probability measure is uniform over $\{0,1\}^k$).

Thus, the exponential   improvement over John's distortion bound that we obtain in Theorem~\ref{thm:average john} is made possible by allowing the  distances to be preserved only on average, and simultaneously  introducing an inherent  nonlinearity through snowflaking; performing only  one of these two modifications of John's theorem does not suffice. We will soon see that, despite the fact that the distance preservation guarantee that is furnished by Theorem~\ref{thm:average john} is (necessarily) weaker than that of John's theorem, it has quite substantial implications. It is also worthwhile to note that unlike John's embedding, which is achieved explicitly by considering the ellipsoid of maximal volume that is inscribed in the unit ball of $X$, our proof of Theorem~\ref{thm:average john} establishes  the existence of the stated embedding implicitly through reliance on a duality argument; obtaining a more constructive proof would be valuable.

Theorem~\ref{thm:average john}  is in fact  a special case of a stronger theorem  that treats embeddings into targets that are not necessarily Hilbertian, $L_p$ variants of the quadratic requirement~\eqref{eq:def quadratic dist}, and other snowflakes of $X$, and it also obtains improved embeddings (i.e.~with less snowflaking)  if $X$ satisfies an additional geometric assumption; see Theorem~\ref{thm:really main}  below. It is beneficial to start by presenting the above basic version (quadratic, Hilbertian, without any assumption on the geometry of $X$) because it does not require the introduction of further terminology, and it has  a noteworthy geometric consequence that we wish to explain first, prior to passing to the somewhat more involved setup of Section~\ref{sec:UC} below.

\begin{remark} In the spirit of Theorem~\ref{thm:average john}, it is simple to find other examples of metric spaces $(\MM,d_\MM)$ whose quadratic average distortion into some Banach space $Z$ is significantly smaller than their bi-Lipschitz distortion into $Z$. Indeed, it is straightforward to check that if $(\MM,d_\MM)$ is an infinite equilateral space, i.e., $d_\MM(x,y)=1$ for all distinct $x,y\in \MM$, then $\MM$ embeds into $Z=\R$ with finite quadratic average distortion, but $\MM$ does not admit a bi-Lipschitz embedding into $\R^n$ for any $n\in \N$. Much more substantially, any weighted planar graph (equipped with its shortest-path metric) or any $O(1)$-doubling metric space (see~\cite{Hei01}) embeds into the real line with $O(1)$-quadratic average distortion (see~\cite[Section~7]{Nao14} for a justification of this, which adapts the reasoning in~\cite{Rab08}), while such spaces need not even admit a bi-Lipschitz embedding into a Hilbert space~\cite{Bou86,Laa02}. Also, if $2<p<\infty$, then $\ell_p$ does not admit a bi-Lipschitz embedding into a Hilbert space (see~\cite{BL00}), but it follows from~\cite{Nao14} that $\ell_p$ embeds into a Hilbert space with quadratic average distortion $O(p)$, and that this bound is optimal. More such examples  will be obtained below.
\end{remark}

\subsection{Notation, terminology, conventions}\label{sec:notation} Below, all metric spaces will be tacitly assumed to be separable. While some of the ensuing statements hold without a separability assumption, adhering to this convention avoids measurability side-issues that would otherwise obscure the main geometric content. Alternatively, one could harmlessly consider throughout only  finitely supported measures.

In addition to the usual $O(\cdot), o(\cdot), \Omega(\cdot)$ notation, we will use the following (also standard) asymptotic notation. For $Q,Q'>0$, the notations
$Q\lesssim Q'$ and $Q'\gtrsim Q$ mean that $Q\le \mathsf{K}Q'$ for a
universal constant $\mathsf{K}>0$. The notation $Q\asymp Q'$
stands for $(Q\lesssim Q') \wedge  (Q'\lesssim Q)$. If  we need to allow for dependence on parameters, we indicate this by subscripts. For example, in the presence of auxiliary parameters $\psi,\xi$, the notation $Q\lesssim_{\psi,\xi} Q'$ means that $Q\le c(\psi,\xi)Q' $, where $c(\psi,\xi) \in(0,\infty)$ may depend only on $\psi$ and $\xi$, and analogously for the notations $Q\gtrsim_{\psi,
\xi} Q'$ and $Q\asymp_{\psi,\xi} Q'$.

 We will use notions of Banach spaces~\cite{LT77,LT79}, metric embeddings~\cite{Mat02,Ost13}  and expanders~\cite{HLW06,AS16}. Any undefined term in the ensuing discussion is entirely standard and appears in the aforementioned references, but in this short subsection  we recall a modicum of simple concepts.

For a normed space $(Z,\|\cdot\|_Z)$ and $p\ge 1$, the normed space $\ell_p(Z)$ consists of all those  $Z$-valued sequences $x=(x_1,x_2,\ldots)\in Z^{\N}$ such that $\|x\|_{\ell_p(Z)}^p=\sum_{i=1}^\infty \|x_i\|_Z^p<\infty$. One writes $\ell_p(\R)=\ell_p$.

Theorem~\ref{thm:average john} tensorizes in a straightforward manner to give the same conclusion for $\ell_1(X)$. In order to facilitate later reference, it is beneficial to record this  fact as the following separate statement.
\begin{corollary}\label{cor:ell1X} For any normed space  $X$ of dimension $k\ge 2$, the  $\frac12$-snowflake of $\ell_1(X)$ embeds into a Hilbert space with quadratic average distortion $\mathsf{C}\sqrt{\log k}$, where $\mathsf{C}>0$ is a universal constant.
\end{corollary}

\begin{proof} For every $i\in \N$ let $\mathfrak{c}_i:\ell_1(X)\to X$ denote the $i$'th coordinate projection, i.e., $\mathfrak{c}_i(x)=x_i$ for each $x=(x_1,x_2,\ldots,)\in \ell_1(X)$. Fix  any Borel probability measure $\mu$ on $\ell_1(X)$. For each $i\in \N$, an application of Theorem~\ref{thm:average john}  to the measure $(\mathfrak{c}_i)_\sharp \mu$ on $X$  (the image of $\mu$ under $\mathfrak{c}_i$) yields  $f_i:X\to H$ which is $\frac12$-H\"older with constant $\mathsf{C}\sqrt{\log k}$, where $(H,\|\cdot\|_H)$ is a Hilbert space, that satisfies
$$\iint_{\ell_1(X)\times \ell_1(X)} \big\|f_i\big(\mathfrak{c}_i(x)\big)-f_i\big(\mathfrak{c}_i(y)\big)\big\|_{\!H}^{\phantom{p}\!\!\!2}\ud\mu(x)\ud\mu(y)\ge\iint_{\ell_1(X)\times \ell_1(X)} \|\mathfrak{c}_i(x)-\mathfrak{c}_i(y)\|_{\!X}^{\phantom{p}}\ud\mu(x)\ud\mu(y).$$
The desired embedding $f:\ell_1(X)\to \ell_2(H)$ is now defined by $f(x)=(f_1(x_1),f_2(x_2),\ldots)$.
\end{proof}

The {bi-Lipschitz distortion} of a metric space $(\MM,d_\MM)$ in a normed space $(Z,\|\cdot\|_Z)$ is a numerical  invariant  denoted
 $\cc_Z(\MM)$ that is defined to be the infimum over those $D\in [1,\infty]$ for which there exists a $D$-Lipschitz mapping $f:\MM\to Z$ satisfying $\|f(x)-f(y)\|_Z\ge d_\MM(x,y)$ for all $x,y\in \MM$.

The most natural setting to discuss average distortion of embeddings is that of metric probability spaces, namely  triples $(\MM,d_\MM,\mu)$ where $(\MM,d_\MM)$ is a metric space and $\mu$ is a Borel probability measure on $\MM$. In this context, as an obvious variant of~\eqref{eq:def quadratic dist}, for $p>0$ and $D\ge 1$, say that $(\MM,d_\MM,\mu)$ embeds with $p$-average distortion $D$ into a Banach space $(Z,\|\cdot\|_Z)$ if there is a $D$-Lipschitz mapping $f:\MM\to Z$ such that $\iint_{\MM\times \MM} \|f(x)-f(y)\|_{\!Z}^p\ud\mu(x)\ud\mu(y)\ge \iint_{\MM\times \MM} d_\MM(x,y)^p\ud\mu(x)\ud\mu(y)$. If this holds with  $p=1$, one simply says that $(\MM,d_\MM,\mu)$ embeds with average distortion $D$ into $(Z,\|\cdot\|_Z)$.

When a  {\em finite} metric space $(\MM,d_\MM)$ is said to embed with $p$-average distortion $D$ into a Banach space $(Z,\|\cdot\|_Z)$ without explicitly specifying the underlying probability measure $\mu$, it will always be understood that $\mu$ is the uniform probability measure on $\MM$. Embeddings  of finite metric spaces with controlled average distortion have several interesting applications, and their systematic investigation was initiated by Rabinovich~\cite{Rab08}. If $(\MM,d_\MM)$ is an {\em infinite} metric space, then  when we say that it embeds with $p$-average distortion $D$ into $(Z,\|\cdot\|_Z)$ we mean that  for {\em every} probability measure $\mu$ on $\MM$ the metric probability space $(\MM,d_\MM,\mu)$ embeds with $p$-average distortion $D$ into $(Z,\|\cdot\|_Z)$. The difference between the terminology for finite and infinite spaces is natural because finite spaces carry a canonical probability (counting) measure while infinite spaces do not. We chose these conventions  so as to  be consistent with the terminology in the literature, which only treats finite spaces.

Using the above terminology, we record for ease of later reference the following immediate consequence of Corollary~\ref{cor:ell1X} (with the universal constant $\mathsf{C}\in [1,\infty)$ the same).

\begin{corollary}\label{cor:transfer to Hilbert}
Suppose that $(X,\|\cdot\|_X)$ is a normed space of dimension $k\ge 2$ and that $(\MM,d_\MM,\mu)$ is a metric probability space that embeds into $\ell_1(X)$ with average distortion $D$. Then, $(\MM,\sqrt{d_\MM},\mu)$ embeds into a Hilbert space with quadratic  average distortion $\mathsf{C}\sqrt{D\log k}$.
\end{corollary}
The following proposition demonstrates that the notion of average distortion is robust to changes of the moments of distances that one wishes to approximately preserve, as well as to snowflaking.

\begin{proposition}\label{prop:other exponents} Fix $p,q,D\in [1,\infty)$ and $\omega\in (0,1]$. Suppose that an infinite separable metric space $(\MM,d_\MM)$ embeds with $p$-average distortion $D$ into a Banach space $(Y,\|\cdot\|_Y)$.  Then, the $\omega$-snowflake of $(\MM,d_\MM)$ embeds with $q$-average distortion $D'=D'(p,q,\omega)\ge 1$ into $(Y,\|\cdot\|_Y)$, where
\begin{equation}\label{eq:our Delta in changed exponent}
D'\lesssim_{p,q,\omega} D^{\max\left\{\frac{p}{q},\omega\right\}}.
\end{equation}
\end{proposition}
We postpone discussion of Proposition~\ref{prop:other exponents} to Section~\ref{sec:normalization} below, where it is proved and the implicit dependence on $p,q,\omega$   in~\eqref{eq:our Delta in changed exponent} is specified; see~\eqref{eq:changed distortion explicit}. It suffices to say here that Proposition~\ref{prop:other exponents} shows that phenomena such as Theorem~\ref{thm:average john} (as well as more refined  results that we will soon state), in which upon performing a certain  amount of snowflaking the average distortion decreases from power-type behavior to logarithmic behavior, are independent of the choice of ``$p$'' in the notion of $p$-average distortion that one considers, and they persist if one performs an even greater amount of snowflaking.

Given $n\in \N$, let  $\bigtriangleup^{\!n-1}=\{\pi=(\pi_1,\ldots,\pi_n)\in [0,1]^n:\ \sum_{i=1}^n\pi_i=1\}$ denote the simplex of probability measures on $\n$. When we say that a matrix $\A=(a_{ij})\in \M_n(\R)$ is stochastic we always mean row-stochastic, i.e., $(a_{i1},\ldots,a_{in})\in  \bigtriangleup^{\!n-1}$ for every $i\in \n$. Given  $\pi\in \bigtriangleup^{\!n-1}$, a stochastic matrix $\A=(a_{ij})\in \M_n(\R)$ is $\pi$-reversible if $\pi_ia_{ij}=\pi_ja_{ji}$ for every $i,j\in \n$. In this case, $\A$ is a self-adjoint contraction on $L_2(\pi)$ and the decreasing rearrangement of the eigenvalues of $\A$ is denoted $1=\lambda_1(\A)\ge\ldots\ge \lambda_n(\A)\ge -1$. The spectral gap $1-\lambda_2(\A)$  can be interpreted by straightforward linear algebra (expanding the squares and expressing in an eigenbasis of $\A$) as the largest factor (in the left hand side) for which the following   quadratic inequality holds true.
\begin{equation}\label{eq:energy}
\forall\, x_1,\ldots,x_n\in \ell_2,\qquad \big(1-\lambda_2(\A)\big)\sum_{i=1}^n\sum_{j=1}^n\pi_i\pi_j\|x_i-x_j\|_{\ell_2}^2\le \sum_{i=1}^n\sum_{j=1}^n \pi_i a_{ij}\|x_i-x_j\|_{\ell_2}^2.
\end{equation}

If $\G=(\n,E_\G)$ is a connected graph, then the shortest-path metric that it induces is denoted $d_\G:\n\times \n\to \N\cup\{0\}$. If $\G$ is $\Delta$-regular for some $\Delta\in \{2,\ldots,n\}$, then
 the normalized adjacency matrix of $\G$, denoted $\A_\G\in \M_n(\R)$, is the symmetric stochastic matrix whose entry at $(i,j)\in \n\times \n$ is equal to $\frac{1}{\Delta}\1_{\{i,j\}\in E_\G}$. Write $\lambda_2(\G)=\lambda_2(\A_\G)$.


\subsection{A spectral gap is an obstruction to metric dimension reduction}\label{sec:dim reduction intro}  For every $n\in \N$ there is a $O(1)$-regular graph $\G_n=(\n,E_{\G_n})$ with $1/(1-\lambda_2(\G_n))=O(1)$. See the survey~\cite{HLW06} for this statement and much more on such {\em  expanders}. In particular, it is well-known that an argument of Linial, London and Rabinovich~\cite{LLR95} gives that if the $\frac12$-snowflake of $(\n,d_{\G_n})$ embeds with quadratic average distortion $D\ge 1$ into a Hilbert space, then necessarily $D\gtrsim \sqrt{\log n}$.

We will next recall why this nonembeddability statement holds, following an influential formulation of the approach of~\cite{LLR95} due to Matou\v{s}ek~\cite{Mat97} and Gromov~\cite{Gro00}. Before doing so, note that this establishes the aforementioned optimality of the distortion bound of Theorem~\ref{thm:average john}, since the  Fr\'echet embedding~\cite{Fre06} yields an $n$-point subset $S$ of $X=\ell_\infty^{n}$ that is isometric to $(\n,d_{\G_n})$, and therefore if $\mu$ is the uniform measure on $S$, then the quadratic average distortion of any embedding of the $\frac12$-snowflake of  $(X,\|\cdot\|_{\ell_\infty^n},\mu)$ is at least a universal constant multiple of $\sqrt{\log n}=\sqrt{\log \dim(X)}$.

So, fix an integer $n\ge 4$ and $\Delta\in \{3,\ldots,n-1\}$. Suppose that $\G=(\n,E_\G)$ is a connected $\Delta$-regular graph.  Let $(H,\|\cdot\|_H)$ be a Hilbert space and assume that $f:\n\to H$ satisfies
\begin{equation}\label{eq:quad for expander}
\sum_{i=1}^n\sum_{j=1}^n\|f(i)-f(j)\|_{\!H}^{\phantom{p}\!\!\!2}\ge\sum_{i=1}^n\sum_{j=1}^nd_\G(i,j)\qquad\mathrm{and}\qquad \forall\{i,j\}\in E_\G,\quad \|f(i)-f(j)\|_{\!H}^{\phantom{p}}\le D.
\end{equation}
By a simple and standard counting argument  (e.g.~\cite[page~193]{Mat97}), a positive universal constant fraction of the pairs of vertices $(i,j)\in \n\times \n$ satisfy $d_\G(i,j)\gtrsim \frac{\log n}{\log\Delta}$. Hence,
\begin{equation}\label{eq:average dist power 1}
\frac{1}{n^2}\sum_{i=1}^n\sum_{j=1}^n\|f(i)-f(j)\|_{\!H}^{\phantom{p}\!\!\!2}\stackrel{\eqref{eq:quad for expander}}{\ge}\frac{1}{n^2}\sum_{i=1}^n\sum_{j=1}^nd_\G(i,j)\gtrsim \frac{\log n}{\log \Delta}.
\end{equation}
With this observation, the average distortion $D$ can be bounded from below through an application of the interpretation~\eqref{eq:energy} of a spectral gap to the normalized adjacency matrix of $\G$, as follows.

\begin{equation}\label{eq:penultimate}
D^2\stackrel{\eqref{eq:quad for expander}}{\ge} \frac{1}{|E_\G|}\sum_{\{i,j\}\in E_\G}\|f(i)-f(j)\|_{\!H}^{\phantom{p}\!\!\!2}\stackrel{\eqref{eq:energy}}{\ge} \frac{1-\lambda_2(\G)}{n^2}\sum_{i=1}^n\sum_{j=1}^n\|f(i)-f(j)\|_{\!H}^{\phantom{p}\!\!\!2}\stackrel{ \eqref{eq:average dist power 1}}{\gtrsim} \frac{1-\lambda_2(\G)}{\log \Delta}\log n,
\end{equation}
We have thus shown (following~\cite{LLR95,Mat97,Gro00}) that
\begin{equation}\label{eq:quad dist for 1/2 snow}
D\gtrsim \frac{\sqrt{1-\lambda_2(\G)}}{\sqrt{\log \Delta}}\sqrt{\log n}.
\end{equation}
So, $D\gtrsim \sqrt{\log n}$ when $\frac{1}{1-\lambda_2(\G)}\lesssim 1$ and $\Delta\lesssim 1$, i.e., for  expanders. In general, we have

\begin{theorem}\label{thm:l1X} Fix $D\ge 1$ and integers $n,\Delta\ge 3$ with $\Delta\le n$. Let $\G=(\n,E_\G)$ be a $\Delta$-regular connected graph. Suppose that $(X,\|\cdot\|_X)$ is a finite-dimensional normed space such that the metric space $(\n,d_\G)$ embeds with average distortion $D$ into $\ell_1(X)$. Then necessarily
\begin{equation}\label{eq:dim lower first}
\dim(X)\ge n^{\frac{\eta(\G)}{D}},\qquad \mathrm{where}\qquad  \eta(\G)\gtrsim \frac{1-\lambda_2(\G)}{\log \Delta}.
\end{equation}
\end{theorem}

\begin{proof} By combining~\eqref{eq:quad dist for 1/2 snow} with Corollary~\ref{cor:transfer to Hilbert}, it follows that
$
\sqrt{D\log\dim(X)}\gtrsim \frac{\sqrt{1-\lambda_2(\G)}}{\sqrt{\log \Delta}}\sqrt{\log n}.
$
\end{proof}

\begin{remark} The reasoning by which we deduced~\eqref{eq:quad dist for 1/2 snow} from~\eqref{eq:quad for expander} did not use the entirety of the second condition of~\eqref{eq:quad for expander}.  Namely, in addition to the first inequality in~\eqref{eq:quad for expander},  it suffices to assume that
\begin{equation}\label{eq:W12}
\bigg(\frac{1}{|E_\G|}\sum_{\{i,j\}\in E_\G} \|f(i)-f(j)\|_{\!H}^{\phantom{p}\!\!\!2}\bigg)^{\!\frac12}\le D.
\end{equation}
 Thus, we only  need to have an upper bound on the  discrete Sobolev $W^{1,2}$ norm  of the embedding $f$, namely the left hand side of~\eqref{eq:W12}, rather than an upper bound on the Lipschitz constant of $f$. By using this variant in place of the average distortion assumption in Theorem~\ref{thm:l1X}, one deduces mutatis mutandis that the same conclusion~\eqref{eq:dim lower first} holds true if there exists $f:\n\to \ell_1(X)$ that satisfies the following two requirements (which can be described, respectively, as ``quantitative invertibility on average,'' combined with a bound on the discrete  Sobolev $W^{1,1}$ norm).
\begin{equation*}
\frac{1}{n^2}\sum_{i=1}^n\sum_{j=1}^n \|f(i)-f(j)\|_{\ell_1(X)}^{\phantom{p}}\ge \frac{1}{n^2}\sum_{i=1}^n\sum_{j=1}^nd_\G(i,j),
\end{equation*}
and
\begin{equation*}
 \frac{1}{|E_\G|}\sum_{\{i,j\}\in E_\G} \|f(i)-f(j)\|_{\ell_1(X)}^{\phantom{p}}\le D.
\end{equation*}
In words, if $\G$ has a  spectral gap and  large average distance, and one is given  $x_1,\ldots,x_n\in X$ for which the  averages of the two sets of distances $\{\|x_i-x_j\|_X\}_{(i,j)\in \n\times \n}$ and $\{\|x_i-x_j\|_X\}_{\{i,j\}\in E_\G}$ are within a fixed but potentially large constant factor from the corresponding averages of distances in $\G$, then this crude geometric information about two specific ``distance statistics'' of the  finite point configuration  $\{x_1,\ldots,x_n\}\subset X$ forces the continuous ambient space $X$ to be high-dimensional.
\end{remark}

The natural variant of Theorem~\ref{thm:l1X} for $\ell_1$ products of $k$-dimensional normed spaces $\{(X_i,\|\cdot\|_{X_i})\}_{i=1}^\infty$ that need not all be the same space $(X,\|\cdot\|_X)$ holds mutatis mutandis by the same reasoning. Also, a version with $\ell_1(X)$ replaced by $\ell_p(X)$ for any $p\ge 1$ appears in Section~\ref{sec:matirx} below, where it is explained that this seemingly more general setting is, in fact, a formal consequence of its counterpart for $\ell_1(X)$ that we  deduced above from Theorem~\ref{thm:average john}. Section~\ref{sec:history} is a description of the history of the questions that Theorem~\ref{thm:l1X} resolves, as well as an indication of subsequent algorithmic developments that rely on it and were found since the initial posting of a preliminary version of the present work.

\subsubsection{A matrix-dimension inequality}\label{sec:matirx} Let $(X,\|\cdot\|_X)$ be a normed space of dimension $k\ge 2$. Fix $n\in \N$ and $\pi \in\nsimplex$. For every $n$-tuple of vectors $x_1,\ldots, x_n\in X$ we can apply Theorem~\ref{thm:average john} to obtain a function $f=f_{\pi,x_1,\ldots,x_n}:X\to H$, where $(H,\|\cdot\|_H)$ is a Hilbert space, satisfying
\begin{equation}\label{eq:apply john xi}
\forall\, x,y\in X,\qquad \|f(x)-f(y)\|_{\!H}^{\phantom{p}}\le \mathsf{C}\sqrt{\log k}\cdot\sqrt{\|x-y\|_{\!X}^{\phantom{'}}},
\end{equation}
as well as
\begin{equation}\label{eq:apply john pi}
\sum_{i=1}^n\sum_{j=1}^n \pi_j\pi_j\|f(x_i)-f(x_j)\|_{\!H}^{\phantom{p}\!\!\!2}\ge \sum_{i=1}^n\sum_{j=1}^n \pi_j\pi_j\|x_i-x_j\|_{\!X}^{\phantom{p}}.
\end{equation}
 If $\A=(a_{ij})\in \M_n(\R)$ is a stochastic and $\pi$-reversible matrix, then
\begin{align}\label{eq:use John}
\begin{split}
\mathsf{C^2}\log k &\stackrel{\eqref{eq:apply john xi}}{\ge}  \frac{\sum_{i=1}^n\sum_{j=1}^n \pi_ia_{ij}\|f(x_i)-f(x_j)\|_{\!H}^2}{\sum_{i=1}^n\sum_{j=1}^n \pi_ia_{ij}\|x_i-x_j\|_{\!X}^{\phantom{p}}}\\
&\stackrel{\eqref{eq:energy}}{\ge}  \big(1-\lambda_2(\A)\big) \frac{\sum_{i=1}^n\sum_{j=1}^n \pi_i\pi_j\|f(x_i)-f(x_j)\|_{\!H}^2}{\sum_{i=1}^n\sum_{j=1}^n \pi_ia_{ij}\|x_i-x_j\|_{\!X}^{\phantom{p}}}\\&\!\stackrel{\eqref{eq:apply john pi}}{\ge}
\big(1-\lambda_2(\A)\big) \frac{\sum_{i=1}^n\sum_{j=1}^n \pi_i\pi_j\|x_i-x_j\|_{\!X}^{\phantom{p}}}{\sum_{i=1}^n\sum_{j=1}^n \pi_ia_{ij}\|x_i-x_j\|_{\!X}^{\phantom{p}}},
\end{split}
\end{align}
 Above, and  henceforth, the ratios that appear in~\eqref{eq:use John} are interpreted to be  equal to $0$ when  their denominator   $\sum_{i=1}^n\sum_{j=1}^n \pi_ia_{ij}\|x_i-x_j\|_{\!X}^{\phantom{p}}$ vanishes (this ``disconnected'' case will never be of interest).

We have thus established that, as a consequence of Theorem~\ref{thm:average john}, every finite configuration of points $x_1,\ldots,x_n\in X$ imposes the following geometric lower bound on the dimension of the ambient space.
\begin{equation}\label{eq:C2}
\dim(X)\ge \exp\!\bigg(\frac{1-\lambda_2(\A)}{\mathsf{C}^2}\cdot \frac{\sum_{i=1}^n\sum_{j=1}^n \pi_i\pi_j \|x_i-x_j\|_{\!X}^{\phantom{p}}}{\sum_{i=1}^n\sum_{j=1}^n \pi_ia_{ij}\|x_i-x_j\|_{\!X}^{\phantom{p}}}\bigg).
\end{equation}
The $\ell_1$ bound~\eqref{eq:C2}  formally implies the following $\ell_p$ counterpart for every $p\ge 1$.
\begin{equation}\label{eq:beta p version}
\dim(X)\ge \exp\!\Bigg(\frac{1-\lambda_2(\A)}{\beta(p)}\bigg(\frac{\sum_{i=1}^n\sum_{j=1}^n \pi_i\pi_j \|x_i-x_j\|_{\!X}^p}{\sum_{i=1}^n\sum_{j=1}^n \pi_ia_{ij}\|x_i-x_j\|_{\!X}^p}\bigg)^{\!\!\frac{1}{p}}\Bigg),
\end{equation}
where the constant  $\beta(p)>0$  depends only on $p$. This is so because, by Proposition~\ref{prop:other exponents}, it is a formal consequence of Theorem~\ref{thm:average john} that there is also an Hilbertian embedding of the $\frac12$-snowflake of $X$  with $(2p)$-average distortion $\mathsf{C}(p)\sqrt{\log k}$ for every $p\ge 1$ (the special case of Proposition~\ref{prop:other exponents} that was used here is due to~\cite[Section~7.4]{Nao14}). By substituting this into the above reasoning, one arrives at~\eqref{eq:beta p version}.

While the asymptotic dependence in (the above use of) Proposition~\ref{prop:other exponents} that we obtain in Section~\ref{sec:normalization} improves over what is available in the  literature, it does not yield the sharp dependence of $\beta(p)$ on $p$ as $p\to \infty$, and more care is needed in order to derive Theorem~\ref{thm:lp version Kp} below, which we will prove in Section~\ref{sec:pass to larger p}. It obtains what we expect to be the sharp asymptotic dependence on $p$, though at present we do not see a proof of this; the desired optimality would follow from Conjecture~\ref{conj:p depend} below.

\begin{theorem}\label{thm:lp version Kp} There is a universal constant $\mathsf{K}>0$ such that for every $p\ge 1$ and $n\in \N$, if $(X,\|\cdot\|_X)$ is a normed space, $\pi \in \nsimplex$ and $\A=(a_{ij})\in \M_n(\R)$ is a stochastic $\pi$-reversible matrix, then
\begin{equation}\label{eq:matrix dim p sharp}
\forall\, x_1,\ldots,x_n\in X,\qquad \dim(X)\ge \exp\!\Bigg(\frac{1-\lambda_2(\A)}{\mathsf{K}p}\bigg(\frac{\sum_{i=1}^n\sum_{j=1}^n \pi_i\pi_j \|x_i-x_j\|_{\!X}^p}{\sum_{i=1}^n\sum_{j=1}^n \pi_ia_{ij}\|x_i-x_j\|_{\!X}^p}\bigg)^{\!\!\frac{1}{p}}\Bigg).
\end{equation}
\end{theorem}
We thus obtain the following variant of Theorem~\ref{thm:l1X} in which $\ell_1(X)$ is replaced by $\ell_p(X)$ for $p\ge 1$.

\begin{corollary}\label{thm:lpX} Fix $D,p\ge 1$ and integers $n,\Delta\ge 3$ with $\Delta\le n$. Let $\G=(\n,E_\G)$ be a $\Delta$-regular connected graph. Suppose that $(X,\|\cdot\|_X)$ is a normed space such that $(\n,d_\G)$ embeds with average distortion $D$ into $\ell_p(X)$. Then, $\dim(X)\ge n^{\eta(\G)/(pD)}$ for $\eta(\G)>0$ as in~\eqref{eq:dim lower first}.
\end{corollary}

\begin{proof} For $f:\n\to \ell_p(X)$  and $m\in \N$ denote the $m$'th entry of $f$ by $f_m:\n\to X$, i.e., for every $i\in \n$ we have $f(i)=(f_1(i),f_2(i),\ldots)\in \ell_p(X)$. Then, by Theorem~\ref{thm:lp version Kp} applied to the finite configuration $\{f_m(1),\ldots,f_m(n)\}$ of points in $X$ for each $m\in \N$ separately, we see that
\begin{equation*}
\dim(X)\ge \exp\!\Bigg(\frac{1-\lambda_2(\A)}{\mathsf{K}p}\sup_{m\in \N} \bigg(\frac{\frac{1}{n^2}\sum_{i=1}^n\sum_{j=1}^n \|f_m(i)-f_m(j)\|_{\!X}^p}{\frac{1}{n}\sum_{i=1}^n\sum_{j=1}^n a_{ij}\|f_m(i)-f_m(j)\|_{\!X}^p}\bigg)^{\!\!\frac{1}{p}}\Bigg).
\end{equation*}
It remains to observe that if  $f:\n\to \ell_p(X)$ is $D$-Lipschitz (with respect to the shortest-path metric $d_\G$), yet $\sum_{i=1}^n\sum_{j=1}^n \|f(i)-f(j)\|_{\ell_p(X)}\ge\sum_{i=1}^n\sum_{j=1}^nd_\G(i,j)$, then

\begin{multline*}
 \sup_{m\in \N}\bigg(\frac{\frac{1}{n^2}\sum_{i=1}^n\sum_{j=1}^n \|f_m(i)-f_m(j)\|_{\!X}^p}{\frac{1}{n}\sum_{i=1}^n\sum_{j=1}^n a_{ij}\|f_m(i)-f_m(j)\|_{\!X}^p}\bigg)^{\!\!\frac{1}{p}}\ge \bigg(\frac{\frac{1}{n^2}\sum_{i=1}^n\sum_{j=1}^n  \|f(i)-f(j)\|_{\ell_p(X)}^p}{\frac{1}{n}\sum_{i=1}^n\sum_{j=1}^n a_{ij}\|f(i)-f(j)\|_{\ell_p(X)}^p}\bigg)^{\!\!\frac{1}{p}}\\\ge \frac{1}{Dn^2}\sum_{i=1}^n\sum_{j=1}^n \|f(i)-f(j)\|_{\ell_p(X)}\ge\frac{1}{Dn^2}\sum_{i=1}^n\sum_{j=1}^nd_\G(i,j)\gtrsim \frac{\log n}{D\log\Delta},
\end{multline*}
where the first step holds because $\frac{a_1+a_2+\ldots}{b_1+b_2+\ldots}\le \sup_{m\in \N}\frac{a_m}{b_m}$ for any  $\{a_m\}_{m=1}^\infty,\{b_m\}_{m=1}^\infty\subset (0,\infty)$, the second step is an application of Jensen's inequality in the numerator and the $D$-Lipschitz condition in the denominator, the third step is our second assumption on $f$, and the final step is~\eqref{eq:average dist power 1}.
\end{proof}

We expect that Corollary~\ref{thm:lpX} is sharp in terms of its dependence on $p$, as expressed in Conjecture~\ref{conj:p depend} below, whose positive resolution might have algorithmic applications; this could be quite tractable by adapting available methods, specifically those of~\cite{JLS87,Mat92,Mat96,Mat97,ABN11}.

\begin{conjecture}\label{conj:p depend} There is a universal constant $\mathsf{C}\ge 1$ with the following property. For every $p,D\ge 1$ there exists $n_0=n_0(p,D)\in \N$ such that if $n\ge n_0$, then for every $n$-point metric space $(\MM,d_\MM)$ there exists $k\in \N$ with $k\le n^{\mathsf{C}/(pD)}$ and a $k$-dimensional normed space $(X,\|\cdot\|_X)$ such that $\MM$ embeds with bi-Lipschitz distortion $D$ into $\ell_p(X)$. Conceivably this  even holds true for $X=\ell_\infty^k$.
\end{conjecture}

\subsection{Uniform convexity and smoothness}\label{sec:UC} Henceforth, the closed unit ball of a normed space $(X,\|\cdot\|_X)$ will be denoted by $B_X=\{x\in X:\, \|x\|_X\le 1\}$. The moduli~\cite{Day44} of uniform convexity and uniform smoothness of $(X,\|\cdot\|_X)$, commonly denoted $\d_X:[0,2]\to [0,\infty)$ and $\rho_X:[0,\infty)\to [0,\infty)$, respectively, are the (point-wise) smallest such functions for which every $x,y\in \partial B_X$ and $\tau\in [0,\infty)$ satisfy $\|x+y\|_X\le 2(1-\d_X(\|x-y\|_X))$ and $\|x+\tau y\|_X+\|x-\tau y\|_X\le 2(1+\rho_X(\tau))$.

Given $p,q\in [1,\infty)$, one says that $(X,\|\cdot\|_X)$ has moduli of uniform convexity and uniform smoothness of power type $q$ and $p$, respectively, if $\d_X(\e)\gtrsim_{X,q} \e^q$ and $\rho_X(\tau)\lesssim_{X,p} \tau^p$ for all $\e\in [0,2]$ and $\tau\in [0,\infty)$. By the parallelogram identity, a Hilbert space has moduli of uniform convexity and uniform smoothness of power type $2$; conversely, Figiel and Pisier~\cite{FP74} proved (confirming a conjecture of Lindenstrauss~\cite{Lin63}) that if a Banach space has this property, then it is isomorphic to a Hilbert space. In the reflexive range $p\in (1,\infty)$, the works of Clarkson~\cite{Cla36} and Hanner~\cite{Han56} show that any $L_p(\mu)$ space has moduli of uniform convexity and uniform smoothness of power type $\max\{p,2\}$ and $\min\{p,2\}$, respectively.

An important theorem of Pisier~\cite{Pis75} asserts that if $\d_X(\e)>0$ for all $\e\in (0,2]$, then there exists $q\in [2,\infty)$ and an equivalent norm on $X$ with respect to which it has modulus of uniform convexity of power type $q$. Analogously, if $\lim_{\tau\to 0^+} \rho_X(\tau)/\tau=0$, then there exists $p\in (1,2]$ and an equivalent norm on $X$ with respect to which it has modulus of uniform smoothness of power type $p$. For this reason, we will focus below  only on uniform convexity and smoothness with power-type behavior.


\begin{theorem}\label{thm:really main} Fix $p,q\in [1,\infty)$ that satisfy $p\le 2\le q$.  Let $(X,\|\cdot\|_X)$ and $(Y,\|\cdot\|_Y)$ be Banach spaces that have moduli of uniform smoothness and uniform convexity of power type $p$ and $q$, respectively.  Then, there exists $D=D(\rho_X,\d_Y,q)\in [1,\infty)$ satisfying
\begin{equation}\label{eq:rough D upper bound convexity smoothness}
D\lesssim_{\rho_X,\d_Y,p,q}\big(\log(\cc_Y(X)+1)\big)^{\!\frac{1}{q}},
\end{equation}
such that the $\frac{p}{q}$-snowflake of $(X,\|\cdot\|_X)$ embeds  with $q$-average distortion $D$ into $\ell_q(Y)$.

Furthermore, if $(Y,\|\cdot\|_Y)$ has modulus of uniform smoothness of power type $r>p$, then for every $\e\in [0,\frac{r-p}{q}]$ the $(\frac{p}{q}+\e)$-snowflake of $(X,\|\cdot\|_X)$ embeds  with $q$-average distortion $D$ into $\ell_q(Y)$, where
\begin{equation}\label{eq:D cases}
D\lesssim_{\rho_X,\d_Y,\rho_Y,p,q} \left\{\begin{array}{ll}\big(\log(\cc_Y(X)+1)\big)^{\!\frac{1}{q}}&\mathrm{if\ } 0\le \e\le \frac{r-p}{q\log(\cc_Y(X)+1)},\\ \Big(\frac{r-p}{\e}\Big)^{\!\frac{1}{q}}\cc_Y(X)^{\frac{r\e}{r-p}}&\mathrm{if\ }\frac{r-p}{q\log(\cc_Y(X)+1)}\le \e\le \frac{r-p}{q}. \end{array}\right.
\end{equation}
\end{theorem}


Because $\ell_q(\ell_q(Y))$ is isometric to $\ell_q(Y)$ and by~\cite{Fig76} if $Y$ satisfies the assumption of Theorem~\ref{thm:really main}, then so does $\ell_q(Y)$, Theorem~\ref{thm:really main} establishes that the worst-case bi-Lipschitz distortion into $\ell_q(Y)$ is exponentially larger than its average-case counterpart. Specifically, in the setting of Theorem~\ref{thm:really main}, if one finds any Borel probability measure $\mu$ on $X$  such that no embedding of the $(p/q)$-snowflake of $(X,\|\cdot\|_X,\mu)$ into $\ell_q(Y)$ has $q$-average distortion less than $D\ge 1$, then any embedding of $X$ into $\ell_q(Y)$ incurs bi-Lipschitz distortion at least $\exp(\beta D^q)$, where $\beta>0$ depends only on $\rho_X,\d_Y,p,q$.

Theorem~\ref{thm:average john} is a special case of Theorem~\ref{thm:really main}. Indeed, if $\dim(X)=k$ and $Y=\ell_2$, then $\cc_{\ell_2}(X)\le \sqrt{k}$ by John's theorem (the simpler Auerbach lemma~\cite[Lemma~2.22]{Ost13} suffices for this application). The assumptions of Theorem~\ref{thm:really main} hold  with $p=1$ for any Banach space $(X,\|\cdot\|_X)$ and with $q=2$ when $Y=\ell_2$, in which case $\ell_2(Y)$ is still a Hilbert space, so we arrive at the conclusion of Theorem~\ref{thm:average john}.

When $(X,\|\cdot\|_X)$ has modulus of uniform smoothness of power type $1< p\le 2$, Theorem~\ref{thm:really main} obtains the desired embedding with H\"older regularity that improves with $p$, namely a lesser amount of snowflaking. In particular, if $(X,\|\cdot\|_X)$ has modulus of uniform smoothness of power type $2$ and $(Y,\|\cdot\|_Y)$ has modulus of uniform convexity of power type $2$, then the embedding of Theorem~\ref{thm:really main} is of the original metric on $(X,\|\cdot\|_X)$ without any snowflaking (i.e., it is Lipschitz rather than H\"older).

Returning to an examination of the special case when $(Y,\|\cdot\|_Y)$ is a Hilbert space, the first part of Theorem~\ref{thm:really main} yields the same quadratic average  distortion as that of Theorem~\ref{thm:average john}, but now this is achieved for the $\frac{p}{2}$-snowflake of $(X,\|\cdot\|_X)$. For $p>1$ this is better (a less dramatic deformation of the original metric on $X$) than the $\frac12$-snowflake of Theorem~\ref{thm:average john}, albeit under the stronger assumption that the modulus of uniform smoothness of $(X,\|\cdot\|_X)$ is of power type $p$. For every $1\le p<2$ (thus  also covering the setting of Theorem~\ref{thm:average john}), this amount of snowflaking is sharp in the sense that for any fixed exponent that is strictly larger than $\frac{p}{2}$, the dependence on $\cc_Y(X)$ in~\eqref{eq:rough D upper bound convexity smoothness} must sometimes grow as $\cc_Y(X)\to\infty$ at a rate that is at least a definite positive power of $\cc_Y(X)$. This is the content of the following lemma, which also establishes that the exponent of $\cc_Y(X)$ in~\eqref{eq:D cases}, which is equal  to $2\e/(2-p)$ when $Y$ is a Hilbert space (hence $q=r=2$), cannot be improved in general.

\begin{lemma}\label{prop:p/2 sharp} Fix  $p\in [1,2)$, $\alpha\ge 1$,  $\beta\ge 2$ and $\e\in [0,1-\frac{p}{2}]$. For arbitrarily large $\cc\ge 1$ there exists a normed space $(X,\|\cdot\|_X)$ that satisfies the assumptions of Theorem~\ref{thm:really main} with $Y=\ell_2$ (namely, its modulus of uniform smoothness has power type $p$ and $\cc_{\ell_2}(X)=\cc$) such that if the $(\frac{p}{2}+\e)$-snowflake of $(X,\|\cdot\|_X)$ embeds with $\alpha$-average distortion $D\ge 1$ into $\ell_\beta(\ell_2)$, then necessarily
\begin{equation}\label{eq:alpha plus beta}
D\gtrsim \frac{1}{\sqrt{\alpha+\beta}}\cc^{\frac{2\e}{2-p}}.
\end{equation}
\end{lemma}

We do not know if the analogue of Lemma~\ref{prop:p/2 sharp} holds in the full range of parameters of Theorem~\ref{thm:really main}, namely when either $q\neq 2$ or $r\neq 2$ (note that by the aforementioned Figiel--Pisier characterization of Hilbert space~\cite{FP74}, if $q=r=2$, then $Y$ must be isomorphic to a Hilbert space).

\begin{question} Is the exponent of $\cc_Y(X)$ in~\eqref{eq:D cases}, namely $\frac{r\e}{r-p}$, optimal also when $(q,r)\neq (2,2)$?
\end{question}

Despite the optimality for $Y=\ell_2$ of the amount of snowflaking that is required for achieving the logarithmic behavior~\eqref{eq:rough D upper bound convexity smoothness}, as expressed in Lemma~\ref{prop:p/2 sharp}, at the endpoint case of $\frac{p}{2}$-snowflakes the potential optimality of the dependence on $\cc_{\ell_2}(X)$ in~\eqref{eq:rough D upper bound convexity smoothness} is much more mysterious.
The  case $p=1$ is an exception, because we have already seen (in the beginning of Section~\ref{sec:dim reduction intro}) that the distortion bound of Theorem~\ref{thm:average john} is sharp; more generally, Remark~\ref{eq"log n 1/q} below shows that~\eqref{eq:rough D upper bound convexity smoothness} is sharp for $p=1$ and any $q\ge 2$. However, this is  proved by considering $X=\ell_\infty^k$, which is not pertinent to the range $p\in (1,2]$. The (in our opinion unlikely) possibility remains that if $p\in (1,2]$, then for every Banach space $X$ whose modulus of uniform smoothness has power type $p$ there exists $\omega=\omega(X)\in (0,1]$ such that the $\omega$-snowflake of $X$ embeds with average distortion\footnote{For concreteness we chose to discuss average distortion, but note that due to Proposition~\eqref{prop:other exponents}, for any $q\ge 1$ such a qualitative statement is equivalent to the same statement with ``average distortion'' replaced by ``$q$-average distortion.''}  $D=D(X)\in [1,\infty)$ into a Hilbert space; we do not know  an obstruction to this holding even for the maximal possible exponent $\omega=\frac{p}{2}$.

A Banach space $X$ is called {\em superreflexive} if it admits an equivalent uniformly smooth norm, namely a norm for which $\lim_{\tau\to 0^+} \rho_X(\tau)/\tau=0$,  and this holds if an only if~\cite{Enf72} it admits an equivalent uniformly convex norm, namely a norm for which $\d_X(\e)>0$ for all $\e\in (0,2]$. These are not the original definitions of superreflexivity (due to James~\cite{Jam72}), but they are equivalent to them by deep work of Enflo~\cite{Enf72}. By the aforementioned renorming theorem of Pisier~\cite{Pis75}, superreflexivity is equivalent to admitting an equivalent norm whose modulus of uniform smoothness has power type $p$ for some $p>1$. Therefore, the above discussion coincides with the following question.

\begin{question}\label{Q:superreflexive}
Does every superreflexive Banach space $X$ admit $\omega(X)\in (0,1]$ and $D(X)\in [1,\infty)$ such that the $\omega(X)$-snowflake of $X$ embeds with average distortion $D(X)$ into a Hilbert space?
\end{question}
We conjecture that the answer to Question~\eqref{Q:superreflexive} is negative; constructing an example that demonstrates this conjecture would be an important achievement. Perhaps even~\eqref{eq:rough D upper bound convexity smoothness}  is sharp, but at present  we do not have sufficient evidence  in support of this  more ambitious conjecture  (other than  that this is so when $p=1$). Notwithstanding the above expectation, if it were the case that Question~\eqref{Q:superreflexive} had a positive answer, then this would be a truly remarkable theorem, asserting that the mere presence of uniform convexity implies ``bounded distance on average'' from Hilbertian geometry. In particular, a positive answer to Question~\eqref{Q:superreflexive} would resolve a central open question (see e.g.~\cite{Laf08,Pis10,MN14}) by demonstrating that every classical expander is a super-expander; we will explain this deduction in Remark~\ref{rem:super-expander} below, after the relevant concepts are recalled.

\subsubsection{Embedding a complex interpolation family into its endpoint} Theorem~\ref{thm:really main}, and therefore also its special case Theorem~\ref{thm:average john}  (average John) and its corollaries Theorem~\ref{thm:l1X} (impossibility of dimension reduction for expanders) and   Theorem~\ref{thm:lp version Kp} (matrix-dimension inequality), are all consequences of the structural statement for complex interpolation spaces that appears in Theorem~\ref{thm:interpolation implicit first exposure} below. We recall the (standard) background in Section~\ref{sec:interpolation proof} below; here we  explain the idea in broad stokes.

Following Calder\'on~\cite{Cal64} and Lions~\cite{Lio60}, to a pair $(X,\|\cdot\|_X), (Z,\|\cdot\|_Z)$ of complex Banach spaces that satisfies a mild compatibility assumption (which will be immediate in our setting), one associates a one-parameter family $(\theta\in [0,1])\mapsto [X,Z]_\theta$ of Banach spaces which interpolates between them, namely $[X,Z]_0=X$ and $[X,Z]_1=Z$. This provides a useful  way to deform the geometry of $(X,\|\cdot\|_X)$ to that of $(Z,\|\cdot\|_Z)$, and Theorem~\ref{thm:interpolation implicit first exposure} is a technical statement that quantifies the extent to which elements of this complex interpolation family differ from its $\theta=1$ endpoint.

\begin{theorem}\label{thm:interpolation implicit first exposure} Fix $\theta\in (0,1]$ and $p,q\in [1,\infty)$ with $1\le p\le 2\le q$.
Let $(X,\|\cdot\|_X)$ and $(Z,\|\cdot\|_Z)$ be a compatible pair of complex Banach spaces such that the moduli of uniform smoothness and convexity of $[X,Z]_\theta$ and $Z$ are of power type $p$ and $q$, respectively. Then, there exists $D\ge 1$ satisfying
\begin{equation}\label{eq: D interpolation implicit}
D\lesssim_{\rho_{[X,Z]_\theta},\d_Z,p,q} \Big(\frac{1}{\theta}\Big)^{\!\frac{1}{q}},
\end{equation}
such that the $\frac{p}{q}$-snowflake of $[X,Z]_\theta$  embeds with $q$-average distortion $D$ into $\ell_q(Z)$.
\end{theorem}

The implicit dependence in~\eqref{eq: D interpolation implicit} on the data $\rho_{[Y,Z]_\theta},\d_Z,p,q$ is specified in Section~\ref{sec:interpolation proof} below, where Theorem~\ref{thm:interpolation implicit first exposure} is proved; see specifically inequality~\eqref{eq:specified D upper for interpolation}.   Section~\ref{sec:proof of 9 from interpolation} below demonstrates that Theorem~\ref{thm:interpolation implicit first exposure} implies Theorem~\ref{thm:really main}. The basic idea is as follows. Suppose that  $(X,\|\cdot\|_X)$ and $(Y,\|\cdot\|_Y)$ are Banach spaces that satisfy the assumptions of Theorem~\ref{thm:really main}. Assume for simplicity that they are a compatible pair of complex Banach spaces (this side-issue is treated in Section~\ref{sec:interpolation proof} via a standard complexification step). By estimating the relevant modulus of $[X,Y]_\theta$, contrasting the bound~\eqref{eq: D interpolation implicit} with a bound on the distance of $[X,Y]_\theta$ to $X$, and then optimizing over $\theta$, we arrive at Theorem~\ref{thm:really main}. Thus, the role of complex interpolation in the proof of Theorem~\ref{thm:really main} is in essence as a Banach space-valued   flow  starting at $X$ and terminating at $Y$, parameterized by $\theta\in [0,1]$. At positive times $\theta>0$, this flow consists of spaces that embed on average into $\ell_q(Y)$, by Theorem~\ref{thm:interpolation implicit first exposure}. The desired embedding of $X$ itself is obtained since  this flow tends to  $X$ as $\theta\to 0^+$ (at a definite rate).

\begin{remark} Theorem~\ref{thm:interpolation implicit first exposure} (combined with Proposition~\ref{prop:other exponents}) shows that Question~\ref{Q:superreflexive} has a positive answer for spaces of the form $[X,H]_\theta$ where $X,H$ is a compatible pair of complex Banach spaces with $H$ being a Hilbert space, and $\theta\in (0,1]$ (thus, by Pisier's extrapolation theorem~\cite{Pis79-lattices-interpolation}, Question~\ref{Q:superreflexive} has a positive answer for superreflexive Banach lattices). An inspection of the ensuing proof of Theorem~\ref{thm:interpolation implicit first exposure} reveals that this holds also for  subspaces of quotients of the class of $\theta$-Hilbertian Banach spaces that was introduced by Pisier in~\cite{Pis10} (we will not recall the definition here). So, a proof of our conjectured negative answer to Question~\ref{Q:superreflexive} would entail constructing a superreflexive Banach space which is not a subspace of a quotient of any $\theta$-Hilbertian Banach space. Such spaces are  not yet known to exist; see~\cite[pages~15--16]{Pis10} for a discussion of this intriguing open  question in structural Banach space theory. As stated above and  justified in Remark~\ref{rem:super-expander}, also disproving our conjecture by answering  Question~\ref{Q:superreflexive} positively would have interesting ramifications.

\end{remark}

\subsection{Historical discussion}\label{sec:history} Adopting terminology of~\cite[Definition~2.1]{LLR95}, given  $n\in \N$, an $n$-point metric space $\MM$ and $D \ge 1$, define an integer $\dim_D(\cM)$, called the (bi-Lipschitz distortion-$D$) {\em metric dimension} of $\cM$, to be the minimum $k\in \N$ for which there exists a $k$-dimensional normed space $X_\cM$ such that $\cM$ embeds into $X_\cM$ with bi-Lipschitz distortion $D$.  By the Fr\'echet isometric embedding into $\ell_\infty^{n-1}$, we always have $\dim_D(\cM)\le \dim_1(\cM)\le n-1$.

Johnson and Lindenstrauss asked~\cite[Problem~3]{JL84} if   $\dim_D(\cM)=O(\log n)$ for some $D=O(1)$ and every $n$-point metric space $\cM$. The $O(\log n)$ bound arises naturally here, because it cannot be improved due to a standard volumetric argument when one considers embeddings of the  equilateral space of size $n$. See Remark~\ref{rem:ribe} below and mainly the survey~\cite{Nao18} for background on  this question (and, more generally, the field of {\em metric dimension reduction}, to which the present investigations belong), including how it initially arose in the context of the {Ribe program}.

Bourgain proved~\cite[Corollary~4]{Bou85} that the Johnson--Lindenstrauss question has a negative answer. Specifically, he showed that for arbitrarily large $n\in \N$ there is  an $n$-point metric space $\cM_n$ such that $\dim_D(\cM_n)\gtrsim (\log n)^2/(D\log\log n)^2$ for every $D\ge 1$. This naturally led him to raise the question~\cite[page~48]{Bou85} of determining the asymptotic behavior of the maximum of $\dim_D(\cM)$ over all $n$-point metric spaces $\cM$. It took over a decade for this question to be  resolved.

Towards this goal, Johnson, Lindenstrauss and Schechtman~\cite{JLS87} proved that there exists a universal constant $\alpha>0$ such that for every $D\ge 1$ and $n\in \N$ we have $\dim_D(\cM)\lesssim_D n^{\alpha/D}$ for any $n$-point metric space $\cM$.  In~\cite{Mat92,Mat96}, Matou\v{s}ek improved this result by showing that one can actually embed any such $\cM$ with distortion $D$ into $\ell_\infty^k$ for some $k\in \N$ satisfying $k\lesssim_D n^{\alpha/D}$, i.e., the target normed space need not depend on $\cM$ (Matou\v{s}ek's proof is also simpler than that of~\cite{JLS87}, and it yields a smaller value of the constant $\alpha$; see the exposition in Chapter~15 of the monograph~\cite{Mat02}).

For small distortions, Arias-de-Reyna and Rodr{\'{\i}}guez-Piazza proved~\cite{AR92}  the satisfactory assertion that for arbitrarily large $n\in \N$ there exists an $n$-point metric space $\cM_n$ such that $\dim_D(\cM_n) \gtrsim_D n$ for every $1\le D<2$. For larger distortions, they asked~\cite[page~109]{AR92} if for every $D\in (2,\infty)$ and $n\in \N$ we have $\dim_D(\cM)\lesssim_D (\log n)^{O(1)}$ for any $n$-point metric space $\cM$. For this distortion regime, an asymptotic improvement (as $n\to \infty$) over the aforementioned lower bound of Bourgain~\cite{Bou85} was made by Linial, London and Rabinovich~\cite[Proposition~4.2]{LLR95}, who showed that  for arbitrarily large $n\in \N$ there exists an $n$-point metric space $\cM_n$ such that $\dim_D(\cM_n)\gtrsim (\log n)^2/D^2$ for every $D\ge 1$.

  In~\cite{Mat96}, Matou\v{s}ek answered  the above questions by proving Theorem~\ref{thm:matousek} below via an ingenious argument that relies on (a modification of) graphs of large girth with many edges and an existential counting argument (inspired by ideas of Alon, Frankl and R\"odl~\cite{AFR85})  that uses the classical theorem of  Milnor~\cite{Mil64} and Thom~\cite{Tho65} from real algebraic geometry.
\begin{theorem}[Matou\v{s}ek~\cite{Mat96}]\label{thm:matousek} For every $D\ge 1$ and arbitrarily large $n\in \N$, there exists  an $n$-point metric space $\cM_{n}(D)$ such that $\dim_D\!\big(\cM_{n}(D)\big)\gtrsim_D n^{c/D}$, where $c>0$ is a universal constant.
 \end{theorem}
 Due to the Johnson--Lindenstrauss--Schechtman upper bound~\cite{JLS87}, Theorem~\ref{thm:matousek}  is a complete (and unexpected) answer to the aforementioned questions of Johnson--Lindenstrauss~\cite{JL84}, Bourgain~\cite{Bou85} and Arias-de-Reyna--Rodr{\'{\i}}guez-Piazza~\cite{AR92},  up to the value of the universal constant $c$. Theorem~\ref{thm:l1X}  furnishes a new resolution of these questions, via an analytic approach for deducing dimension lower bounds from rough  metric information that differs markedly from Matou\v{s}ek's algebraic argument.

 Our solution has some novel features. It shows that the spaces $\cM_n(D)$ of Theorem~\ref{thm:matousek} can actually be taken to be independent of the distortion $D$, while the construction of~\cite{Mat96} depends on $D$ (it is based on graphs whose girth is of order $D$). One could alternatively achieve this by considering the disjoint union of the spaces $\{\cM_n(2^k)\}_{k=0}^{m}$ for  $m\asymp \log n$, which is a metric space of size $O(n\log n)$.

Rather than using an ad-hoc construction (and a non-constructive existential statement) as in~\cite{Mat96}, here we specify a natural class of metric spaces, namely the shortest-path metrics on  expanders (see also Remark~\ref{rem:Qcube and others} below), for which Theorem~\ref{thm:l1X}  holds. The question of determining the metric dimension of expanders was first considered by Linial--London--Rabinovich~\cite{LLR95}. Indeed, their aforementioned lower bound $\dim_D(\cM_n)\gtrsim (\log n)^2/D^2$ was obtained when $\MM_n$ is the shortest-path metric on an $n$-vertex expander $\G=(\n,E_\G)$. This lower bound remained the best-known prior to our proof of Theorem~\ref{thm:l1X} that establishes the exponential improvement  $\dim_D(\n,d_\G)\ge n^{\eta/D}$ for some $\eta=\Omega(1)$, which is best-possible up to the value of $\eta$.

We were motivated to revisit this old question because it arose more recently in the work~\cite{ANRW16} of Andoni,  Nguyen, Nikolov, Razenshteyn and Waingarten  on approximate nearest neighbor search (NNS). They devised  an approach for proving an impossibility result for NNS that requires the existence of an $n$-vertex expander that embeds with bi-Lipschitz distortion $O(1)$ into some normed space of dimension $n^{o(1)}$. By Theorem~\ref{thm:l1X}  no such expander exists, thus resolving (negatively) a question that Andoni--Nguyen--Nikolov--Razenshteyn--Waingarten posed in~\cite[Section~1.6]{ANRW16}.

Unlike Theorem~\ref{thm:matousek}, the lower bound $\dim(X)\ge n^{\Omega(1)}$ of Theorem~\ref{thm:l1X} assumes (when the underlying graph $\G$ is an $n$-vertex expander) that the embedding has $O(1)$ average distortion rather than the worst-case control that $O(1)$ bi-Lipschitz distortion entails. In fact, we only need to assume that there is $f:\n\to X$  that preserves up to constant factors two specific distance sums, i.e., that
$\sum_{\{i,j\}\in E_\G}\|f(i)-f(j)\|_{X}\asymp \sum_{\{i,j\}\in E_\G} d_\G(i,j)$ and  $\sum_{i=1}^n\sum_{j=1}^n\|f(i)-f(j)\|_{X}\asymp \sum_{i=1}^n\sum_{j=1}^n d_\G(i,j)$.
Moreover, we also deduce the lower bound $\dim(X)\ge n^{\Omega(1)}$ from the existence of an embedding into $\ell_1(X)$ with these properties. We do not see how the algebraic technique of~\cite{Mat96} could address such issues, namely distance preservation being only on average and infinite dimensional targets.

Thus, the new approach that we devise here is both  more robust than that of~\cite{Mat96}, in the sense that it relies on significantly less stringent assumptions, and it also provides an explicit criterion (spectral gap) for intrinsic (largest-possible) high-dimensionality. Both of these features, as well as ideas within our proof,  turned out to be important for subsequent developments that occurred  since a preliminary version of the present work was posted (November 2016). Specifically, in a series of collaborations with Andoni, Nikolov, Razenshteyn and Waingarten~\cite{ANNRW18,ANNRW-FOCS18,ANNRW18-general,ANNRW18-interpolation}, we studied the algorithmic question (NNS) that Theorem~\ref{thm:l1X} resolves negatively (recall that it is a negative solution to a question that would have implied an algorithmic {\em impossibility result}). These works design NNS data structures for arbitrary high-dimensional norms that were previously believed to be unattainable. For this purpose, the  robustness of the average-case requirements in combination with our  use below of a  recently developed theory of {\em nonlinear spectral gaps} are both crucial for uncovering new structural information about general norms (a randomized hierarchical partitioning scheme that is governed by the intrinsic geometry). We refer to~\cite{ANNRW18,ANNRW-FOCS18,ANNRW18-general,ANNRW18-interpolation} and the surveys~\cite{AIR18,Nao18} for more information on these more recent  algorithmic developments which rely on the present work.

\begin{remark}\label{rem:ribe} The {\em Ribe program} aims to uncover an explicit ``dictionary'' between the local  theory of Banach spaces and general metric spaces, inspired by a rigidity theorem of Ribe~\cite{Rib76} that indicates that a dictionary of this sort should exist. See~\cite{Bou86} as well as the surveys~\cite{Kal08-survey,Nao12,Bal13,Nao18} and the monograph~\cite{Ost13} for more on this area. While much of the more recent research on dimension reduction is driven by the need to compress data, the initial motivation of the above question of~\cite{JL84} arose in the Ribe program. It is simplest to include here a direct quotation of Matou\v{s}ek's explanation in~\cite[page~334]{Mat96} for the origin of the investigations that led to his Theorem~\ref{thm:matousek}.

\blockquote{\em ...This investigation started in the context of the
local Banach space theory, where the general idea was to obtain some analogs for
general metric spaces of notions and results dealing with the structure of finite
dimensional subspaces of Banach spaces. The distortion of a mapping should
play the role of the norm of a linear operator, and the quantity $\log n$, where $n$ is
the number of points in a metric space, would serve as an analog of the dimension
of a normed space. Parts of this programme have been carried out by Bourgain,
Johnson, Lindenstrauss, Milman and others...}

Despite many previous successes of the Ribe program, not all of the questions that it raised turned out to have a positive answer (e.g.~\cite{MN13-convexity}).  Theorem~\ref{thm:matousek} is among the most extreme examples of failures of natural steps in the Ribe program, with the final answer being exponentially worse than the  predictions.  Here we provide a different derivation (and strengthening) of this phenomenon.

It is an  amusing coincidence  that while Johnson and Lindenstrauss raised~\cite[Problem~3]{JL84} as a step toward  a metric version of John's theorem (see~\cite[Problem~4]{JL84}; this was resolved by Bourgain~\cite{Bou85}, who took a completely different route than the one proposed in~\cite{JL84}), the present work finds another nonlinear version of John's theorem  and demonstrates that in fact it serves as an obstruction to the dimension reduction phenomenon that Johnson and Lindenstrauss were hoping for.
\end{remark}

\subsection{Roadmap} Section~\ref{sec:gap and duality} recalls the theory of nonlinear spectral gaps that was alluded to above. Further background on uniform convexity and smoothness, as well as background on Ball's notion of Markov type (both of which are tools for subsequent proofs) appears, respectively, in Section~\ref{sec:smoothness} and Section~\ref{sec:Mtype}. The link between nonlinear spectral gaps and  Theorem~\ref{thm:average john} is through a duality statement that we proved in~\cite{Nao14}; Section~\ref{sec:duality first}   describes a convenient enhancement of this duality which is proved  in (the mainly technical)  Section~\ref{sec:duality}. Section~\ref{sec:interpolation proof} treats complex interpolation, leading to Theorem~\ref{thm:interpolation implicit first exposure}. A  key inequality (Theorem~\ref{thm:interpolation markov}) about nonlinear spectral gaps in complex interpolation spaces appears in Section~\ref{sec: gap interpolation}. Its proof adapts an approach of~\cite{Nao14} where a similar inequality was derived; such an adaptation is required because~\cite{Nao14} relies on somewhat arbitrary choices of distance exponents, due to which we do not see how to use the  results of~\cite{Nao14} to prove Theorem~\ref{thm:average john}. The deduction of Theorem~\ref{thm:really main} (hence also its special case Theorem~\ref{thm:average john}) from Theorem~\ref{thm:interpolation markov} appears in Section~\ref{sec:proof of 9 from interpolation}. The proof of Theorem~\ref{thm:interpolation markov}, namely our main nonlinear spectral gap inequality, appears in Section~\ref{sec:proof of gap interpolation thm}, though it assumes Proposition~\ref{prop:other exponents} whose proof  is postponed to Section~\ref{sec:aux em} which is devoted to several auxiliary embedding results of independent interest. The case $\omega=1$ of Proposition~\eqref{prop:other exponents} (passing from $p$-average distortion to $q$-average distortion) was first broached in~\cite{Nao14} where a similar statement is obtained under an additional assumption that is not needed in our context, and with much (exponentially) worse dependence on $p,q$ than what we derive here; due to the basic nature of these facts and also because obtaining them is not merely a technical adaptation of~\cite{Nao14}, full proofs are included in Section~\ref{sec:revisited}. The more novel case $\omega\in (0,1)$ of  Proposition~\eqref{prop:other exponents} is based on elementary geometric reasoning; again, due to the fundamental nature of this fact (as well as its connection to longstanding open questions), we prove it in Section~\ref{sec:normalization} while taking care to  obtain good asymptotic dependence as $\omega\to 0^+$. The proof of Theorem~\ref{thm:lp version Kp} appears in Section~\ref{sec:degree}. Section~\ref{sec:impossibility} is devoted to several impossibility results, including those that were discussed above, such as Lemmas~\ref{lem:snowflake john} and~\ref{prop:p/2 sharp}.

\subsection*{Acknowledgements} I am grateful to Noga Alon, Alexandr Andoni, Emmanuel Breuillard, Ilya Razenshteyn, John Pardon, Igor Rodnianski, Gideon Schechtman, Tasos Sidiropoulos, Vijay Sridhar and Ramon van Handel for helpful discussions and feedback. I also thank the anonymous referees for their careful reading of this manuscript and their useful corrections and feedback.

Throughout the work on this project (and many others), I benefited immensely from thought-provoking, inspiring and illuminating conversations with my friend and colleague Eli Stein, who passed away as the the final revision of this work was completed. He is dearly missed.

\section{Nonlinear spectral gaps and duality}\label{sec:gap and duality}

Suppose that  $(\MM,d_\MM)$ is a metric space, $p>0$ and $n\in \N$. If $\pi\in \bigtriangleup^{\!n-1}$ and  $\A=(a_{ij})\in \M_n(\R)$ is  a stochastic and $\pi$-reversible matrix, then in analogy to~\eqref{eq:energy} one measures the magnitude of the (reciprocal of) the nonlinear spectral gap of $\A$ relative to the kernel $d_\MM^p:\MM\times \MM\to [0,\infty)$ through a quantity $\gamma(\A,d_\MM^p)\in [0,\infty]$ which is defined~\cite{MN14} as the infimum over those $\gamma\in [0,\infty]$ such that
\begin{equation}\label{eq:nonliear gap def}
\forall\, x_1,\ldots,x_n\in \MM,\qquad \sum_{i=1}^n\sum_{j=1}^n \pi_i\pi_j d_\MM(x_i,x_j)^p\le \gamma\sum_{i=1}^n\sum_{j=1}^n \pi_ia_{ij}d_\MM(x_i,x_j)^p.
\end{equation}

Even though~\eqref{eq:nonliear gap def} is analogous to~\eqref{eq:energy}, a nonlinear spectral gap can differ markedly from the usual (reciprocal of the) gap in the (linear) spectrum; see~\cite{MN14,MN15} for some of the subtleties and mysteries that arise from this generalization. As explained in~\cite{MN14}, unless $\MM$ is a singleton, if $\gamma(\A,d_\MM^p)$ is finite, then $\lambda_2(\A)$ is bounded away from $1$ by a positive quantity that depends on $\gamma(\A,d_\MM^p)$. So, the property of a matrix that is being considered here (determined by its interaction with the geometry of a metric space) is more stringent than requiring that it has a  spectral gap in the classical sense.

A quite substantial theory of nonlinear spectral gaps was developed in  a series of works,  including~\cite{Mat97,Gro03,BLMN05,IN05,Laf08,Laf09,Pis10,NS11,Kon12,MN13-bary,MN14,Nao14,MN15,Mim15,Che16,LS17,ANN18}, for several geometric applications, though many fundamental questions remain open. Establishing the utility of nonlinear spectral gaps to the results presented in the Introduction is a key conceptual contribution of the present work, and this underlies the algorithmic applications that were  developed  in~\cite{ANNRW18,ANNRW-FOCS18,ANNRW18-general,ANNRW18-interpolation}.

\begin{remark}\label{rem:super-expander} Fix $\Delta\in \N$. A sequence of $\Delta$-regular  graphs $\{\G_n=(V_n,E_n)\}_{n=1}^\infty$ is an expander with respect to a metric space $(\MM,d_\MM)$ if $\lim_{n\to \infty} |V_n|=\infty$ and $\sup_{n\in \N}\gamma(\A_{\G_n},d_\MM^2)<\infty$, where we recall that $\A_{\G_n}$ is the normalized adjacency matrix of $\G_n$. $\{\G_n\}_{n=1}^\infty$  is called a {\em super-expander} if it is an expander with respect to {\em every} superreflexive Banach space. It is a major open problem if a sequence $\{\G_n\}_{n=1}^\infty$ of  bounded-degree regular graphs  is a super-expander whenever $\sup_{n\in \N} 1/(1-\lambda_2(\G_n))<\infty$, i.e., when $\{\G_n\}_{n=1}^\infty$ is an expander in the classical sense. If Question~\ref{Q:superreflexive} had a positive answer, then any classical expander would  be a super-expander. Indeed, let $(X,\|\cdot\|_X)$ be a superreflexive Banach space. It suffices to prove that for every regular  graph $\G=(\n,E_\G)$ we have
\begin{equation}\label{eq:desired gamma omega}
\gamma\big(\A_\G,\|\cdot\|_{\!X}^{2\phantom{p}}\!\big)\lesssim_X \bigg( \frac{1}{1-\lambda_2(\G)}\bigg)^{\!\!\frac{1}{\omega(X)}}.
\end{equation}
To establish~\eqref{eq:desired gamma omega}, by the hypothesized positive answer to Question~\ref{Q:superreflexive}, there are $\omega(X)\in (0,1]$ and $D(X)\in [1,\infty)$ such that the $\omega(X)$-snowflake of $X$ embeds with average distortion $D(X)$ into $\ell_2$. Proposition~\ref{prop:other exponents} with parameters  $\omega=\omega(X)$, $D=D(X)$, $p=1$ and $q=2/\omega(X)$ shows that there exists $D'(X)\in [1,\infty)$ such that the $\omega(X)$-snowflake of $X$ embeds with $(2/\omega(X))$-average distortion $D'(X)$ into $\ell_2$. If $x_1,\ldots,x_n\in X$, then an application of this conclusion to the uniform measure on $\{x_1,\ldots,x_n\}$ provides and embedding $f:X\to \ell_2$ which is $\omega(X)$-H\"older with constant $D'(X)$ and
\begin{equation}\label{eq:2/omegaX}
\frac{1}{n^2}\sum_{i=1}^n\sum_{j=1}^n \|f(x_i)-f(x_j)\|_{\ell_2}^{\!\frac{2}{\omega(X)}}\ge \frac{1}{n^2}\sum_{i=1}^n\sum_{j=1}^n \|x_i-x_j\|_{\!X}^{2\phantom{p}}.
\end{equation}
By~\cite[Lemma~5.5]{BLMN05} (see also~\eqref{eq:L 2} below), there exists a universal constant $C\in (0,\infty)$ such that
\begin{equation}\label{eq:use BLMN}
\frac{1}{n^2}\sum_{i=1}^n\sum_{j=1}^n \|f(x_i)-f(x_j)\|_{\ell_2}^{\!\frac{2}{\omega(X)}}\le \bigg( \frac{C}{\omega(X)\sqrt{1-\lambda_2(\G)}}\bigg)^{\!\!\frac{2}{\omega(X)}}\frac{1}{|E_\G|}\sum_{\{i,j\}\in E_\G} \|f(x_i)-f(x_j)\|_{\ell_2}^{\!\frac{2}{\omega(X)}}.
\end{equation}
The fact that $f$ is $\omega(X)$-H\"older with constant $D'(X)$ gives
\begin{equation}\label{eq:D'}
\frac{1}{|E_\G|}\sum_{\{i,j\}\in E_\G} \|f(x_i)-f(x_j)\|_{\ell_2}^{\!\frac{2}{\omega(X)}} \le  \frac{D'(X)^{\frac{2}{\omega(X)}}}{|E_\G|}\sum_{\{i,j\}\in E_\G} \|x_i-x_j\|_{\!X}^{2\phantom{p}}.
\end{equation}
By substituting~\eqref{eq:D'} into~\eqref{eq:use BLMN}, and then substituting the resulting inequality into~\eqref{eq:2/omegaX}, we see that
$$
\frac{1}{n^2}\sum_{i=1}^n\sum_{j=1}^n \|x_i-x_j\|_{\!X}^{2\phantom{p}}\le \bigg( \frac{CD'(X)}{\omega(X)\sqrt{1-\lambda_2(\G)}}\bigg)^{\!\!\frac{2}{\omega(X)}}\frac{1}{|E_\G|}\sum_{\{i,j\}\in E_\G} \|x_i-x_j\|_{\!X}^{2\phantom{p}}.
$$
Recalling~\eqref{eq:nonliear gap def}, since this holds for every $x_1,\ldots,x_n\in X$, the justification of~\eqref{eq:desired gamma omega} is complete.
\end{remark}

Prior to explaining how nonlinear spectral gaps relate to the embedding results that we stated in the Introduction, were recall some terminology and notation that will be used in what follows.

\subsection{Uniform convexity and smoothness}\label{sec:smoothness} Let $(X,\|\cdot\|_{\!X}^{\phantom{p}})$ be a normed space and fix $p,q\ge 1$ satisfying $ p\le2\le q$. In the Introduction we recalled the traditional definitions of when it is said that $X$ has moduli of smoothness and convexity of power type $p$ and $q$, respectively. However, it is often convenient to work with an equivalent formulation of these properties due to Ball, Carlen and Lieb~\cite{BCL94} (inspired by contributions of Pisier~\cite{Pis75} and Figiel~\cite{Fig76}), which we shall now recall.

The $p$-smoothness constant of $X$, denoted $\mathscr{S}_p(X)$, is the infimum over those $\mathscr{S}\in [1,\infty]$ such that
\begin{equation}\label{eq:def smoothness}
\forall\, x,y\in X,\qquad \frac{\|x+y\|_{\!X}^p+\|x-y\|_{\!X}^p}{2}-\|x\|_{\!X}^p\le\mathscr{S}^p\|y\|_{\!X}^p.
\end{equation}
By the triangle inequality we always have $\mathscr{S}_1(X)=1$. The $q$-convexity constant of $X$, denoted $\mathscr{K}_q(X)$, is the infimum over those $\mathscr{K}\in [1,\infty]$ such that
$$
\forall\, x,y\in X,\qquad \|y\|_{\!X}^q\le \mathscr{K}^q \bigg(\frac{\|x+y\|_{\!X}^q+\|x-y\|_{\!X}^q}{2}-\|x\|_{\!X}^q\bigg).
$$

As shown in~\cite{BCL94},  $X$ has moduli of smoothness and convexity of power type $p$ and $q$, respectively, if and only if $\mathscr{S}_p(X)<\infty$ and $\mathscr{K}_q(X)<\infty$, respectively. It is beneficial to work with the coefficients $\mathscr{S}_p(X), \mathscr{K}_q(X)$ rather than the aforementioned classical moduli $\d_X,\rho_X$ because they are well-behaved with respect to basic operations, an example of which is the duality $\mathscr{K}_{p/(p-1)}(X^*)=\mathscr{S}_p(X)$, as shown in~\cite{BCL94}. Another example that is directly relevant to the present work is their especially clean behavior under complex interpolation; see Section~\ref{sec:track}  below. Further useful properties of these parameterizations of uniform convexity and uniform smoothness can be found in~\cite[Section~6.2]{MN14}.

\subsection{Markov type}\label{sec:Mtype} Following Ball~\cite{Bal92}, a metric space   $(\MM,d_\MM)$ is said to have Markov type  $p\ge 1$ if there exists $M\ge 1$ with the following property. Suppose that $n\in \N$ and $\pi\in \bigtriangleup^{\!n-1}$.  Then, for every stochastic and $\pi$-reversible matrix $\A=(a_{ij})\in \M_n(\R)$, every $x_1,\ldots,x_n\in X$ and every $s\in \N$,
\begin{equation}\label{eq:Mtype def}
\bigg(\sum_{i=1}^n\sum_{j=1}^n \pi_i (\A^{\!s})_{ij} d_\MM(x_i,x_j)^p\bigg)^{\!\!\frac{1}{p}}\le Ms^{\frac{1}{p}} \bigg(\sum_{i=1}^n\sum_{j=1}^n \pi_i a_{ij} d_\MM(x_i,x_j)^p\bigg)^{\!\!\frac{1}{p}}.
\end{equation}
The infimum over those $M\ge 1$ which satisfy~\eqref{eq:Mtype def} is called the Markov type $p$ constant of $\MM$, and is denoted $\mathbf{M}_p(\MM)$. This  nomenclature arises from a natural probabilistic interpretation~\cite{Bal92} of~\eqref{eq:Mtype def} in terms of how stationary reversible Markov chains interact with the geometry of $\MM$. We omit this description since it will not be needed below, though it is very important for other applications.

The following theorem, due\footnote{Formally,  \cite[Theorem~2.3]{NPSS06} asserts that $\mathbf{M}_p(X)\lesssim_p \mathscr{S}_p(X)$ with the implicit constant  tending to $\infty$ as $p\to 1^+$. However, \cite[Theorem~4.3]{Nao14} adjusts the martingale argument of~\cite{NPSS06} so as to make that implicit constant universal.} to~\cite[Theorem~2.3]{NPSS06}, will be used in the proof of Theorem~\ref{thm:really main}.
\begin{theorem}\label{thm:quote NPSS} For every $p\in [1,2]$, every Banach space $(X,\|\cdot\|_{\!X}^{\phantom{p}})$ with $\mathscr{S}_p(X)<\infty$ satisfies
\begin{equation}\label{eq:markov type in theorem}
\mathbf{M}_p(X)\lesssim \mathscr{S}_p(X).
\end{equation}
\end{theorem}

\subsection{Duality, compactness and H\"older extension}\label{sec:duality first} The connection between nonlinear spectral gaps and Theorem~\ref{thm:interpolation implicit first exposure} is through Theorem~\ref{thm:full duality} below. For the first part of its statement, we refer to~\cite{Hei80} for background on ultrapowers of Banach spaces. It suffices to say here that to each Banach space $(Z,\|\cdot\|_{Z})$ one associates a (huge) Banach space $Z^\mathscr{U}$, called an ultrapower of $Z$, that has valuable compactness properties. $Z$ is canonically isometric to a subspace of $Z^\mathscr{U}$,  and  any finite-dimensional linear subspace of $Z^\mathscr{U}$ embeds into $Z$ with bi-Lipschitz distortion $1+\e$ for any $\e>0$. Thus,  $Z$ is essentially indistinguishable from any of its ultrapowers in terms of their finitary substructures.  Due to Corollary~\ref{coro:banach space version with smoothness} below, if one does not mind losing a constant factor that depends only on the moduli of uniform convexity and uniform smoothness, then one could drop all mention of ultrapowers in the ensuing discussion, and work  throughout with the classical sequence space $\ell_q(Z)$ instead.

\begin{theorem}\label{thm:full duality} Suppose that $p,q,\mathscr{C}\ge 1$ and $p\le q$. Let $(\MM,d_\MM)$ be a metric space and $(Y,\|\cdot\|_Y)$ be a Banach space such  that for every $n\in \N$, every symmetric stochastic matrix $\A\in \M_n(\R)$ satisfies
\begin{equation}\label{eq:gamma criterion}
\gamma(\A,d_\MM^p)\le \mathscr{C}\gamma(\A,\|\cdot\|_{\!Y}^q).
\end{equation}
  Then, the $\frac{p}{q}$-snowflake of $\MM$ embeds into some ultrapower of $\ell_q(Y)$ with $q$-average distortion $2\mathscr{C}^{\frac{1}{q}}$.

Furthermore, if in addition to the above assumption $\MM$ has Markov type $p$ and the modulus of uniform convexity of $Y$ has power type $q$, then there exists $D\ge 1$ satisfying
\begin{equation*}
D\lesssim \mathbf{M}_p(\MM)^{\!\frac{p}{q}}\mathscr{K}_q(Y)\mathscr{C}^{\frac{1}{q}},
\end{equation*}
such that the $\frac{p}{q}$-snowflake of $\MM$ embeds into $\ell_q(Y)$ with $q$-average distortion $D$.
\end{theorem}

The following corollary is a combination of the second assertion of Theorem~\ref{thm:full duality} and Theorem~\ref{thm:quote NPSS}.

\begin{corollary}\label{coro:banach space version with smoothness} Suppose that $p,q,\mathscr{C}\ge 1$ and $p\le 2\le q$. Let $(X,\|\cdot\|_X)$ and $(Y,\|\cdot\|_Y)$ be Banach spaces whose moduli of uniform smoothness and uniform convexity have power type $p$ and $q$, respectively. Suppose also that for every $n\in \N$ and every symmetric stochastic matrix $\A\in \M_n(\R)$ we have
\begin{equation*}
\gamma(\A,\|\cdot\|_{\!X}^p)\le \mathscr{C}\gamma(\A,\|\cdot\|_{\!Y}^q).
\end{equation*}
Then, the $\frac{p}{q}$-snowflake of $X$ embeds into $\ell_q(Y)$ with $q$-average distortion $D$, where
\begin{equation}\label{eq:D upper with extension}
D\lesssim \mathscr{S}_p(X)^{\frac{p}{q}}\mathscr{K}_q(Y)\mathscr{C}^{\frac{1}{q}}.
\end{equation}
\end{corollary}

Theorem~\ref{thm:full duality} is deduced in Section~\ref{sec:duality}  as a formal consequence of~\cite[Theorem~1.3]{Nao14}, which was proved by a duality argument and implies the first assertion of Theorem~\ref{thm:full duality} for average distortion embeddings of finite subsets of $X$. Those who are only interested in our application to metric dimension reduction could therefore skip Section~\ref{sec:duality}, since for this finitary application one could use~\cite[Theorem~1.3]{Nao14} as a ``black box.'' Theorem~\ref{thm:full duality} is needed only for the full embedding statements in the Introduction, which treat arbitrary Borel measures and require that the embeddings have a controlled Lipschitz constant on all of $X$ rather than only on the support of the given measure.

 The deduction of Theorem~\ref{thm:full duality} from~\cite[Theorem~1.3]{Nao14} amounts to a somewhat tedious but quite straightforward compactness argument, combined with a deep H\"older extension theorem that we use for the second assertion of Theorem~\ref{thm:full duality}, namely to obtain an embedding into $\ell_q(Y)$. It remains open if a loss as in~\eqref{eq:D upper with extension} of a multiplicative factor that depends on the data $p,q,\mathscr{S}_p(X),\mathscr{K}_q(Y)$  is needed if one wishes to obtain an embedding into $\ell_q(Y)$ rather than into its ultrapower.

\begin{remark} A version of Theorem~\ref{thm:full duality} is available in which the target $Y$ need not be a Banach space, but for that purpose further background in metric geometry is required (see~\cite{MN13-bary,Nao14}; the pertinent concepts are metric Markov cotype $q$ and $\mathsf{W}_q$-barycentric spaces). We omit the discussion since its treatment in full generality will lead to a needlessly lengthy digression.
\end{remark}

\section{Complex interpolation}\label{sec:interpolation proof} We briefly present background on the vector-valued complex interpolation method of Calder\'on~\cite{Cal64} and Lions~\cite{Lio60}; an extensive treatment can be found in e.g.~\cite[Chapter~4]{BL76}.
A pair of Banach spaces $(X,\|\cdot\|_X), (Z,\|\cdot\|_Z)$ over the complex scalars $\C$ is said to be a compatible pair of Banach spaces if they are both subspaces of a complex linear space $W$ with $X+Z=W$. The space $W$ is a complex Banach space under the norm $\|w\|_W=\inf\{\|x\|_X+\|z\|_Z:\ (x,z)\in X\times Z\ \mathrm{and}\ x+z=w\}$. Let $\mathcal{F}(X,Z)$ denote the space of all bounded continuous functions $\psi:\{\zeta\in \C:\ 0\le \Re(\zeta)\le 1\}\to  W$ that are analytic on  $\{\zeta\in \C:\ 0<\Re(\zeta)<1\}$, such that for all $t\in \R$ we have $f(ti)\in X$ and $f(1+ti)\in Z$, the mappings $t\mapsto  f(ti)$ and $t\mapsto  f(1+ti)$ are continuous relative to the norms $\|\cdot\|_X$ and $\|\cdot\|_Z$, respectively, and $\lim_{|t|\to \infty} \|f(ti)\|_X=\lim_{|t|\to \infty} \|f(1+ti)\|_Z=0$. To each $\theta\in [0,1]$ one associates  as follows a Banach space $[X,Z]_\theta$. The underlying vector space is $\{\psi(\theta):\ \psi\in \mathcal{F}(X,Z)\}\subset W$, and the norm of $w\in [X,Z]_\theta$ is $\|w\|_{[X,Z]_\theta}=\inf_{\{\psi\in \mathcal{F}(X,Z):\ \psi(\theta)=w\}}\max\{\sup_{t\in \R} \|\psi(ti)\|_{X},\sup_{t\in \R}\|\psi(1+ti)\|_Z\}$. This turns $[X,Z]_\theta$ into a Banach space. By~\cite[Theorem~4.2.1]{BL76} we have $[X,X]_\theta=X$ for  $\theta\in [0,1]$.

By~\cite[Theorem~4.2.1]{BL76}, if $X\cap Z$ is dense in both $X$ and $Z$, then  $[X,Z]_0=X$ and  $[X,Z]_1=Z$. In what follows, whenever we say that $(X,\|\cdot\|_X), (Z,\|\cdot\|_Z)$ is a compatible pair of Banach spaces we will tacitly  assume that $X\cap Z$ is dense in both $X$ and $Z$, thus ensuring that  $\{[X,Z]_\theta\}_{\theta\in [0,1]}$ is a one-parameter family of Banach spaces starting at $X$   and terminating at $Z$.

The {\em reiteration theorem}~\cite[Section~12.3]{Cal64} (see also~\cite{Cwi78} and the exposition in~\cite[Section~4.6]{BL76}) asserts that if $(X,\|\cdot\|_X),(Z,\|\cdot\|_Z)$ is a compatible pair of complex Banach spaces, then
\begin{equation}\label{eq:reiteration}
\forall\, \alpha,\beta,\theta\in [0,1],\qquad \big[[X,Z]_\alpha,[X,Z]_\beta\big]_\theta=[X,Z]_{(1-\theta)\alpha+\theta\beta}.
\end{equation}
The equality in~\eqref{eq:reiteration} means that the corresponding spaces  are linearly isometric (over $\C$). Going forward,  this is how all the ensuing equalities between complex Banach spaces are to be interpreted.

A  basic property of vector-valued complex interpolation~\cite{Lio60,Cal64}  is that if $(X,\|\cdot\|_X),(Z,\|\cdot\|_Z)$ and $(U,\|\cdot\|_U),(V,\|\cdot\|_V)$ are two compatible pairs of complex Banach spaces and $T:X\cap Z\to U\cap V$ is a linear operator that extends to a bounded linear operator from $(X,\|\cdot\|_X)$ to $(U,\|\cdot\|_U)$ and from $(Z,\|\cdot\|_Z)$ to $(V,\|\cdot\|_V)$, then the following operator norm bounds hold true.
\begin{equation}\label{eq:riesz thorin}
\forall\, \theta\in [0,1],\qquad \|T\|_{\![X,Z]_\theta\to [U,V]_\theta}^{\phantom{p}}\le \|T\|_{\!X\to U}^{1-\theta\phantom{p}}\|T\|_{\!Z\to V}^{\theta\phantom{p}}.
\end{equation}

Fix $p\ge 1$ and $n\in \N$. For a complex Banach space $(X,\|\cdot\|_X)$  and a weight $\xi:\n\to [0,\infty)$, we denote (as usual) by $L_p(\xi;X)$ the vector space $X^n$ equipped with the norm that is given by
$$
\forall(x_1,\ldots,x_n)\in X^n,\qquad \|(x_1,\ldots,x_n)\|_{L_p(\xi;X)}\eqdef \Big(\xi(1)\|x\|_{\!X}^p+\ldots+\xi(n)\|x_n\|_{\!X}^p\Big)^{\frac{1}{p}}.
$$
In particular, if $\xi(1)=\ldots=\xi(n)=1$, then $L_p(\xi;X)=\ell_p^n(X)$. Calder\'on's vector-valued version of Stein's interpolation theorem~\cite[Theorem~2]{Ste56} (see part(i) of $\S13.6$ in~\cite{Cal64} or Theorem~5.6.3 in~\cite{BL76}) asserts that if $(X,\|\cdot\|_X),(Z,\|\cdot\|_Z)$  is a compatible pair of complex Banach spaces, then for every $a,b\in [1,\infty]$, $\theta\in [0,1]$ and $\xi,\zeta:\n\to [0,\infty)$ we have
 \begin{equation}\label{eq:stein weiss}
 \left[L_a(\xi;X),L_b(\zeta;Z)\right]_\theta= L_{\frac{ab}{\theta a +(1-\theta) b}}\Big(\xi^{\frac{(1-\theta)b}{\theta a +(1-\theta) b}}\zeta^{\frac{\theta a}{\theta a +(1-\theta) b}};[X,Z]_\theta\Big).
 \end{equation}
 The special case $\xi=\zeta$ of~\eqref{eq:stein weiss}, in combination with~\eqref{eq:riesz thorin}, corresponds to the vector-valued version of the classical Riesz--Thorin interpolation theorem~\cite{Rie27,Tho48}.

\subsection{Nonlinear spectral gaps along an interpolation family}\label{sec: gap interpolation}
Our main technical result is the following theorem which (under certain geometric assumptions) controls the growth of nonlinear spectral gaps along an interpolation family $\{[X,Z]_\theta\}_{\theta\in [0,1]}$ as $\theta\to 0^+$, in terms of their value at the endpoint $\theta=1$. The relevance to Theorem~\ref{thm:really main} is through the duality of Theorem~\ref{thm:full duality}.

\begin{theorem}\label{thm:interpolation markov} There is a universal constant $\alpha\ge 1$ with the following property. Fix $\theta\in (0,1]$ and $(p,q)\in [1,2]\times [2,\infty)$. Let $(X,\|\cdot\|_X), (Z,\|\cdot\|_Z)$ be a compatible pair of complex Banach spaces. Then, for every $n\in \N$, any symmetric stochastic matrix  $\A\in \M_n(\R)$ satisfies the following inequality.
\begin{equation}\label{eq:interpolation bound for gamma}
\gamma\big(\A,\|\cdot\|_{\![X,Z]_\theta}^p\big)\le \big(\alpha\mathscr{K}_q(Z)\big)^{\!q}\cdot \frac{\mathbf{M}_p([X,Z]_\theta)^p}{\theta}\gamma(\A,\|\cdot\|_{\!Z}^q)\stackrel{\eqref{eq:markov type in theorem}}{\lesssim} \big(\alpha\mathscr{K}_q(Z)\big)^{\!q}\cdot  \frac{\mathscr{S}_p([X,Z]_\theta)^p}{\theta}\gamma(\A,\|\cdot\|_{\!Z}^q).
\end{equation}
\end{theorem}

By Theorem~\ref{thm:full duality}, Theorem~\ref{thm:interpolation markov} directly implies Theorem~\ref{thm:interpolation implicit first exposure}, yielding the following version of~\eqref{eq: D interpolation implicit} with the implicit dependence of the constant factor on the relevant parameters specified explicitly.
\begin{equation}\label{eq:specified D upper for interpolation}
D\lesssim \mathbf{M}_p([X,Z]_\theta)^{\!\frac{2p}{q}}\mathscr{K}_q(Z)^2\Big(\frac{1}{\theta}\Big)^{\!\frac{1}{q}}\stackrel{\eqref{eq:markov type in theorem}}{\lesssim} \mathscr{S}_p([X,Z]_\theta)^{\!\frac{2p}{q}}\mathscr{K}_q(Z)^2\Big(\frac{1}{\theta}\Big)^{\!\frac{1}{q}}.
\end{equation}
It is worthwhile to note in passing  that, by the first part of Theorem~\ref{thm:full duality}, a smaller factor is achievable in~\eqref{eq:specified D upper for interpolation} if one considers embeddings into an ultrapower of $\ell_q(Z)$ rather than into $\ell_q(Z)$ itself.

\subsection{Deduction of Theorem~\ref{thm:really main} from Theorem~\ref{thm:interpolation implicit first exposure}}\label{sec:proof of 9 from interpolation} We first derive some preparatory estimates.


\subsubsection{Tracking the  coefficients $\mathscr{S}_p([X,Z]_\theta)$ and $\mathscr{K}_q([X,Z]_\theta)$ as a function of $\theta\in [0,1]$}\label{sec:track} Suppose that $(X,\|\cdot\|_X), (Z,\|\cdot\|_Z)$ is a compatible pair of complex Banach spaces. Cwikel and Reisner estimated~\cite{CR82} the moduli of uniform convexity and uniform smoothness of $\{[X,Z]_\theta\}_{\theta\in [0,1]}$ in terms of the corresponding moduli of $X$ and $Z$. By combining the bounds of~\cite{CR82} with~\cite{BCL94}, it follows that for every $p_1,p_2,\in (1,2]$ and $q_1,q_2\in [2,\infty)$ we have
$$
\mathscr{S}_{\frac{p_1p_2}{\theta p_1+(1-\theta)p_2}}([X,Z]_\theta)\lesssim_{p_1,p_2} \mathscr{S}_{p_1}(X)^{1-\theta} \mathscr{S}_{p_2}(Z)^\theta,
$$
and
$$
\mathscr{K}_{\frac{q_1q_2}{\theta q_1+(1-\theta)q_2}}([X,Z]_\theta)\lesssim_{q_1,q_2} \mathscr{K}_{q_1}(X)^{1-\theta} \mathscr{K}_{q_2}(Z)^\theta.
$$

We will next adjust the approach of~\cite{CR82} so as to obtain these estimates without any dependence on $p_1,p_2,q_1,q_2$ in the implicit multiplicative factors. Namely, we will demonstrate that
\begin{equation}\label{eq:without p1,p2}
\mathscr{S}_{\frac{p_1p_2}{\theta p_1+(1-\theta)p_2}}([X,Z]_\theta)\le \mathscr{S}_{p_1}(X)^{1-\theta} \mathscr{S}_{p_2}(Z)^\theta,
\end{equation}
and
\begin{equation}\label{eq:without q1,q2}
\mathscr{K}_{\frac{q_1q_2}{\theta q_1+(1-\theta)q_2}}([X,Z]_\theta)\le \mathscr{K}_{q_1}(X)^{1-\theta} \mathscr{K}_{q_2}(Z)^\theta.
\end{equation}
Removing the dependence on $p_1,p_2,q_1,q_2$  of the  constant factors in the Cwikel--Reisner estimates is important in our context, as the parameters will  be optimized so as to depend on other quantities that we wish to track. The ensuing reasoning is nothing more than an adaptation of~\cite{CR82}.

 Suppose that $p_1,p_2\in [1,2]$ and that the smoothness constants $\mathscr{S}_{p_1}(X),\mathscr{S}_{p_2}(Z)$ are finite. Denote for simplicity $\mathscr{S}_1=\mathscr{S}_{p_1}(X)$ and $\mathscr{S}_2=\mathscr{S}_{p_2}(Z)$. Then by~\eqref{eq:def smoothness} we have
 \begin{equation}\label{eq:S1}
 \forall\, y_1,y_2\in Y,\qquad \|y_1+y_2\|_Y^{p_1}+\|y_1-y_2\|_Y^{p_1}\le 2\|y_1\|_Y^{p_1}+2\mathscr{S}_1^{p_1}\|y_2\|_Y^{p_1},
 \end{equation}
 and
  \begin{equation}\label{eq:S2}
 \forall\, z_1,z_2\in Z,\qquad \|z_1+z_2\|_Z^{p_2}+\|z_1-z_2\|_Z^{p_2}\le 2\|z_1\|_Z^{p_2}+2\mathscr{S}_2^{p_2}\|z_2\|_Z^{p_2}.
 \end{equation}

For every $\mathscr{S}>0$ and $p\ge 1$ define $\xi(\mathscr{S},p):\{1,2\}\to (0,\infty)$ by $\xi(\mathscr{S},p)(1)=2$ and $\xi(\mathscr{S},p)(2)=2\mathscr{S}^p$. Also, denote the constant function $\1_{\{1,2\}}$ by $\zeta:\{1,2\}\to (0,\infty)$, i.e., $\zeta(1)=\zeta(2)=1$. With this notation, if we consider the linear operator $T:(X+Z)\times (X+Z)\to (X+Z)\times (X+Z)$ that is given by setting $T(w_1,w_2)=(w_1+w_2,w_1-w_2)$ for every $w_1,w_2\in Y+Z$, then
\begin{equation}\label{eq:p1 p2 norms}
\|T\|_{\!L_{p_1}(\xi(\mathscr{S}_1,p_1);X)\to L_{p_1}(\zeta;Y)}^{\phantom{p}}\stackrel{\eqref{eq:S1}}{\le} 1\qquad\mathrm{and}\qquad \|T\|_{\!L_{p_2}(\xi(\mathscr{S}_2,p_2);Z)\to L_{p_2}(\zeta;Z)}^{\phantom{p}}\stackrel{\eqref{eq:S2}}{\le} 1.
\end{equation}

Denoting $r=p_1p_2(\theta p_1+(1-\theta) p_2)^{-1}$,  observe that $\xi(\mathscr{S}_1,p_1)^{\frac{(1-\theta)r}{p_1}}\xi(\mathscr{S}_2,p_2)^{\frac{\theta r}{p_2}}=\xi(\mathscr{S}_1^{1-\theta}\mathscr{S}_2^\theta,r)$. Hence, by~\eqref{eq:stein weiss} we have $[L_{p_1}(\xi(\mathscr{S}_1,p_1);X),L_{p_2}(\xi(\mathscr{S}S_2,p_2);Z)]_\theta=L_r(\xi(\mathscr{S}_1^{1-\theta}\mathscr{S}_2^\theta,r);[X,Z]_\theta)$ and also $[L_{p_1}(\zeta;X);L_{p_2}(\zeta;Z)]_\theta=L_r(\zeta,[X,Z]_\theta)$. In combination with~\eqref{eq:riesz thorin} and~\eqref{eq:p1 p2 norms}, this implies that the norm of $T$ as an operator from $L_r(\xi(\mathscr{S}_1^{1-\theta}\mathscr{S}_2^\theta,r);[X,Z]_\theta)$ to $L_r(\zeta,[X,Z]_\theta)$ is at most $1$.  Thus,
$$
\forall\, w_1,w_2\in [X,Z]_\theta,\qquad  \|w_1+w_2\|_{[X,Z]_\theta}^{r}+\|w_1-w_2\|_{[X,Z]_\theta}^{r}\le 2\|w_1\|_{[X,Z]_\theta}^{r}+2\big(\mathscr{S}_1^{1-\theta}\mathscr{S}_2^\theta\big)^{\!r}\|w_2\|_{[X,Z]_\theta}^{r}.
$$
This is precisely~\eqref{eq:without p1,p2}. The bound~\eqref{eq:without q1,q2} is justified  mutatis mutandis  via  the same reasoning (only~\eqref{eq:without p1,p2} will be used below); alternatively, one could derive~\eqref{eq:without q1,q2} from~\eqref{eq:without p1,p2} by a duality argument.

\subsubsection{Complexification}\label{sec:complexification} To make Theorem~\ref{thm:interpolation markov}, which treats  complex Banach spaces, relevant to Theorem~\ref{thm:really main}, which treats real normed spaces, we  use a standard  complexification procedure. Specifically, for a real normed space $(W,\|\cdot\|_W)$  associate as follows a complex normed space $(W_\C,\|\cdot\|_{W_\C})$. The underlying vector space is $W_\C=W\times W$, which is viewed as a vector space over $\C$ by setting $(\alpha+\beta i)(x,y)=(\alpha x-\beta y,\beta x+\alpha y)$ for  $\alpha,\beta\in \R$ and $x,y\in W$. The norm on $W_\C$ is given by
\begin{equation}\label{eq:def complexification}
\forall\, (x,y)\in W\times W,\qquad \|(x,y)\|_{\!W_\C}^{\phantom{p}}=\bigg(\frac{1}{\pi}\int_0^{2\pi}\big\|(\cos\theta)x-(\sin\theta) y\big\|^{2\phantom{p}}_{\!W}\ud \theta\bigg)^{\!\frac12}.
\end{equation}

The normalization of the integral in~\eqref{eq:def complexification} was chosen so as to ensure that $x\mapsto (x,0)$ is an isometric embedding of $W$  into $W_\C$. It is straightforward to check that for every $p\in [1,\infty]$ and $(x,y)\in W_\C$,
\begin{equation}\label{complexification is lp}
\|(x,y)\|_{\!W_\C}^{\phantom{p}}\asymp \big(\|x\|_{\!W}^{p}+\|y\|_{\!W}^{p}\big)^{\!\frac1{p}}= \|(x,y)\|_{\ell_p^2(W)}.
\end{equation}
Hence, $\gamma(\A,\|\cdot\|_{\!X}^p)\asymp\gamma(\A,\|\cdot\|_{\!X_\C}^p)$ for every $n\in \N$ and every symmetric stochastic matrix $\A\in \M_n(\R)$. Also, $\mathscr{S}_p(W_\C)\asymp\mathscr{S}_p(W)$ and $\mathscr{K}_q(W_\C)\asymp\mathscr{K}_q(W)$ for $p\in [1,2]$ and $q\in [2,\infty]$. If one were to let the implicit constants in these  equivalences to depend on $p,q$, then they would follow from~\cite{FP74,Fig76,BCL94}. The fact that the constants can be taken to be universal follows by reasoning with more care, as done in~\cite{Nao12-azuma,MN14}; see specifically Lemma~6.3 and Corollary 6.4 of~\cite{MN14}.

\subsubsection{Proof of Theorem~\ref{thm:really main}} Suppose that we are in the setting of Theorem~\ref{thm:really main}. Thus,  $(X,\|\cdot\|_X)$ and $(Y,\|\cdot\|_Y)$ are Banach spaces that satisfy $\cc_Y(X),\mathscr{S}_p(X), \mathscr{K}_q(Y)<\infty$ where $1\le p\le 2\le q<\infty$.

Fix $\cc>\cc_Y(X)$. Since $Y$ is uniformly convex and hence (by the Milman--Pettis theorem~\cite{Mil38,Pet39}) in particular reflexive, by a classical differentiation argument of Aronszajn~\cite{Aro76}, Christensen~\cite{Chr73} and Mankiewicz~\cite{Man72} (see also~\cite[Chapter~7]{BL00} for a thorough treatment of such reductions to the linear setting) there exists a {\em linear} operator $T:X\to Y$ which satisfies
\begin{equation}\label{eq:norm T assumptions}
\forall\, x\in X,\qquad \|x\|_{\!X}^{\phantom{p}}\le \|Tx\|_{\!Y}^{\phantom{p}}\le \cc \|x\|_{\!X}^{\phantom{p}}.
\end{equation}

We define a normed space $(Z,\|\cdot\|_Z)$ by setting $Z=X$ and $\|x\|_Z\eqdef\|Tx\|_Y$ for every $x\in X$. Thus, $X$ and $Z$ coincide as linear spaces and $Z$ is linearly isometric to a subspace of $Y$, via the embedding $T$. Let $X_\C$ and $Z_\C$ be the complexifications of $X$ and $Z$, respectively. Then $X_\C,Z_\C$ is a compatible pair of Banach spaces. So, we may consider the complex interpolation family $\{[X_\C,Z_\C]_\theta\}_{\theta\in [0,1]}$.

It follows from a substitution of~\eqref{eq:norm T assumptions} into the definition~\eqref{eq:def complexification} that
\begin{equation}\label{eq:lift to complexification}
\forall(x,y)\in X\times X,\qquad \|(x,y)\|^{\phantom{p}}_{\!X_\C}\le \|(x,y)\|^{\phantom{p}}_{\!Z_\C} \le \cc \|(x,y)\|^{\phantom{p}}_{\!X_\C}.
\end{equation}
Hence, the following operator norm bounds hold true for the formal identity $\Id_{X\times X}:X_\C\to X\times X$.
\begin{equation}\label{eq:norm bounds}
\|\mathsf{Id}_{X\times X}\|_{\!X_\C\to X_\C}^{\phantom{p}}\le 1,\qquad  \|\mathsf{Id}_{X\times X}\|_{\!X_\C\to Z_\C}^{\phantom{p}}\le \cc,\qquad \mathrm{and}\qquad \|\mathsf{Id}_{X\times X}\|_{\!Z_\C\to X_\C}^{\phantom{p}}\le 1.
\end{equation}
The first  inequality in~\eqref{eq:norm bounds} is tautological, and the rest of~\eqref{eq:norm bounds} is a restatement of~\eqref{eq:lift to complexification}.

For every $\theta\in [0,1]$ we have
$$
\|\mathsf{Id}_{X\times X}\|_{\![X_\C,Z_\C]_\theta\to X_\C}^{\phantom{p}}=\|\mathsf{Id}_{X\times X}\|_{\![X_\C,Z_\C]_\theta\to [X_\C,X_\C]_\theta}^{\phantom{p}}\stackrel{\eqref{eq:riesz thorin}}{\le} \|\mathsf{Id}_{X\times X}\|_{\!X_\C\to X_\C}^{1-\theta\phantom{p}}\|\mathsf{Id}_{X\times X}\|_{\!Z_\C\to X_\C}^{\theta\phantom{p}}\stackrel{\eqref{eq:norm bounds}}{\le} 1,
$$
and
$$
\|\mathsf{Id}_{X\times X}\|_{\! X_\C\to [X_\C,Z_\C]_\theta}^{\phantom{p}}=\|\mathsf{Id}_{X\times X}\|_{ \![X_\C,X_\C]_\theta\to [X_\C,Z_\C]_\theta}^{\phantom{p}}\stackrel{\eqref{eq:riesz thorin}}{\le} \|\mathsf{Id}_{X\times X}\|_{\!X_\C\to X_\C}^{1-\theta\phantom{p}}\|\mathsf{Id}_{X\times X}\|_{\!X_\C\to Z_\C}^{\theta\phantom{p}}\stackrel{\eqref{eq:norm bounds}}{\le} \cc^\theta.
$$
In other words, this simple reasoning yields the following bounds.
\begin{equation}\label{eq:distance to interpolant}
\forall (x,y)\in X\times X,\qquad \|(x,y)\|_{\!X_\C}^{\phantom{p}}\le \|(x,y)\|_{\![X_\C,Z_\C]_\theta}^{\phantom{p}}\le \cc^\theta \|(x,y)\|_{\!X_\C}^{\phantom{p}}.
\end{equation}

Fix $\s\ge 1$. By Theorem~\ref{thm:interpolation implicit first exposure}, applied with the value of $D$ in~\eqref{eq:specified D upper for interpolation}, for any Borel probability measure $\mu$ on $X$ (viewed as a subset of $X_\C$) there is $f:[X_\C,Z_\C]_\theta\to \ell_q(Z_\C)$ satisfying
\begin{equation}\label{eq:before removinf Z}
\forall\, x,y\in X,\qquad \|f(x)-f(y)\|_{\ell_q(Z_\C)}\lesssim \frac{\mathbf{M}_\s([X_\C,Z_\C]_\theta)^{\!\frac{2\sigma}{q}}\mathscr{K}_q(Z_\C)^2}{\theta^{\frac{1}{q}}}
\|x-y\|_{\![X_\C,Z_\C]_\theta}^{\!\frac{\sigma}{q}},
\end{equation}
and
\begin{equation}\label{eq:q average on Z}
\iint_{X\times X} \|f(x)-f(y)\|_{\ell_q(Z_\C)}^q\ud\mu(x)\ud\mu(y)\ge \iint_{X\times X} \|x-y\|_{\!X}^{\s\phantom p}\mu(x)\ud\mu(y).
\end{equation}
By~\eqref{complexification is lp} and the definition of $Z$, there is a linear map $S:\ell_q(Z_\C)\to \ell_q(Y)$ with $\|Sw\|_{\ell_q(Y)}\asymp \|w\|_{\ell_q(Z_\C)}$ for all $w\in \ell_q(X\times X)$. Indeed, if we write $w=((x_1,y_1),(x_2,y_2),\ldots)$ for $\{x_i\}_{i=1}^\infty,\{y_i\}_{i=1}^\infty\subset X$, then simply take $Sw=(Tx_1,Ty_1,Tx_2,Ty_2,\ldots)$. So, by considering $\f=S\circ f: X\to \ell_q(Y)$, we get
\begin{equation}\label{eq:before s[pecifying sigma 1}
\forall\, x,y\in X,\qquad \|\f(x)-\f(y)\|_{\ell_q(Y)}\stackrel{\eqref{eq:distance to interpolant}\wedge \eqref{eq:before removinf Z}}{\lesssim} \frac{\cc^{\!\frac{\s\theta}{q}}\mathbf{M}_\s([X_\C,Z_\C]_\theta)^{\!\frac{2\s}{q}}}{\theta^{\frac{1}{q}}}\mathscr{K}_q(Y)^2\|x-y\|_{\!X}^{\!\frac{\s}{q}},
\end{equation}
and
\begin{equation}\label{eq:before s[pecifying sigma 2}
\iint_{X\times X} \|\f(x)-\f(y)\|_{\ell_q(Y)}^q\ud\mu(x)\ud\mu(y)\stackrel{\eqref{eq:q average on Z}}{\gtrsim} \iint_{X\times X} \|x-y\|_{\!X}^{\s\phantom p}\mu(x)\ud\mu(y).
\end{equation}

The first part of Theorem~\ref{thm:really main} makes no assumption on the uniform smoothness of $Y$. So, apply~\eqref{eq:before s[pecifying sigma 1} with $\sigma=p$  while noting that $\mathbf{M}_p([X_\C,Z_\C]_\theta)\le \cc^\theta \mathbf{M}_p(X_\C)\lesssim \cc^\theta \mathscr{S}_p(X_\C)\lesssim  \cc^\theta \mathscr{S}_p(X)$, where the first step holds due to~\eqref{eq:distance to interpolant} and  the second step is Theorem~\ref{thm:quote NPSS}. We thus arrive at the H\"older condition
\begin{equation}\label{eq:use the interpolation gap theorem}
\forall\, x,y\in X,\qquad \|\f(x)-\f(y)\|_{\ell_q(Y)}\lesssim \frac{\cc^{\!\frac{3p\theta}{q}}}{\theta^{\frac{1}{q}}}\mathscr{K}_q(Y)^2\mathscr{S}_p(X)^{\!\frac{2p}{q}}\|x-y\|_{\!X}^{\!\frac{p}{q}}.
\end{equation}
The optimal choice of $\theta$ in~\eqref{eq:use the interpolation gap theorem} satisfies $\theta\asymp\frac{1}{\log(\cc+1)}$, yielding the following explicit version of~\eqref{eq:rough D upper bound convexity smoothness}.
\begin{equation}\label{eq:optimal theta substituted}
D\lesssim \mathscr{S}_p(X)^{\!\frac{2p}{q}}\mathscr{K}_q(Y)^2\big(\log(\cc_Y(X)+1)\big)^{\!\frac{1}{q}}.
\end{equation}
We note in passing that with more care one obtains~\eqref{eq:use the interpolation gap theorem} with the term $\cc^{\frac{3p\theta}{q}}$ replaced by $\cc^{\frac{p\theta}{q}}$. But, upon choosing the optimal $\theta$ as we do here, this only influences the universal constant factor in~\eqref{eq:optimal theta substituted}.

For the second part of Theorem~\ref{thm:really main}, we are now assuming that $Y$ is more uniformly smooth than $X$, namely that $\mathscr{S}_r(Y)<\infty$ for some $r\in (p,2]$. Under this stronger assumption, we fix $\e\in [0,\frac{r-p}{q}]$ and the aim is now to obtain an embedding of $X$ into $\ell_q(Y)$ with higher regularity than in the first part of Theorem~\ref{thm:really main}, namely an embedding of the $(\frac{p}{q}+\e)$-snowflake of $X$ rather than of its $\frac{p}{q}$-snowflake. To this end, we apply~\eqref{eq:before s[pecifying sigma 1} and~\eqref{eq:before s[pecifying sigma 2} with $\sigma=p+\e q$ and  $\theta\in [0,1]$ satisfying
\begin{equation}\label{eq:range of theta}
 \frac{\e  q r}{(r-p)(p+\e q)}\le \theta\le 1.
\end{equation}
Note that the range of possible values of $\theta$ in~\eqref{eq:range of theta} is nonempty due to the assumption   $\e\le \frac{r-p}{q}$. The lower bound on $\theta$ in~\eqref{eq:range of theta} is equivalent to  $p+\e q\le \frac{pr}{\theta p+(1-\theta)r}$, and therefore
\begin{align}\label{eq:p plus eps q}
\begin{split}
\mathbf{M}_{p+\e q}([X_\C,Z_\C]_\theta)\stackrel{\eqref{eq:markov type in theorem}}{\lesssim} \mathscr{S}_{p+\e q}([X_\C,Z_\C]_\theta)&\le \mathscr{S}_{\frac{pr}{\theta p+(1-\theta)r}}([X_\C,Z_\C]_\theta)\\&\!\!\stackrel{\eqref{eq:without p1,p2}}{\le} \mathscr{S}_p(X_\C)^{1-\theta}\mathscr{S}_r(Z_\C)^\theta\lesssim \mathscr{S}_p(X)^{1-\theta}\mathscr{S}_r(Y)^\theta,
\end{split}
\end{align}
where the second step of~\eqref{eq:p plus eps q} uses the fact that $p\mapsto \mathscr{S}_p(\cdot)$ is increasing (see~\cite{BCL94} or~\cite[Section~6.2]{MN14}) and the last step of~\eqref{eq:p plus eps q} holds as $Z$ is isometric to a subspace of $Y$.   A substitution of~\eqref{eq:p plus eps q} into~\eqref{eq:before s[pecifying sigma 1}  shows that the $(\frac{p}{q}+\e)$-snowflake of $X$ embeds with $q$-average distortion $D$ into $\ell_q(Y)$, where
\begin{equation}\label{eq: Dr}
D\lesssim \mathscr{S}_p(X)^{2\left(\frac{p}{q}+\e\right) }\mathscr{K}_q(Y)^2\cdot \frac{\bigg(\frac{\mathscr{S}_r(Y)^2}{\mathscr{S}_p(X)^2}\cc_Y(X)\bigg)^{\!\!\left(\frac{p}{q}+\e \right)\theta}}{\theta^{\frac{1}{q}}}.
\end{equation}
By choosing $\theta$ so as to minimise the right hand side of~\eqref{eq: Dr} subject to the constraint~\eqref{eq:range of theta}, this leads to the following more refined version of the desired bound~\eqref{eq:D cases}.
\begin{equation*}
D\lesssim \left\{\begin{array}{ll} \mathscr{S}_p(X)^{2\left(\frac{p}{q}+\e\right) }\mathscr{K}_q(Y)^2\left(\log\Big(\frac{\mathscr{S}_r(Y)^2}{\mathscr{S}_p(X)^2}\cc_Y(X)+2\Big)\right)^{\!\frac{1}{q}}&\mathrm{if\ } 0\le \e\le \frac{r-p}{qr\log\Big(\frac{\mathscr{S}_r(Y)^2}{\mathscr{S}_p(X)^2}\cc_Y(X)+2\Big)},\\\left(\frac{r-p}{\e }\right)^{\!\frac{1}{q}}\mathscr{K}_q(Y)^2\mathscr{S}_r(Y)^{\frac{2\e  r}{r-p}} \mathscr{S}_p(X)^{2p\left(1-\frac{\e q }{r-p}\right)}\cc_Y(X)^{\frac{\e  r}{r-p}} &\mathrm{if\ } \frac{r-p}{qr\log\Big(\frac{\mathscr{S}_r(Y)^2}{\mathscr{S}_p(X)^2}\cc_Y(X)+2\Big)}\le \e\le \frac{r-p}{q}.
\end{array}\right.
\end{equation*}
This completes the deduction  of Theorem~\ref{thm:really main} from Theorem~\ref{thm:interpolation implicit first exposure}.\qed

\begin{remark} Continuing with the notation and assumptions of the above proof of Theorem~\ref{thm:really main} in the special case when $Y=H$ is a Hilbert space and $q=2$, Corollary~4.7 of~\cite{Nao14} asserts\footnote{We note that in~\cite{Nao14} (specifically, in the statement of~\cite[Theorem~4.5]{Nao14}) we have the following misprint: \eqref{eq:XCHC} is stated there for the transposed interpolation space  $[Z_\C,X_\C]_\theta$ rather than the correct space $[X_\C,Z_\C]_\theta$ as above. } that
\begin{equation}\label{eq:XCHC}
\gamma\!\left(\A,\|\cdot\|_{\![X_\C,Z_\C]}^2\right)\lesssim \frac{\mathscr{S}_p([X_\C,Z_\C])^2}{\theta^{\frac{2}{p}}\big(1-\lambda_2(\A)\big)^{\!\frac{2}{p}}}.
\end{equation}
Since $Z_\C$ is (isometrically) a Hilbert space and therefore $\mathscr{S}_2(Z_\C)=1$, by~\eqref{eq:without p1,p2} we have
$$
\mathscr{S}_{\frac{2p}{p\theta+2(1-\theta)}}([X_\C,Z_\C])\le \mathscr{S}_p(X_\C)^{1-\theta}\lesssim \mathscr{S}_p(X)^{1-\theta}.
$$
Arguing the same as above, by substituting this bound into~\eqref{eq:XCHC} and using~\eqref{eq:distance to interpolant} we get that
\begin{equation}\label{eq:to optimize gamma 2}
\gamma\big(\A,\|\cdot\|^2_{\!X}\big)\lesssim \frac{\cc_2(X)^{2\theta} \mathscr{S}_p(X)^{2(1-\theta)}}{\theta^{\frac{2}{p}}\big(1-\lambda_2(\A)\big)^{\!\theta+\frac{2(1-\theta)}{p}}}.
\end{equation}
By choosing $\theta\in [0,1]$ so as to minimize the right hand side of~\eqref{eq:to optimize gamma 2}, we see that
\begin{equation}\label{eq:stronger 2}
\gamma\big(\A,\|\cdot\|^2_{\!X}\big)\lesssim \frac{\mathscr{S}_p(X)^2}{\big(1-\lambda_2(\A)\big)^{\!\frac{2}{p}}}\bigg(\log\bigg(\frac{\cc_2(X)^p\big(1-\lambda_2(\A)\big)^{\!1-\frac{p}{2}}}
{\mathscr{S}_p(X)^p}+1\bigg)\bigg)^{\!\!\frac{2}{p}}.
\end{equation}

In particular, if $\dim(X)=k\in \{2,3,\ldots\}$ and $p=1$, by John's theorem~\eqref{eq:stronger 2} implies that
\begin{equation}\label{eq:get weaker quadratic}
\gamma\big(\A,\|\cdot\|^2_{\!X}\big)\lesssim \bigg(\frac{\log\big(\cc_2(X)\sqrt{1-\lambda_2(\A)}+1\big)}{1-\lambda_2(\A)}\bigg)^{\!2}\lesssim \frac{(\log k)^2}{\big(1-\lambda_2(\A)\big)^2}.
\end{equation}
Note that if one is interested only in the rightmost quantity in~\eqref{eq:get weaker quadratic} as an upper bound on  $\gamma\big(\A,\|\cdot\|^2_{\!X}\big)$, then one simply needs to substitute $\theta\asymp 1/\log k$ into~\eqref{eq:XCHC} and use~\eqref{eq:distance to interpolant} as above. This slightly weaker estimate can be rewritten as the assertion that there exists a universal constant $\mathsf{K}\ge 1$ such that
\begin{align}\label{eq:quadratic matrix dimension}
\begin{split}
\dim(X)&\ge \exp \!\bigg(\frac{1-\lambda_2(\A)}{\mathsf{K}}\sqrt{\gamma\big(\A,\|\cdot\|^2_{\!X}\big)}\bigg)\\&\!\!\stackrel{\eqref{eq:nonliear gap def}}{=}\sup_{x_1,\ldots,x_n\in X} \exp\!\Bigg(\frac{1-\lambda_2(\A)}{\mathsf{K}}\bigg(\frac{\frac{1}{n^2}\sum_{i=1}^n\sum_{j=1}^n \|x_i-x_j\|_{\!X}^2}{\frac{1}{n}\sum_{i=1}^n\sum_{j=1}^n a_{ij}\|x_i-x_j\|_{\!X}^2}\bigg)^{\!\!\frac{1}{2}}\Bigg).
\end{split}
\end{align}
\eqref{eq:quadratic matrix dimension} corresponds to the case $p=2$ of the matrix-dimension inequality~\eqref{eq:beta p version}. As explained in the Introduction, \eqref{eq:quadratic matrix dimension}  is a formal consequence of the average John theorem of Theorem~\ref{thm:average john}. We do not see how to deduce Theorem~\ref{thm:average john} formally from~\eqref{eq:quadratic matrix dimension}; we conjecture   that such a reverse implication is impossible in general but we did not devote substantial effort to obtain a counterexample.
\end{remark}

\section{Proof of Theorem~\ref{thm:interpolation markov}}\label{sec:proof of gap interpolation thm}

Suppose that $(\cM,d_\cM)$ is a metric space, $n\in \N$ and $p\in (0,\infty)$. Following~\cite{MN14} and in analogy to~\eqref{eq:nonliear gap def}, the (reciprocal of) the {\em nonlinear absolute spectral gap} with respect to $d_{\cM}^p$ of a symmetric stochastic matrix $\A=(a_{ij})\in \M_n(\R)$, denoted $\gp(\A,d_\cM^p)$, is the smallest $\gp\in (0,\infty]$ such that
\begin{equation}\label{eq:def gp}
\forall\,x_1,\ldots,x_n,y_1,\ldots,y_n\in \cM,\qquad \frac{1}{n^2}\sum_{i=1}^n \sum_{j=1}^n d_{\cM}(x_i,y_j)^p\le \frac{\gp}{n} \sum_{i=1}^n \sum_{j=1}^n a_{ij}d_{\cM}(x_i,y_j)^p.
\end{equation}
The reason for this terminology is that by simple linear-algebraic considerations (e.g.~\cite{MN14}) one has
$$
\gp(\A,d_\R^2)=\gp(\A,|\cdot-\cdot|^2:\R\times \R\to \R)=\frac{1}{1-\max_{i\in \{2,\ldots,n\}}|\lambda_i(\A)|}.
$$
Observe that the definition directly implies that $\gamma(\A,d_\cM^p)\le \gp(\A,d_\cM^p)$.

We record for ease of later reference the following elementary relation~\cite[Lemma~2.3]{Nao14} between nonlinear spectral gaps and their absolute counterparts. In its formulation, as well as in the rest of what follows, for every $n\in \N$ the $n$-by-$n$ identity matrix is denoted $\I_n\in \M_n(\R)$.

\begin{lemma}\label{lem:lazy} Fix $q\in [1,\infty)$ and $n\in \N$. For every symmetric stochastic matrix $\A\in \M_n(\R)$ and every metric space $(\MM,d_\MM)$ we have
\begin{equation}\label{eq:lazy A}
2\gamma(\A,d_\MM^q)\le \gp\Big(\frac12\I_n+\frac12 \A,d_\MM^q\Big)\le 2^{2q+1}\gamma(\A,d_\MM^q).
\end{equation}
\end{lemma}

Even though our ultimate goal here is to bound nonlinear spectral gaps, one of the  advantages of considering nonlinear absolute spectral gaps is that, in the case of uniformly convex normed spaces, they have a useful connection  to operator norms. Specifically, suppose that $(X,\|\cdot\|_X)$ is a normed space, $n\in \N$ and $\A\in \M_n(\R)$ is symmetric and stochastic. For $p\in [1,\infty)$ let $\ell_p^n(X)_0\subset \ell_p^n(X)$ denote the subspace of those $(x_1,\ldots,x_n)\in X^n$ for which $\sum_{i=1}^nx_i=0$. If $\Id_X:X\to X$ denotes the formal identity operator on $X$, then, since $\A$ is symmetric and stochastic, the operator $\A{\boldsymbol{\smallotimes}} \Id_X:X^n\to X^n$ preserves $\ell_p^n(X)_0$, where we recall that $(\A{\boldsymbol{\smallotimes}} \Id_X)(x_1,\ldots,x_n)= (\sum_{j=1}^n a_{1j}x_j,\ldots, \sum_{j=1}^n a_{nj}x_j)$ for each $(x_1,\ldots,x_n)\in X^n$. One can therefore consider the operator norm $\|\A{\boldsymbol{\smallotimes}} \Id_X\|_{\ell_p^n(X)_0\to\ell_p^n(X)_0}$.

The following lemma coincides with~\cite[Lemma~6.1]{MN14}.

\begin{lemma}\label{lem:quote norm bound general} For every $q\in [1,\infty)$, every $n\in \N$, every symmetric and stochastic matrix $\A\in \M_n(\R)$, and every normed space $(X,\|\cdot\|_{\!X}^{\phantom{p}})$, we have
\begin{equation}\label{eq:from norm to spectral}
\gp\!\left(\A,\|\cdot\|_{\!X}^q\right)\le \bigg(1+\frac{4}{1-\|\A{\boldsymbol{\smallotimes}} \Id_X\|_{\ell_q^n(X)_0\to\ell_q^n(X)_0}}\bigg)^{\!q}.
\end{equation}
\end{lemma}

Even a weak converse to Lemma~\ref{lem:quote norm bound general}, namely the ability to bound $\|\A{\boldsymbol{\smallotimes}} \Id_X\|_{\ell_q^n(X)_0\to\ell_q^n(X)_0}$ from above away from $1$ by a quantity that may depend on $\gp(\A,\|\cdot\|_{\!X}^q)$  and $q$ but not on $n$, fails for a general normed space $(X,\|\cdot\|_X)$; see~\cite[Section~6.1]{MN14}. However, if $(X,\|\cdot\|_X)$ is uniformly convex, then we have the following converse statement along these lines, due to~\cite[Lemma~6.6]{MN14}.

\begin{lemma}\label{lem:reverse norm bound in UC} Suppose that $q\in [2,\infty)$ and that $(X,\|\cdot\|_X)$ is a normed space for which $\mathscr{K}_q(X)<\infty$. Then, for every $n\in \N$ and every symmetric and stochastic matrix $\A\in \M_n(\R)$, we have
\begin{equation}\label{eq:exp}
\|\A{\boldsymbol{\smallotimes}} \Id_X\|_{\ell_q^n(X)_0\to\ell_q^n(X)_0}\le \bigg(1-\frac{1}{(2^{q-1}-1)\mathscr{K}_q(X)^q\gp\!\!\left(\A,\|\cdot\|_{\!X}^q\right)}\bigg)^{\!\!\frac{1}{q}}.
\end{equation}
\end{lemma}

The proof of the following lemma is an adaptation of the proof of~\cite[Theorem~4.15]{Nao14}.

\begin{lemma}\label{lem:interpolate coro} Suppose that $p\in [1,\infty)$ and $q\in [2,\infty)$. Let $(X,\|\cdot\|_X), (Z,\|\cdot\|_Z)$ be a compatible pair of complex Banach spaces such that $\mathscr{K}_q(Z)<\infty$. Fix $n\in \N$ and a symmetric stochastic matrix  $\A\in \M_n(\R)$. If $\theta\in (0,1]$ and $s\in \N$ satisfy
\begin{equation}\label{eq:s condition}
s\ge \big(8\mathscr{K}_q(Z)\big)^{\!q}\gamma\!\left(\A,\|\cdot\|_{\!Z}^q\right)\max\left\{\frac{q}{\theta},p,\frac{p(q-1)}{p-1}\right\},
\end{equation}
then,
$$
\gamma\!\left(\Big(\frac12\I_n+\frac12\A\Big)^{\!s},\|\cdot\|_{\![X,Z]_\theta}^p\right)\le \gp\!\left(\Big(\frac12\I_n+\frac12\A\Big)^{\!s},\|\cdot\|_{\![X,Z]_\theta}^p\right)\le e^{O(p)}.
$$  In particular, $\gamma\!\left(\big(\frac12\I_n+\frac12\A\big)^{\!s},\|\cdot\|_{\![X,Z]_\theta}^p\right)\lesssim 1$ for some $s\asymp \frac{1}{\theta}\big(9\mathscr{K}_q(Z)\big)^{\!q}\gamma\!\left(\A,\|\cdot\|_{\!Z}^q\right)$.
\end{lemma}

\begin{proof}Suppose first that $\frac{q}{\theta+(1-\theta)q}\le  p\le \frac{q}{\theta}$, or equivalently that $\max\left\{\frac{q}{\theta},p,\frac{p(q-1)}{p-1}\right\}=\frac{q}{\theta}$. Then,
\begin{equation}\label{eq:choose r}
\frac{1}{p}=\frac{1-\theta}{r}+\frac{\theta}{q}, \qquad\mathrm{where}\qquad r\eqdef \frac{(1-\theta)pq}{q-\theta p} \in [1,\infty].
\end{equation}
For $p$ in the above range, the assumption~\eqref{eq:s condition} on $s$ is equivalent to the bound
\begin{equation}\label{eq:theta s q}
\big(9\mathscr{K}_q(Z)\big)^{\!q}\gamma\!\left(\A,\|\cdot\|_{\!Z}^q\right)\le \frac{\theta s}{q}.
\end{equation}

Let $\mathsf{J_n}\in \M_n(\R)$ be the matrix all of whose entries equal $\frac{1}{n}$. Set $\mathsf{Q}_n\eqdef \I_n-\mathsf{J}_n$. By convexity and the triangle inequality, $\|\mathsf{Q}_n{\boldsymbol{\smallotimes}} \Id_W\|_{\ell_a(W)\to \ell_a(W)_0}\le  2$ for any Banach space $(W,\|\cdot\|_W)$ and $a\ge 1$. So,
\begin{align}\label{use power s}
\begin{split}
\Big\|\Big(\frac12\I_n+\frac12 \A\Big)^{\!s}\mathsf{Q}_n&{\boldsymbol{\smallotimes}} \Id_W\Big\|_{\ell_a^n(W)\to \ell_a^n(W)} \\&=\bigg\|\bigg(\Big(\frac12\I_n+\frac12 \A\Big){\boldsymbol{\smallotimes}} \Id_W\bigg)^{\!s}(\mathsf{Q}_n{\boldsymbol{\smallotimes}} \Id_W)\bigg\|_{\ell_a^n(W)\to \ell_a^n(W)_0}\\
&\le \Big\|\Big(\frac12\I_n+\frac12 \A\Big){\boldsymbol{\smallotimes}} \Id_W\Big\|_{\ell_a^n(W)_0\to \ell_a^n(W)_0}^s\big\|\mathsf{Q}_n{\boldsymbol{\smallotimes}} \Id_W\big\|_{\ell_a^n(W)\to \ell_a^n(W)_0}\\&\le 2\Big\|\Big(\frac12\I_n+\frac12 \A\Big){\boldsymbol{\smallotimes}} \Id_W\Big\|_{\ell_a^n(W)_0\to \ell_a^n(W)_0}^s.
\end{split}
\end{align}
Due to~\eqref{eq:choose r}, by~\eqref{eq:stein weiss} we have $\ell_p^n([X,Z]_\theta)=[\ell_r^n(X),\ell_q^n(Z)]_\theta$. Consequently,
\begin{multline}\label{eq:riesz thorin used here}
\Big\|\Big(\frac12\I_n+\frac12 \A\Big)^{\!s}{\boldsymbol{\smallotimes}} \Id_{[X,Z]_\theta}\Big\|_{\ell_p^n([X,Z]_\theta)_0\to \ell_p^n([X,Z]_\theta)_0}\le \Big\|\Big(\frac12\I_n+\frac12 \A\Big)^{\!s}\mathsf{Q}_n{\boldsymbol{\smallotimes}} \Id_{[X,Z]_\theta}\Big\|_{\ell_p^n([X,Z]_\theta)\to \ell_p^n([X,Z]_\theta)}\\\stackrel{\eqref{eq:riesz thorin}\wedge \eqref{use power s}}{\le} 2\Big\|\Big(\frac12\I_n+\frac12 \A\Big){\boldsymbol{\smallotimes}} \Id_X\Big\|_{\ell_r^n(X)_0\to \ell_r^n(X)_0}^{(1-\theta)s}\Big\|\Big(\frac12\I_n+\frac12 \A\Big){\boldsymbol{\smallotimes}} \Id_Z\Big\|_{\ell_q^n(Z)_0\to \ell_q^n(Z)_0}^{\theta s}.
\end{multline}

We claim that
\begin{equation}\label{eq:use UC for power}
\Big\|\Big(\frac12\I_n+\frac12 \A\Big){\boldsymbol{\smallotimes}} \Id_X\Big\|_{\ell_r^n(X)_0\to \ell_r^n(X)_0}\le 1\quad \mathrm{and}\quad \Big\|\Big(\frac12\I_n+\frac12 \A\Big){\boldsymbol{\smallotimes}} \Id_Z\Big\|_{\ell_q^n(Z)_0\to \ell_q^n(Z)_0}\le \bigg(1-\frac{q}{\theta s}\bigg)^{\!\!\frac{1}{q}}.
\end{equation}
Indeed, the first inequality in~\eqref{eq:use UC for power} follows from the convexity of the $\ell_r^n(X)_0$ norm, because $\frac12\I_n+\frac12 \A$ is a stochastic matrix. The second inequality in~\eqref{eq:use UC for power} is justified as follows.
\begin{align*}
\begin{split}
\Big\|\Big(\frac12\I_n+\frac12 \A\Big){\boldsymbol{\smallotimes}} \Id_Z\Big\|_{\ell_q^n(Z)_0\to \ell_q^n(Z)_0}&\stackrel{\eqref{eq:exp}}{\le} \bigg(1-\frac{1}{2^{q-1}\mathscr{K}_q(Z)^q\gp\!\!\left(\frac12 \I_n+\frac12\A,\|\cdot\|_{\!Z}^q\right)}\bigg)^{\!\!\frac{1}{q}}\\&\stackrel{\eqref{eq:lazy A}}{\le}
\bigg(1-\frac{1}{\big(8\mathscr{K}_q(Z)\big)^{\!q}\gamma\!\!\left(\A,\|\cdot\|_{\!Z}^q\right)}\bigg)^{\!\!\frac{1}{q}}\stackrel{\eqref{eq:theta s q}}{\le} \bigg(1-\frac{q}{\theta s}\bigg)^{\!\!\frac{1}{q}}.
\end{split}
\end{align*}
A substitution of~\eqref{eq:use UC for power} into~\eqref{eq:riesz thorin used here} gives
\begin{equation}\label{eq:2/e}
\Big\|\Big(\frac12\I_n+\frac12 \A\Big)^{\!s}{\boldsymbol{\smallotimes}} \Id_{[X,Z]_\theta}\Big\|_{\ell_p^n([X,Z]_\theta)_0\to \ell_p^n([X,Z]_\theta)_0}\le 2\bigg(1-\frac{q}{\theta s}\bigg)^{\!\!\frac{\theta s}{q}}\le \frac{2}{e}.
\end{equation}
Hence,
$$
\gp\!\left(\Big(\frac12\I_n+\frac12\A\Big)^{\!s},\|\cdot\|_{\![X,Z]_\theta}^p\right)\stackrel{\eqref{eq:from norm to spectral}\wedge \eqref{eq:2/e}}{\le} \Big(\frac{5e-2}{e-2}\Big)^p= e^{O(p)}.
$$

This proves Lemma~\ref{lem:interpolate coro} when $\frac{q}{\theta+(1-\theta)q}\le  p\le \frac{q}{\theta}$. If $p\in [1,\infty]\setminus  [\frac{q}{\theta+(1-\theta)q},\frac{q}{\theta}]$, then define
$$
\tau\eqdef \min\left\{\frac{q}{p},\frac{q(p-1)}{p(q-1)}\right\} \qquad\mathrm{and}\qquad \alpha\eqdef 1-(1-\theta)\max\left\{\frac{p}{p-q},\frac{p(q-1)}{q-p}\right\}.
$$
One checks that $\theta=(1-\tau)\alpha+\tau$. Also, the assumption on $p$ ensures that $\alpha,\tau\in [0,1]$ and
$$
\max\left\{\frac{q}{\tau},p,\frac{p(q-1)}{p-1}\right\}=\frac{q}{\tau}=\max\left\{p,\frac{p(q-1)}{p-1}\right\}.
$$
By the reiteration theorem~\eqref{eq:reiteration}, we have $[X,Z]_\theta=\big[[X,Z]_\alpha,[X,Z]_1\big]_\tau=\big[[X,Z]_\alpha,Z\big]_\tau$. So, the rest of Lemma~\ref{lem:interpolate coro}  becomes the case that we already proved upon replacing $X$ by $[X,Z]_\alpha$ and $\theta$ by $\tau$.\end{proof}

\begin{proof}[Completion of the proof of Theorem~\ref{thm:interpolation markov}] Continue with the notation and assumptions of Theorem~\ref{thm:interpolation markov}. Fix  $x_1,\ldots,x_n\in [X,Z]_\theta$. By the special case $\MM=Y=[X,Z]_\theta$ of Proposition~\ref{prop:other exponents} (which is yet to be proven, but this is done in Section~\ref{sec:normalization}), we know that the $\frac{p}{2}$-snowflake of $[X,Z]_\theta$ embeds with quadratic average distortion $O(1)$ back  into $[X,Z]_\theta$. An application of this conclusion to the uniform measure on $\{x_1,\ldots,x_n\}$  provides new points $y_1,\ldots,y_n\in [X,Z]_\theta$ satisfying
\begin{equation}\label{eq:p/2 holder}
\forall\, i,j\in \n,\qquad \|y_i-y_j\|_{\![X,Z]_\theta}^{\phantom{p}}\lesssim \|x_i-x_j\|_{\![X,Z]_\theta}^{\!\frac{p}{2}},
\end{equation}
and
\begin{equation}\label{eq:average x y in proof}
\frac{1}{n^2}\sum_{i=1}^n\sum_{j=1}^n \|y_i-y_j\|_{\![X,Z]_\theta}^2\ge \frac{1}{n^2}\sum_{i=1}^n\sum_{j=1}^n \|x_i-x_j\|_{\![X,Z]_\theta}^p.
\end{equation}

Write $\A=(a_{ij})\in \M_n(\R)$. By Lemma~\ref{lem:interpolate coro} we know that
\begin{equation}\label{eq:choose our s}
\exists\, s\in \N,\qquad s\asymp \frac{1}{\theta}\big(9\mathscr{K}_q(Z)\big)^{\!q}\gamma\!\left(\A,\|\cdot\|_{\!Z}^q\right)\qquad\mathrm{and}\qquad \gamma\!\left(\Big(\frac12\I_n+\frac12\A\Big)^{\!s},\|\cdot\|_{\![X,Z]_\theta}^p\right)\lesssim  1.
\end{equation}
Fixing $s\in \N$ as in~\eqref{eq:choose our s}, we reason as follows.
\begin{align*}
\frac{1}{n^2}\sum_{i=1}^n\sum_{j=1}^n \|x_i-x_j\|_{\![X,Z]_\theta}^p&\stackrel{\eqref{eq:average x y in proof}}{\le} \frac{1}{n^2}\sum_{i=1}^n\sum_{j=1}^n \|y_i-y_j\|_{\![X,Z]_\theta}^2\\&\stackrel{\eqref{eq:nonliear gap def}}{\le}
\frac{1}{n}\gamma\!\left(\Big(\frac12\I_n+\frac12\A\Big)^{\!s},\|\cdot\|_{\![X,Z]_\theta}^p\right)\sum_{i=1}^n\sum_{j=1}^n \Big(\frac12\I_n+\frac12\A\Big)^{\!s}_{\! ij}\|y_i-y_j\|_{\![X,Z]_\theta}^2\\
&\!\!\!\!\!\!\stackrel{\eqref{eq:choose our s}\wedge \eqref{eq:p/2 holder}}{\lesssim}\frac{1}{n}\sum_{i=1}^n\sum_{j=1}^n \Big(\frac12\I_n+\frac12\A\Big)^{\!s}_{\! ij}\|x_i-x_j\|_{\![X,Z]_\theta}^p\\&
\stackrel{\eqref{eq:Mtype def}}{\le} \frac{\mathbf{M}_p([X,Z]_\theta)^ps}{n}\sum_{i=1}^n\sum_{j=1}^n \frac12a_{ij}\|x_i-x_j\|_{\![X,Z]_\theta}^p\\
&\stackrel{\eqref{eq:choose our s}}{\asymp} \frac{\mathbf{M}_p([X,Z]_\theta)^p\big(9\mathscr{K}_q(Z)\big)^{\!q}\gamma\!\left(\A,\|\cdot\|_{\!Z}^q\right)}{\theta n}\sum_{i=1}^n\sum_{j=1}^n a_{ij}\|x_i-x_j\|_{\![X,Z]_\theta}^p.
\end{align*}
Because this holds for every $x_1,\ldots,x_n\in [X,Z]_\theta$, by the definition~\eqref{eq:nonliear gap def} this is precisely~\eqref{eq:interpolation bound for gamma}.
\end{proof}

\begin{remark} In~\cite[Section~6]{ANNRW18} and~\cite[Section~5]{Nao18} we presented a proof of the quadratic inequality~\eqref{eq:get weaker quadratic}  while stripping away any reference to complex interpolation; it amounts to an expository  repackaging of the same mechanism as our reasoning here, but is more elementary. Applying this proof to the vectors $y_1,\ldots,y_n\in X$ that satisfy~\eqref{eq:p/2 holder} and~\eqref{eq:average x y in proof} with $p=1$ and $\theta=0$ (namely, using the special case $\MM=Y=X$ and $\omega=\frac12$ of Proposition~\ref{prop:other exponents}), and then invoking duality through Theorem~\ref{thm:full duality}, we get an interpolation-free proof of Theorem~\ref{thm:average john}. By applying Lemma~\ref{lem:quote norm bound general} and Lemma~\ref{lem:reverse norm bound in UC}  in place of the linear-algebraic reasoning in~\cite{ANNRW18,Nao18}, one also obtains mutatis mutandis an interpolation-free  proof of the first part~\eqref{eq:rough D upper bound convexity smoothness} of Theorem~\ref{thm:really main}, albeit with a worse asymptotic dependence on $q$ in the implicit  factor in~\eqref{eq:rough D upper bound convexity smoothness}. We do not see how to derive the second part~\eqref{eq:D cases} of Theorem~\ref{thm:really main} without appealing to complex interpolation.  Incorporation of finite-dimensional reasoning in an interpolation argument, as we do here, is also used  in our subsequent works~\cite{ANNRW-FOCS18,ANNRW18-interpolation}; if  interpolation could be avoided in the context of~\cite{ANNRW-FOCS18,ANNRW18-interpolation}, then it would be worthwhile to do so, potentially (depending on the resulting proof)  with algorithmic ramifications.

\end{remark}

\section{Auxiliary embedding results}\label{sec:aux em}

Here we will prove Proposition~\eqref{prop:other exponents} and show how   Theorem~\ref{thm:lp version Kp} (matrix-dimension inequality with what we conjecture is the asymptotically optimal dependence on $p$) follows from Theorem~\ref{thm:average john}.

Henceforth, all balls in a metric space are closed, i.e., for a metric space $(\MM,d_\MM)$, a point $x\in \MM$ and a radius $r\in [0,\infty]$, we write $B_\MM(x,r)=\{y\in \MM:\ d_\MM(y,x)\le r\}$. Given a Borel probability measure $\mu$ on $\MM$ and $p\ge 1$, when in the Introduction we discussed the $p$-average distortion of an embedding of the metric probability space $(\MM,d_\MM,\mu)$ into some Banach space, we did not impose the  integrability requirement $\iint_{\MM\times\MM} d_\MM(x,y)^p\ud\mu(x)\ud\mu(y)<\infty$.  However, it is simple to dispose of the (inconsequential) case of those Borel probability measures $\mu$ on $\MM$ for which $d_\MM(\cdot,\cdot)\notin L_p(\mu\times \mu)$ through the following  straightforward consequence of the triangle inequality.

If $\iint_{\MM\times\MM} d_\MM(x,y)^p\ud\mu(x)\ud\mu(y)=\infty$, then for every $z\in \MM$ and $r>0$ we have
\begin{multline*}
\infty=\bigg(\iint_{\MM\times\MM} d_\MM(x,y)^p\ud\mu(x)\ud\mu(y)\bigg)^{\!\frac{1}{p}}\le \bigg(\iint_{\MM\times \MM} \big(d_\MM(x,z)+d_\MM(y,z)\big)^p\ud\mu(x)\ud\mu(y)\bigg)^{\!\frac{1}{p}}\\\le 2\bigg(\int_{\MM} d_\MM(x,z)^p\ud\mu(x)\bigg)^{\!\frac{1}{p}}\le  2\bigg((2r)^p+\int_{\MM\setminus B_\MM(z,2r)} d_\MM(x,z)^p\ud\mu(x)\bigg)^{\!\frac{1}{p}}.
\end{multline*}
So  $\int_{\MM\setminus B_\MM(z,2r)} d_\MM(x,z)^p\ud\mu(x)=\infty$. There is $r>0$ for which $\mu(B_\MM(z,r))>0$ ($\mu$ is  a probability measure). Noting that $d_\MM(x,z)-d_\MM(y,z)\ge \frac12d_\MM(x,z)$ when $(x,y)\in \big(\MM\setminus B_\MM(z,2r)\big)\times B_\MM(z,r)$,
\begin{multline*}
\iint_{\MM\times \MM} \big|d_\MM(x,z)-d_\MM(y,z)\big|^p\ud\mu(x)\ud\mu(y)\ge\frac12 \iint_{\big(\MM\setminus B_\MM(z,2r)\big)\times B_\MM(z,r)}d_\MM(x,z)^p\ud\mu(x)\ud\mu(y)\\
=\frac12\mu\big(B_\MM(z,r)\big)\int_{\MM\setminus B_\MM(z,2r)} d_\MM(x,z)^p\ud\mu(x)=\infty=\iint_{\MM\times\MM} d_\MM(x,y)^p\ud\mu(x)\ud\mu(y).
\end{multline*}
Thus the $1$-Lipschitz function $x\mapsto d_\MM(x,z)\in \R$ is an embedding of $(\MM,d_\MM,\mu)$ into the real line with $\mu$ average distortion $1$. Due to this (trivial) observation, we will be allowed to assume that we have $\iint_{\MM\times\MM} d_\MM(x,y)^p\ud\mu(x)\ud\mu(y)<\infty$ whenever needed in the ensuing discussion.

\subsection{Section~7.4 of~\cite{Nao14} revisited}\label{sec:revisited} The special case $\omega=1$ of Proposition~\eqref{prop:other exponents} was essentially proved in~\cite[Section~7.4]{Nao14}. Here we will derive this case of Proposition~\eqref{prop:other exponents}  while obtaining asymptotically better bounds than those of~\cite{Nao14} and also removing an additional hypothesis (on Lipschitz extendability) that arose in the  context of~\cite{Nao14} but is not needed for Proposition~\eqref{prop:other exponents} as stated here.

Let  $(\MM,d_\MM)$ be a metric space and fix~\cite{Ban32} an arbitrary isometric embedding $\jj:\MM\to C[0,1]$ of $\MM$ into the space of continuous functions on the interval $[0,1]$, equipped (as usual) with the supremum norm $\|\cdot\|_{C[0,1]}$; it is more convenient (but not crucial) to work below with such an embedding rather than the Fr\'echet embedding into $\ell_\infty$ due to the separability of the target space. Suppose that $q\ge 1$ and let $\mu$ be a Borel probability measure on $\MM$ such that
\begin{equation}\label{eq:q double integrability}
\iint_{\MM\times \MM} d_\MM(x,y)^q\ud\mu(x)\ud\mu(y)<\infty.
\end{equation}
By~\eqref{eq:q double integrability} and Fubini's theorem, $\int_\MM d_\MM(u,x)^q\ud\mu(x)<\infty$ for some  $u\in \MM$. Since $$\forall\, x\in \MM,\qquad \|\mathcal{j}(x)\|_{C[0,1]} \le \|\mathcal{j}(u)\|_{C[0,1]} +\|\mathcal{j}(x)-\mathcal{j}(u)\|_{C[0,1]} = \|\mathcal{j}(u)\|_{C[0,1]} +d_\MM(x,u),$$
we have $\int_\MM \|\jj(x)\|_{C[0,1]}^q\ud\mu(x)<\infty$. Hence $\int_\MM \|\jj(x)\|_{C[0,1]}\ud\mu(x)<\infty$, because $q\ge 1$ and $\mu$ is a probability measure. By Bochner's integrability criterion (see e.g.~\cite[Chapter~5]{BL00}), this implies that the Bochner integral $\int_\MM \mathcal{j}(x)\ud\mu(x)$ is a well-defined element of $C[0,1]$. We can therefore denote
\begin{equation}\label{eq:def Iq}
\bI_q=\bI_q(\mu,\jj)\eqdef \bigg(\int_\MM \Big\|\jj(x)-\int_\MM\jj(w)\ud\mu(w)\Big\|_{C[0,1]}^q\ud\mu(x)\bigg)^{\!\!\frac{1}{q}}.
\end{equation}
Observe that, using the fact that $\jj$ is isometry and the triangle inequality in $L_q(\mu\times \mu)$, we have
\begin{equation}\label{eq:Iq lower}
\bigg(\iint_{\MM\times \MM} d_\MM(x,y)^q\ud\mu(x)\ud\mu(y)\bigg)^{\!\!\frac{1}{q}}=\bigg(\iint_{\MM\times \MM} \|\jj(x)-\jj(y)\|_{C[0,1]}^q\ud\mu(x)\ud\mu(y)\bigg)^{\!\!\frac{1}{q}}\le 2\bI_q.
\end{equation}
Since $q\ge 1$ and $\mu$ is a probability measure, by Jensen's inequality and the fact that $\jj$ is isometry,
\begin{equation}\label{eq:Iq upper}
\bI_q\le \bigg(\iint_{\MM\times \MM} \|\jj(x)-\jj(w)\|_{C[0,1]}^q\ud\mu(x)\ud\mu(w)\bigg)^{\!\!\frac{1}{q}}=\bigg(\iint_{\MM\times \MM} d_\MM(x,y)^q\ud\mu(x)\ud\mu(y)\bigg)^{\!\!\frac{1}{q}}.
\end{equation}

In what follows, for $\tau\ge 1$ we will also consider a subset $A_\tau=A_\tau(\mu,\jj,q)\subset \MM$ that is defined by
\begin{align}\label{eq:phantom ball}
\begin{split}
A_\tau&\eqdef \bigg\{x\in \MM:\ \Big\|\jj(x)-\int_\MM \jj(w)\ud\mu(w)\Big\|_{C[0,1]}\le \tau \bI_q\bigg\}\\&=\jj^{-1}\bigg(B_{C[0,1]}\Big(\int_{\MM} \jj(w)\ud\mu(w),\tau\bI_q\Big)\bigg).
\end{split}
\end{align}
Note that by Markov's inequality we have
\begin{equation}\label{eq:Markov tau}
\forall\, \tau\ge 1,\qquad \mu(\MM\setminus A_\tau)\le \frac{1}{\tau^q}.
\end{equation}

The following lemma provides a convenient upper  bound on the $q$-average distortion of the metric probability space $(\MM,d_\MM,\mu)$ into the real line; in essence, its role in what follows is to treat the ``trivial case'' in which the random variable $\jj(x)$, where $x\in\MM$ is distributed according to $\mu$, is not well-concentrated around its mean in a certain  quantitative sense which is made precise below.

\begin{lemma}\label{lem:q into R} $(\MM,d_\MM,\mu)$ embeds with $q$-average distortion $D_\R\ge 1$ into $\R$, where
\begin{equation}\label{eq:def DR}
D_\R=D_\R(\mu,\jj,q)\eqdef  \inf_{\tau> e^{e^{-q}}}\bigg(\int_{\MM\setminus A_\tau}\Big\|\jj(x)-\int_\MM\jj(w)\ud\mu(w)\Big\|_{C[0,1]}^q\ud\mu(x)\bigg)^{\!\!-\frac{1}{q}}\frac{6\tau \bI_q}{\tau-e^{e^{-q}}} .
\end{equation}
\end{lemma}

\begin{proof} Define $f:\MM\to \R$ by setting
$$
\forall\, x\in \MM,\qquad f(x)\eqdef \Big\|\jj(x)-\int_\MM \jj(w)\ud\mu(w)\Big\|_{C[0,1]}.
$$
Then $f$ is $1$-Lipschitz, because $\jj$ is an isometry.

Suppose that $\tau>e^{e^{-q}}$ and observe that for every $x\in \MM\setminus A_\tau$ and $y\in A_{e^{e^{-q}}}$ we have
$$
f(x)-f(y)\ge \Big\|\jj(x)-\int_\MM \jj(w)\ud\mu(w)\Big\|_{C[0,1]}-e^{e^{-q}}\bI_q\ge \Big(1-\frac{e^{e^{-q}}}{\tau}\Big)\Big\|\jj(x)-\int_\MM \jj(w)\ud\mu(w)\Big\|_{C[0,1]}.
$$
Consequently,
\begin{align*}
&\iint_{\MM\times \MM} \big|f(x)-f(y)\big|^q\ud\mu(x)\ud\mu(y)\\
&\ge 2\Big(1-\frac{e^{e^{-q}}}{\tau}\Big)^{\!q}\mu\left(A_{e^{e^{-q}}}\right)\int_{\MM\setminus A_\tau}\Big\|\jj(x)-\int_\MM \jj(w)\ud\mu(w)\Big\|_{C[0,1]}^q\ud\mu(x)\\
&\ge \frac{\big(\tau-e^{e^{-q}}\big)^q\int_{\MM\setminus A_\tau}\left\|\jj(x)-\int_\MM \jj(w)\ud\mu(w)\right\|_{C[0,1]}^q\ud\mu(x)}{(6\tau\bI_q)^q}\iint_{\MM\times \MM} d_\MM(x,y)^q\ud\mu(x)\ud\mu(y),
\end{align*}
where the final step uses~\eqref{eq:Iq lower} and the bound~\eqref{eq:Markov tau} which gives $
\mu\left(A_{e^{e^{-q}}}\right) \ge 1-e^{-qe^{-q}}\ge 3^{-q}$.
\end{proof}

\begin{lemma}\label{lem:pass to larger q} Fix $p,q,D\ge 1$ with  $q\ge p$. Define $\Delta\ge 1$ by
\begin{equation}\label{eq:D' bigger q}
\Delta\eqdef  D+\frac{q}{p\log\big(e+\frac{q}{pD}\big)}\asymp \left\{ \begin{array}{ll}D &\mathrm{if\ } D\ge \frac{q}{p},\\
\frac{q}{p\log \left(e+\frac{q}{pD}\right)}&\mathrm{if\ } 1\le D\le \frac{q}{p}.\end{array}\right.
\end{equation}
Suppose that a metric probability space $(\MM,d_\MM,\mu)$ embeds with $p$-average distortion less than $D$ into a Banach space $(X,\|\cdot\|_X)$. Then $(\MM,d_\MM,\mu)$ embeds with $q$-average distortion $O(\Delta)$ into $(X,\|\cdot\|_X)$.
\end{lemma}

\begin{proof}  Let $D'\ge 1$ denote the infimum over those $K\ge 1$ for which $(\MM,d_\MM,\mu)$ embeds with $q$-average distortion $K$ into $(X,\|\cdot\|_X)$. Our task is to bound $D'$ from above. To this end, define $\d>0$ by
\begin{equation}\label{eq:d choice}
\d\eqdef \frac{\bI_p}{\bI_q}\stackrel{\eqref{eq:Iq lower}\wedge \eqref{eq:Iq upper}}{\le} 2\frac{\left(\iint_{\MM\times \MM} d_\MM(x,y)^p\ud\mu(x)\ud\mu(y)\right)^{\!\frac{1}{p}}}{\left(\iint_{\MM\times \MM} d_\MM(x,y)^q\ud\mu(x)\ud\mu(y)\right)^{\!\frac{1}{q}}}.
\end{equation}
Since $p< q$, by Jensen's inequality we have $\d\in [0,1]$. The premise of Lemma~\ref{lem:pass to larger q} is that there exists a $D$-Lipschitz mapping $f:\MM\to X$ that satisies
\begin{multline*}
\bigg(\iint_{\MM\times \MM} \|f(x)-f(y)\|_{\!X}^q\ud\mu(x)\ud\mu(y)\bigg)^{\!\!\frac{1}{q}}\ge \bigg(\iint_{\MM\times \MM} \|f(x)-f(y)\|_{\!X}^p\ud\mu(x)\ud\mu(y)\bigg)^{\!\!\frac{1}{p}}\\\ge \bigg(\iint_{\MM\times \MM} d_\MM(x,y)^p\ud\mu(x)\ud\mu(y)\bigg)^{\!\!\frac{1}{p}}\stackrel{\eqref{eq:d choice}}{\ge} \frac{\d}{2} \bigg(\iint_{\MM\times \MM} d_\MM(x,y)^q\ud\mu(x)\ud\mu(y)\bigg)^{\!\!\frac{1}{q}},
\end{multline*}
where the first step is Jensen's inequality. So, the normalized mapping $\frac{2}{\d}f:\MM\to X$ exhibits that
\begin{equation}\label{eq:D/d}
D'\le \frac{2}{\d}D.
\end{equation}
Note that Lemma~\ref{lem:pass to larger q}  already follows from~\eqref{eq:D/d} if $\d\ge \frac12$, so we may assume from now on that $\d\le \frac12$.

Next, fix $\tau>1$ satisfying
\begin{equation}\label{eq:tau conditions}
e^{e^{-q}}<\tau< \Big(\frac{1}{\d}\Big)^{\!\frac{p}{q-p}}.
\end{equation}
The value of $\tau$ will be specified later so as to optimize the ensuing reasoning; see~\eqref{eq:tau choice here}.  Observe that
\begin{multline*}
\!\!\!\!\!\int_{\MM\setminus A_\tau}\Big\|\jj(x)-\int_\MM\jj(w)\ud\mu(w)\Big\|_{C[0,1]}^q\ud\mu(x)\stackrel{\eqref{eq:def Iq}}{=}\bI_q^q-\int_{ A_\tau}\Big\|\jj(x)-\int_\MM\jj(w)\ud\mu(w)\Big\|_{C[0,1]}^q\ud\mu(x)\\
\stackrel{\eqref{eq:phantom ball}}{\ge} \bI_q^q-(\tau \bI_q)^{q-p}\int_{ A_\tau}\Big\|\jj(x)-\int_\MM\jj(w)\ud\mu(w)\Big\|_{C[0,1]}^p\ud\mu(x)\stackrel{\eqref{eq:def Iq}}{\ge} \bI_q^q-(\tau \bI_q)^{q-p}\bI_p^p\stackrel{\eqref{eq:d choice}}{=}(1-\tau^{q-p}\d^p)\bI_q^q.
\end{multline*}
In combination with Lemma~\ref{lem:q into R}, this implies (even for an embedding into $\R\subset X$) that
\begin{equation}\label{eq:D' upper to minimize tau}
D'\le \frac{6\tau}{(\tau-e^{e^{-q}})(1-\tau^{q-p}\d^p)^{\frac{1}{q}}}.
\end{equation}

It is obviously in our interest to choose $\tau$ so as to minimize the right had side of~\eqref{eq:D' upper to minimize tau} subject to the constraints~\eqref{eq:tau conditions}. While the optimal $\tau$ here does not have a closed-form expression, a straightforward (albeit somewhat tedious) inspection of~\eqref{eq:D' upper to minimize tau} reveals that up to a possible loss of a universal constant factor in the final conclusion~\eqref{eq:D' bigger q} of Lemma~\ref{lem:pass to larger q}, one cannot do better than the following choice.
\begin{equation}\label{eq:tau choice here}
\tau=\Big(1-\frac{1}{e^q}\Big)^{\!\frac{1}{q-p}}\Big(\frac{1}{\d}\Big)^{\!\frac{p}{q-p}}.
\end{equation}
For this value of $\tau$ one readily checks that~\eqref{eq:tau conditions} holds (recall that $0\le \d\le \frac12$). So, \eqref{eq:D' upper to minimize tau}  implies that
$$
D'\lesssim 1+\frac{q}{p\log\left(\frac{1}{\d}\right)}.
$$
In combination with~\eqref{eq:D/d} we therefore have
\begin{equation*}
D'\lesssim \max_{0\le \d\le \frac12} \min\left\{\frac{D}{\d},1+\frac{q}{p\log\left(\frac{1}{\d}\right)}\right\}\asymp D+\frac{q}{p\log\big(e+\frac{q}{pD}\big)},\tag*{\qedhere}
\end{equation*}
\end{proof}

\begin{lemma}\label{lem:pass to smaller q} Fix $p,q,D\ge 1$ with $p\ge q$. Suppose that an infinite  metric space $(\MM,d_\MM)$  embeds with $p$-average distortion less than $D$ into a Banach space $(X,\|\cdot\|_X)$. Then, $(\MM,d_\MM)$ embeds with $q$-average distortion $D'=D'(p,q,D)\ge 1$  into $(X,\|\cdot\|_X)$, where for some universal constant $\kappa>1$,
\begin{equation}\label{eq:power p/q}
D'\le (\kappa D)^{\frac{p}{q}}.
\end{equation}
\end{lemma}

\begin{proof} Fix a Borel probability measure $\mu$ on $\MM$ that satisfies $\iint_{\MM\times \MM} d_\MM(x,y)^p\ud\mu(x)\ud\mu(y)<\infty$.  For the purpose of proving Lemma~\ref{lem:pass to smaller q}, it will suffice to consider the subset $A_\tau\subset \MM$ in~\eqref{eq:phantom ball} only for  $\tau=8$. Note that due to Lemma~\ref{lem:q into R}, it suffices to prove  Lemma~\ref{lem:pass to smaller q} under the additional assumption
\begin{equation}\label{eq:less than 8}
\bigg(\int_{\MM\setminus A_8}\Big\|\jj(x)-\int_\MM\jj(w)\ud\mu(w)\Big\|_{C[0,1]}^q\ud\mu(x)\bigg)^{\!\!\frac{1}{q}}\le \frac{\bI_q}{8}.
\end{equation}
Since $\jj$ is an isometry, we have the following point-wise bound for every $x,y\in \MM$.
\begin{equation}\label{eq:phantom triangle}
 d_\MM(x,y)\le \|\jj(x)-\jj(y)\|_{C[0,1]}\le \Big\|x-\int_\MM \jj(w)\ud\mu(w)\Big\|_{C[0,1]}+\Big\|y-\int_\MM \jj(w)\ud\mu(w)\Big\|_{C[0,1]}.
\end{equation}
Using the triangle inequality in $L_q(\mu\times\mu)$, this implies that
\begin{align}\label{eq:1/4 on tail}
\begin{split}
\bigg(\iint_{(\MM\setminus A_8)\times \MM}&d_\MM(x,y)^q\ud\mu(x)\ud\mu(y)\bigg)^{\!\!\frac{1}{q}}\\&\stackrel{\eqref{eq:phantom triangle}\wedge \eqref{eq:def Iq}}{\le} \bigg(\int_{\MM\setminus A_8} \Big\|x-\int_\MM \jj(w)\ud\mu(w)\Big\|_{C[0,1]}^q\ud\mu(x)\bigg)^{\!\!\frac{1}{q}}+\mu(\MM\setminus A_8)^{\frac{1}{q}}\bI_q
\\&\stackrel{\eqref{eq:less than 8}\wedge \eqref{eq:Markov tau}}{\le}  \frac14\bigg(\iint_{\MM\times \MM} d_\MM(x,y)^q\ud\mu(x)\ud\mu(y)\bigg)^{\!\!\frac{1}{q}}.
\end{split}
\end{align}
Therefore,
\begin{align*}
\iint_{\MM\times \MM} &d_\MM(x,y)^q\ud\mu(x)\ud\mu(y)\\ &\stackrel{\eqref{eq:def Iq}}{=} \iint_{A_8\times A_8} d_\MM(x,y)^q\ud\mu(x)\ud\mu(y)+2\iint_{(\MM\setminus A_8)\times \MM}d_\MM(x,y)^q\ud\mu(x)\ud\mu(y)\\
&\stackrel{\eqref{eq:1/4 on tail}}{\le} \iint_{A_8\times A_8} d_\MM(x,y)^q\ud\mu(x)\ud\mu(y)+\frac12 \iint_{\MM\times \MM} d_\MM(x,y)^q\ud\mu(x)\ud\mu(y).
\end{align*}
This simplifies to give
\begin{equation}\label{eq:on A8 lower}
 \iint_{A_8\times A_8} d_\MM(x,y)^q\ud\mu(x)\ud\mu(y)\ge \frac12\iint_{\MM\times \MM} d_\MM(x,y)^q\ud\mu(x)\ud\mu(y).
\end{equation}

An application of the assumption of Lemma~\eqref{lem:pass to smaller q} to the restriction of $\mu$ to $A_8$ (recall that by~\eqref{eq:Markov tau} we have $\mu(A_8)\ge 1-8^{-q}\asymp 1$) yields a $D$-Lipschitz mapping $f:\MM\to X$ that satisfies
\begin{multline*}
\bigg(\frac{1}{\mu(A_8)^2}\iint_{A_8\times A_8} \|f(x)-f(y)\|_{\!X}^p\ud\mu(x)\ud\mu(y)\bigg)^{\!\!\frac{1}{p}}\ge \bigg(\frac{1}{\mu(A_8)^2}\iint_{A_8\times A_8} d_\MM(x,y)^p\ud\mu(x)\ud\mu(y)\bigg)^{\!\!\frac{1}{p}}\\ \ge \bigg(\frac{1}{\mu(A_8)^2}\iint_{A_8\times A_8} d_\MM(x,y)^q\ud\mu(x)\ud\mu(y)\bigg)^{\!\!\frac{1}{q}}\stackrel{\eqref{eq:on A8 lower}}{\ge} \bigg(\frac{1}{2\mu(A_8)^2}\iint_{\MM\times \MM} d_\MM(x,y)^q\ud\mu(x)\ud\mu(y)\bigg)^{\!\!\frac{1}{q}},
\end{multline*}
where the penultimate step is an application of Jensen's inequality, since $q<p$. Hence,
\begin{equation}\label{eq:A8 bigger than q}
\bigg(\iint_{A_8\times A_8} \|f(x)-f(y)\|_{\!X}^p\ud\mu(x)\ud\mu(y)\bigg)^{\!\!\frac{1}{p}}\gtrsim \bigg(\iint_{\MM\times \MM} d_\MM(x,y)^q\ud\mu(x)\ud\mu(y)\bigg)^{\!\!\frac{1}{q}}.
\end{equation}

Next, for every $x,y\in A_8$ we have
$$
d_\MM(x,y)=\|\jj(x)-\jj(y)\|_{C[0,1]}\le \Big\|x-\int_\MM \jj(w)\ud\mu(w)\Big\|_{C[0,1]}+\Big\|y-\int_\MM \jj(w)\ud\mu(w)\Big\|_{C[0,1]}\stackrel{\eqref{eq:phantom ball}}{\le} 16 \bI_q.
$$
Therefore, because $f$ is $D$-Lipschitz, the following point-wise inequality holds true.
$$
\forall\, x,y\in A_8,\qquad \|f(x)-f(y)\|_{\!X}^p\le \big(Dd_\MM(x,y)\big)^{p-q}\|f(x)-f(y)\|_{\!X}^q\le (16D\bI_q)^{p-q}\|f(x)-f(y)\|_{\!X}^q.
$$
Consequently,
\begin{align*}
\iint_{A_8\times A_8} \|f(x)-f(y)\|_{\!X}^p\ud\mu(x)\ud\mu(y)&\le (16D\bI_q)^{p-q}\iint_{A_8\times A_8} \|f(x)-f(y)\|_{\!X}^q\ud\mu(x)\ud\mu(y)\\
&\le (16D\bI_q)^{p-q}\iint_{\MM\times \MM} \|f(x)-f(y)\|_{\!X}^q\ud\mu(x)\ud\mu(y)\\&\!\stackrel{\eqref{eq:Iq upper}}{\le} (16D)^{p-q}\bigg(\iint_{\MM\times \MM} \|f(x)-f(y)\|_{\!X}^q\ud\mu(x)\ud\mu(y)\bigg)^{\!\!\frac{p}{q}}.
\end{align*}
A substitution of this bound into~\eqref{eq:A8 bigger than q} gives
\begin{equation*}
(16 D)^{\frac{p}{q}-1}\bigg(\iint_{\MM\times \MM} \|f(x)-f(y)\|_{\!X}^q\ud\mu(x)\ud\mu(y)\bigg)^{\!\!\frac{1}{q}}\gtrsim \bigg(\iint_{\MM\times \MM} d_\MM(x,y)^q\ud\mu(x)\ud\mu(y)\bigg)^{\!\!\frac{1}{q}}.
\end{equation*}
Thus, for an appropriate universal constant $C>0$, the rescaled function $C(16 D)^{\frac{p}{q}-1}f:\MM\to X$ exhibits the existence of an embedding with the stated bound on its $q$-average distortion.
\end{proof}

\begin{remark}\label{rem:additive needed?}
We did not investigate the optimality of Lemma~\ref{lem:pass to larger q}  and Lemma~\ref{lem:pass to smaller q}. Specifically, we do not know the extent to which the additive term in~\eqref{eq:D' bigger q} and the power $p/q$ in~\eqref{eq:power p/q} are necessary. It would be worthwhile (and probably tractable) to clarify these basic matters in future investigations.
\end{remark}

\subsection{Average embedding of a snowflake of a Banach space into itself}\label{sec:normalization} Note that Lemma~\ref{lem:pass to larger q}  and Lemma~\ref{lem:pass to smaller q} imply the special case $\omega=1$ of Proposition~\eqref{prop:other exponents}, with quite good dependence on $p,q$ in~\eqref{eq:our Delta in changed exponent}; in particular, $D'\lesssim D$ (essentially no loss is incured) when $q\ge p$ and $D\ge q/p$ (and $\omega=1$). We will next treat  Proposition~\eqref{prop:other exponents} for general $\omega\in (0,1]$ in the special case $\MM=Y$ and $D=1$.

For $\omega\in (0,1]$ and $p\in [1,2]$, the $\omega$-snowflake of $L_p(\R)$ embeds isometrically into $L_p(\R)$. The Hilbertian case $p=2$ of this statement is a classical theorem of Schoenberg~\cite{Sch38}, and this statement was proven for general $p\in [1,2]$ by Bretagnolle, Dacunha-Castelle and Krivine~\cite{BDK65}; see also the monograph~\cite{WW75} for an extensive treatment of this and related matters. Understanding the analogous situation when $p\in (2,\infty)$ remains a longstanding open question. Specifically, it is unknown whether or not there exists $\omega\in (0,1)$ and $p\in (2,\infty)$ such that the $\omega$-snowflake of $L_p(\R)$  admits a bi-Lipschitz embedding into $L_p(\R)$; see~\cite{MN04,AB15,Bau16,NS16-Xp,Nao16-riesz,EN18} for results along these lines, but an answer to this seemingly simple question remains stubbornly elusive despite substantial efforts. To the best of our knowledge, even the following more general question remains unknown.

\begin{question}\label{eq:general snowflake}
Does there exist $\omega\in (0,1)$ and an infinite dimensional Banach space $(X,\|\cdot\|_{\!X}^{\phantom{p}})$ whose $\omega$-snowflake does not admit a bi-Lipschitz embedding into  $(X,\|\cdot\|_{\!X}^{\phantom{p}})$?
\end{question}

Proposition~\ref{prop:average of snowflake back into X} below treats the easier variant of Question~\eqref{eq:general snowflake} in the  setting of average distortion.

\begin{proposition}\label{prop:average of snowflake back into X} Fix $p\in (0,\infty)$ and $\omega\in (0,1)$. Suppose that $D\in \R$ satisfies
\begin{equation}\label{eq:D assumption}
D>\frac{2^{(1-\omega)\left(1+\frac{1}{p\omega}\right)}}{\eta(p,\omega)},
\end{equation}
where $\eta(p,\omega)\in [0,1]$ is defined by
\begin{equation}\label{eq:def eta}
\eta(p,\omega)\eqdef \inf_{\sigma\in [0,1)}\frac{1-\sigma^\omega}{1-\sigma} (1+\sigma^{p\omega})^{\!\frac{1-\omega}{p\omega}}.
\end{equation}
Then, for any Banach space $X$, the $\omega$-snowflake of $X$ embeds with $p$-average distortion $D$ into $X$.
\end{proposition}

Below we will provide estimates on the quantity $\eta(p,\omega)$ in~\eqref{eq:def eta}, but the main  significance  of Proposition~\ref{prop:average of snowflake back into X} is that the $p$-average distortion $D$ in~\eqref{eq:D assumption} can be taken to be a finite quantity that depends only on $p$ and $\omega$. With this at hand, we will now complete the proof of Proposition~\eqref{prop:other exponents}.

\begin{proof}[Proof of Proposition~\ref{prop:other exponents} assuming Proposition~\ref{prop:average of snowflake back into X}] Denote
\begin{equation}\label{eq:choose beta omega max}
\beta\eqdef \max\{q\omega,1\}.
\end{equation}
The assumption of Proposition~\eqref{prop:other exponents} is that an infinite metric space  $(\MM,d_\MM)$ embeds with $p$-average distortion $D\ge 1$ into a Banach space  $(Y,\|\cdot\|_Y)$. By Lemma~\ref{lem:pass to larger q}  and Lemma~\ref{lem:pass to smaller q}  there is a universal constant $\alpha>1$ such that  $(\MM,d_\MM)$ embeds with $\beta$-average less than $D_1>1$ in $(Y,\|\cdot\|_Y)$, where
\begin{equation}\label{eq:def D1}
D_1\eqdef  \bigg(\alpha D+\frac{\alpha\beta}{p\log\big(e+\frac{\beta}{pD}\big)}\bigg)^{\!\!\max\left\{\frac{p}{\beta},1\right\}}.
\end{equation}
Thus, for every Borel probability measure $\mu$ on $\MM$ there exists a mapping $f:\MM\to Y$ that satisfies
\begin{equation}\label{eq:D1 lip}
\forall\, x,y\in \MM,\qquad \|f(x)-f(y)\|_Y\le D_1d_\MM(x,y),
\end{equation}
and
\begin{equation}\label{eq:beta lower average}
\iint_{\MM\times \MM} \|f(x)-f(y)\|_Y^\beta\ud\mu(x)\ud\mu(y)\ge \iint_{\MM\times \MM} d_\MM(x,y)^\beta\ud\mu(x)\ud\mu(y).
\end{equation}

An application of Proposition~\eqref{prop:average of snowflake back into X} to the probability measure $f_\sharp \mu$ on $Y$ (the push-forward of $\mu$ under $f$) yields a mapping $g:Y\to Y$ that satisfies
\begin{equation}\label{eq:D2}
\forall\, u,v\in Y,\qquad \|g(u)-g(v)\|_Y\le D_2\|u-v\|_Y^\omega,\qquad\mathrm{where}\qquad D_2\eqdef \frac{2^{1+\frac{1-\omega}{\beta\omega}}}{\eta\big(\frac{\beta}{\omega},\omega\big)},
\end{equation}
and
\begin{equation}\label{eq:beta over omega average}
\iint_{\MM\times \MM} \|g\circ f(x)-g\circ (y)\|_Y^{\frac{\beta}{\omega}}\ud\mu(x)\ud\mu(y)\ge \iint_{\MM\times \MM} \|f(x)-f(y)\|_Y^\beta\ud\mu(x)\ud\mu(y).
\end{equation}

Therefore,
$$
\forall\, x,y\in \MM,\qquad \|g\circ f(x)-g\circ f(y)\|_Y\stackrel{\eqref{eq:D1 lip}\wedge \eqref{eq:D2}}{\le} D_2D_1^\omega d_\MM(x,y)^\omega,
$$
and
$$
\iint_{\MM\times \MM} \|g\circ f(x)-g\circ (y)\|_Y^{\frac{\beta}{\omega}}\ud\mu(x)\ud\mu(y)\stackrel{\eqref{eq:beta lower average}\wedge \eqref{eq:beta over omega average}}{\ge} \iint_{\MM\times \MM}d_\MM(x,y)^\beta\ud\mu(x)\ud\mu(y)
$$
This is the same as saying that the metric space $(\MM,d_\MM^\omega)$ embeds with $(\beta/\omega)$-average distortion $D_2D_1^\omega$ into $(Y,\|\cdot\|_Y)$. Recalling the definition~\eqref{eq:choose beta omega max} of $\beta$, we have $\beta/\omega\ge q$. Hence, by Lemma~\ref{lem:pass to smaller q} the $\omega$-snowflake of $(\MM,d_\MM)$ embeds with $q$-average distortion $D'\ge 1$ into $(Y,\|\cdot\|_Y)$, where
\begin{equation*}
D'\le\big(\kappa D_2D_1^\omega\big)^{\!\frac{\beta}{q\omega}}.\qedhere
\end{equation*}
\end{proof}

Having noted the validity of Proposition~\ref{prop:average of snowflake back into X}, the following open question arises naturally.

\begin{question}\label{Q:D0}
Does there exist a universal constant $D$ with the property that for every $p\ge 1$ and $\omega\in (0,1)$, the $\omega$-snowflake of any Banach space $X$ embeds with $p$-average distortion $D$ into $X$? We do not know if this is so even in the special cases of greatest interest $p\in \{1,2\}$ and $\omega\to 0^+$.
\end{question}

Prior to proving Proposition~\ref{prop:average of snowflake back into X}, we will collect some elementary estimates on the quantity $\eta(p,\omega)$ that is defined in~\eqref{eq:def eta}. While explicit bounds on $\eta(p,\omega)$ are not crucial for the main geometric consequences of the present work, it is worthwhile to record such estimates here due to the intrinsic geometric interest of such embeddings;    we also expect that explicit bounds will be needed for future applications. Lemma~\ref{lem bounds on eta}
below reflects the fact that the function of $\sigma\in [0,1)$ in~\eqref{eq:def eta}  whose infimum defines $\eta(p,\omega)$ is decreasing when $p\omega\ge 1$ and increasing when $p\le 1$. In the remaining range $p\in (1,\frac{1}{\omega})$, the function in question can behave in a more complicated manner and in particular in parts of this range it attains  its global minimum in the interior of the interval $[0,1]$.
\begin{lemma}\label{lem bounds on eta} Fix $\omega\in (0,1)$. If $p\ge \frac{1}{\omega}$, then $\eta(p,\omega)=\omega2^{\frac{1-\omega}{p\omega}}$. If $0<p\le 1$, then $\eta(p,\omega)=1$.
\end{lemma}

\begin{proof}
Define $h:[0,1]\times (0,\infty)\times (0,1)\to \R$ by setting for every $(\sigma,p,\omega)\in [0,1]\times (0,\infty)\times (0,1)$,
$$
 h(\sigma,p,\omega)\eqdef \log\left(\frac{1-\sigma^\omega}{1-\sigma} (1+\sigma^{p\omega})^{\frac{1-\omega}{p\omega}}\right) =\log(1-\sigma^\omega)-\log(1-\sigma)+\frac{1-\omega}{p\omega}\log\left(1+\sigma^{p\omega}\right),
$$
with the endpoint convention $h(1,p,\omega)\eqdef\lim_{\sigma\to 1} h(\sigma,p,\omega)= \log\left(\omega2^{\frac{1-\omega}{p\omega}}\right)$. Then,
\begin{equation}\label{eq:did h}
\frac{\partial h}{\partial \sigma}(\sigma,p,\omega)=\frac{1}{1-\sigma}-\frac{\omega}{\sigma^{1-\omega}(1-\sigma^\omega)}+\frac{1-\omega}{\sigma}\cdot \frac{\sigma^{p\omega}}{1+\sigma^{p\omega}}.
\end{equation}
Since the mapping $u\mapsto \frac{u}{1+u}$ is increasing on $[0,\infty)$, for every fixed $\sigma,\omega\in (0,1)$ the right hand side of~\eqref{eq:did h} is a decreasing function of $p$. Consequently, for every $\sigma,\omega\in (0,1)$ and $p\ge \frac{1}{\omega}$,
\begin{align}\label{eq:integral identity 1}
\begin{split}
\frac{\partial h}{\partial \sigma}(\sigma,p,\omega)&\le \frac{\partial h}{\partial \sigma}\Big(\sigma,\frac{1}{\omega},\omega\Big)\stackrel{\eqref{eq:did h}}{=}\frac{\omega\sigma^{2-\omega}-(2-\omega)\sigma+(2-\omega)\sigma^{1-\omega}-\omega}{(1-\sigma^2)(1-\sigma^\omega)\sigma^{1-\omega}}
\\ &=-\frac{\omega(1-\omega)(2-\omega)}{{(1-\sigma^2)(1-\sigma^\omega)\sigma^{1-\omega}}}\int_\sigma^1 \frac{(1-s)(s-\sigma)}{s^{\omega+1}}\ud s<0,
\end{split}
\end{align}
where the penultimate step of~\eqref{eq:integral identity 1} is a straightforward evaluation of the definite integral. Therefore $h(\sigma,p,\omega)$ is decreasing in $\sigma$ if $p\omega\ge 1$. In particular, in this range we have $h(\sigma,p,\omega)\ge h(1,p,\omega)$, i.e., the maximum defining $\eta(p,\omega)$ in~\eqref{eq:def eta} is attained at $\sigma=1$, as required. Also, if $p\in (0,1]$, then
\begin{align}\label{eq:integral identity 2}
\begin{split}
\frac{\partial h}{\partial \sigma}(\sigma,p,\omega)&\ge \frac{\partial h}{\partial \sigma}\Big(\sigma,1,\omega\Big)\stackrel{\eqref{eq:did h}}{=}\frac{\sigma^{1-\omega}-\sigma^\omega+(2\omega-1)\sigma+1-2\omega}{(1-\sigma)(1-\sigma^{2\omega})\sigma^{1-\omega}}
\\ &=\frac{\omega}{(1-\sigma)(1-\sigma^{2\omega})\sigma^{1-\omega}}\int_\sigma^1 s^\omega(s-\sigma)\Big(2(1-\omega)+\frac{1}{s^{1+\omega}}-1\Big)\ud s>0,
\end{split}
\end{align}
where the penultimate step of~\eqref{eq:integral identity 2} is a straightforward evaluation of the definite integral and the final step of~\eqref{eq:integral identity 2} holds because the integrand is point-wise positive. Hence, $h(\sigma,p,\omega)$ is increasing in $\sigma$ if $p\in (0,1]$. In particular, $h(\sigma,p,\omega)\ge h(0,p,\omega)=0$ when $p\in (0,1]$, i.e., in this range the maximum defining $\eta(p,\omega)$ in~\eqref{eq:def eta} is attained at $\sigma=0$, as required.
\end{proof}

The following corollary is nothing more than a substitution of Lemma~\ref{lem bounds on eta} into Proposition~\ref{prop:average of snowflake back into X}.

\begin{corollary}\label{cor:specific bounds} For  $p>0$ and $\omega\in (0,1]$, let $\mathsf{Av}(p,\omega)$ be the infimum over  those $D\ge 1$ such that for any Banach space $X$ the $\omega$-snowflake of  $X$ embeds with $p$-average distortion $D$ into $X$. Then
\begin{equation*}
\mathsf{Av}(1,\omega)\le 2^{\frac{1-\omega^2}{\omega}}\qquad\mathrm{and}\qquad  p\omega\ge 1\implies \mathsf{Av}(p,\omega)\le\frac{2^{1-\omega}}{\omega}.
\end{equation*}
In particular, both $\mathsf{Av}\big(1,\frac12\big)$ and $\mathsf{Av}\big(2,\frac12\big)$ are at most $2\sqrt{2}$ and $\mathsf{Av}\big(p,\frac{1}{p}\big)\le 2p$ for all $p\ge 1$.
\end{corollary}

\begin{remark} We proved above that the distortion $D'$ of Proposition~\eqref{prop:other exponents} satisfies $D'\le\big(\kappa D_2D_1^\omega\big)^{\!\frac{\beta}{q\omega}}$. Here $\kappa>1$ is a universal constant and $\beta, D_1,D_2$ are given in~\eqref{eq:choose beta omega max}, \eqref{eq:def D1}, \eqref{eq:D2}, respectively. So, using Lemma~\ref{lem bounds on eta}, we have the following version of~\eqref{eq:our Delta in changed exponent}, in which $\mathsf{K}>1$ is a universal constant.
\begin{equation}\label{eq:changed distortion explicit}
D'\lesssim \frac{\mathsf{K}^{\frac{p}{q}+\frac{1}{q\omega}}}{\omega^{\max\left\{1,\frac{1}{q\omega}\right\}}}
\bigg(D+\frac{q\omega}{p\log\big(e+\frac{q\omega}{pD}\big)}\bigg)^{\!\!\max\left\{\frac{p}{q},\omega\right\}}.
\end{equation}
We have no reason to suspect that~\eqref{eq:changed distortion explicit} is sharp; it would be worthwhile to find the optimal bound.
\end{remark}

If one uses~\eqref{eq:changed distortion explicit} in the setting of Theorem~\ref{thm:average john}, in which $\MM$ is the $\frac12$-snowflake of a $k$-dimensional normed space $X$ and $D\lesssim \sqrt{\log k}$, one see that for every $\omega\in (0,\frac12]$ the $\omega$-snowflake of $X$ embeds with quadratic average distortion  at most $\mathsf{C}_\omega\sqrt{\log k}$ into a Hilbert space, where $\mathsf{C}_\omega>0$ depends only on $\omega$.  We suspect that the power of the logarithm may not be sharp here, thus leading to the following conjecture whose investigation we postpone to future research; we will see in Section~\ref{sec:impossibility} that its positive resolution would yield an asymptotically sharp bound (for $\omega$ fixed and $k\to \infty$).

\begin{conjecture}\label{conj:power of log} For every $\omega\in (0,\frac12)$ there is $\mathsf{C}_\omega>0$ such that for $k\in \{2,3,\ldots\}$ the $\omega$-snowflake of any $k$-dimensional normed space embeds with quadratic average distortion $C_\omega(\log k)^\omega$ into $\ell_2$.
\end{conjecture}

If in the formulation of Conjecture~\ref{conj:power of log} quadratic average distortion is replaced by $q$-average distortion for $q>2$, then by~\eqref{eq:changed distortion explicit} the asymptotics of the distortion decreases to $\mathsf{C}_{\omega,q}(\log k)^{\max\{2/q,\omega\}}$. So, the  analogue of Conjecture~\ref{conj:power of log}  for $(2/\omega)$-average distortion has a positive answer. For example, if, say, one considers $4$-average distortion of the $\frac14$-snowflake, then the bound becomes of order $\sqrt[4]{\log k}$.

Note also that if in Theorem~\ref{thm:average john} one considers average distortion (i.e., $1$-average distortion) in place of its quadratic counterpart, then by~\eqref{eq:changed distortion explicit} we get that the $\frac12$-snowflake of any $k$-dimensional norm embeds into a Hilbert space with average distortion $O(\log k)$. It is conceivable that this bound could be reduced to $O(\sqrt{\log k})$, but we did not investigate this matter yet.

The proof of Proposition~\eqref{prop:average of snowflake back into X} relies on a natural ``fractional normalization map'' for which sharp bounds are contained in Lemma~\eqref{lem:sharp holder bounds on normalization} below; cruder  estimates on the modulus of continuity of such maps appear in several places, but we could not locate their optimal from in the literature.

\begin{lemma}\label{lem:sharp holder bounds on normalization} Fix $\omega\in (0,1)$. For a Banach space $(X,\|\cdot\|_{\!X}^{\phantom{p}})$, define $f_\omega=f_\omega^X:X\to X$ by setting
\begin{equation}\label{eq:def f omega}
\forall\,x\in X\setminus \{0\},\qquad f_\omega(x)\eqdef \frac{1}{\|x\|_{\!X}^{1-\omega}}x,
\end{equation}
and $f_\omega(0)\eqdef 0$.  Then, for every $p\in (0,\infty)$ we have
\begin{equation}\label{eq:holder bounds f omega}
\forall\, x,y\in X,\qquad \frac{\eta(p,\omega)\|x-y\|_{\!X}^{\phantom{p}}}{\left(\|x\|_{\!X}^{p\omega}+\|y\|_{\!X}^{p\omega}\right)^{\!\!\frac{1-\omega}{p\omega}}}\le \|f_\omega(x)-f_\omega(y)\|_{\!X}^{\phantom{p}}\le 2^{1-\omega}\|x-y\|_{\!X}^{\omega\phantom{p}}.
\end{equation}
Both of the constants $\eta(p,\omega)$ and $2^{1-\omega}$ in the two inequalities appearing in~\eqref{eq:holder bounds f omega} cannot be improved.
\end{lemma}

Note that if $p\ge 1$ and $X$ is an $L_p(\mu)$ space, and we apply the mapping in~\eqref{eq:def f omega} point-wise, then we get the mapping $(\phi\in L_p(\mu))\mapsto \sign(\phi)|\phi|^\omega$, which is the classical Mazur map~\cite{Maz29}  from $L_p(\mu)$ to $L_q(\mu)$ for $q=p/\omega$. However,  $f_\omega(\phi)=\|\phi\|_{L_p(\mu)}^{\omega-1}\phi$, so $f_\omega$ itself is different  from the Mazur map.

\begin{proof}[Proof of Lemma~\ref{lem:sharp holder bounds on normalization}] The optimality of the first inequality in~\eqref{eq:holder bounds f omega} is seen by considering $x=\sigma y$ for every $\sigma\in [0,1]$, and the optimality  of the second inequality in~\eqref{eq:holder bounds f omega} is seen by considering $x=-y$.

Suppose that $x,y\in X\setminus \{0\}$ satisfy $\|x\|_{\!X}^{\phantom{p}}< \|y\|_{\!X}^{\phantom{p}}$. Then,
\begin{align}\label{eq:for lower f omega}
\begin{split}
\|f_\omega(y)&-f_\omega(x)\|_{\!X}^{\phantom{p}}\stackrel{\eqref{eq:def f omega}}{=}\bigg\|\frac{1}{\|y\|_{\!X}^{1-\omega}}(y-x)-
\bigg(\frac{1}{\|x\|_{\!X}^{1-\omega}}-\frac{1}{\|y\|_{\!X}^{1-\omega}}\bigg)x\bigg\|_{\!X}^{\phantom{p}}\\ &\ge \frac{\|y-x\|_{\!X}^{\phantom{p}}}{\|y\|_{\!X}^{1-\omega}}-\|x\|_{\!X}^{\omega}+\frac{\|x\|_{\!X}^{\phantom{p}}}{\|y\|_{\!X}^{1-\omega}}=
\frac{1}{\|y\|_{\!X}^{1-\omega}}\bigg(1-\frac{\|x\|_{\!X}^{\omega}\|y\|_{\!X}^{1-\omega}-\|x\|_{\!X}^{\phantom{p}}}
{\|x-y\|_{\!X}^{\phantom{p}}}\bigg)\|x-y||_{\!X}^{\phantom{p}}.
\end{split}
\end{align}
Observe that $(\|x\|_{\!X}^\omega\|y\|_{\!X}^{1-\omega}-\|x\|_{\!X}^{\phantom{p}})/\|x-y\|_{\!X}^{\phantom{p}}\le (\|x\|_{\!X}^\omega\|y\|_{\!X}^{1-\omega}-\|x\|_{\!X}^{\phantom{p}})/(\|y\|_{\!X}^{\phantom{p}}-\|x\|_{\!X}^{\phantom{p}})$,  because $\|x\|_{\!X}^\omega\|y\|_{\!X}^{1-\omega}-\|x\|_{\!X}^{\phantom{p}}> 0$ and $\|x-y\|_{\!X}^{\phantom{p}}\ge \|y\|_{\!X}^{\phantom{p}}-\|x\|_{\!X}^{\phantom{p}}$. By substituting this into~\eqref{eq:for lower f omega}, we see that
\begin{multline*}
\|f_\omega(y)-f_\omega(x)\|_{\!X}^{\phantom{p}}\ge \frac{1}{\|y\|_{\!X}^{1-\omega}}\bigg(1-\frac{\|x\|_{\!X}^\omega\|y\|_{\!X}^{1-\omega}-\|x\|_{\!X}^{\phantom{p}}}{\|y\|_{\!X}^{\phantom{p}}-
\|x\|_{\!X}^{\phantom{p}}}\bigg)\|x-y||_{\!X}^{\phantom{p}}
\\=
\frac{1-\sigma^\omega}{1-\sigma} (1+\sigma^{p\omega})^{\!\!\frac{1-\omega}{p\omega}}\cdot \frac{\|x-y\|_{\!X}^{\phantom{p}}}{\left(\|x\|_{\!X}^{p\omega}+\|y\|_{\!X}^{p\omega}\right)^{\!\!\frac{1-\omega}{p\omega}}}\stackrel{\eqref{eq:def eta}}{\ge}\frac{\eta(p,\omega)\|x-y\|_{\!X}^{\phantom{p}}}{\left(\|x\|_{\!X}^{p\omega}+\|y\|_{\!X}^{p\omega}\right)^{\!\!\frac{1-\omega}{p\omega}}},
\end{multline*}
where in the penultimate step we write  $\sigma\eqdef\frac{\|x\|_{\!X}^{\phantom{p}}}{\|y\|_{\!X}^{\phantom{p}}}\le 1$. This justifies the first inequality in~\eqref{eq:holder bounds f omega}.

For the second inequality in~\eqref{eq:holder bounds f omega}, note that
\begin{align}\label{eq:for cases on x}
\begin{split}
\left\|f_\omega(x)-f_\omega(y)\right\|_{\!X}^{\phantom{p}}&=\bigg\|\bigg(\frac{1}{\|x\|_{\!X}^{1-\omega}}-\frac{1}{\|y\|_{\!X}^{1-\omega}}\bigg)x
+\frac{1}{\|y\|_{\!X}^{1-\omega}}(x-y)\bigg\|_{\!X}^{\phantom{p}}\\
&\le\bigg(\frac{1}{\|x\|_{\!X}^{1-\omega}}-\frac{1}{\|y\|_{\!X}^{1-\omega}}\bigg)\|x\|_{\!X}^{\phantom{p}}+\frac{\|x-y\|_{\!X}^{\phantom{p}}}
{\|y\|_{\!X}^{1-\omega}}
=\frac{\|x-y\|_{\!X}^{\phantom{p}}-\|x\|_{\!X}^{\phantom{p}}}{\|y\|_{\!X}^{1-\omega}}+\|x\|_{\!X}^{\omega\phantom{p}}.
\end{split}
\end{align}
The quantity  $\f(\|y\|_{\!X}^{\phantom{p}})\eqdef (\|x-y\|_{\!X}^{\phantom{p}}-\|x\|_{\!X}^{\phantom{p}})/\|y\|_{\!X}^{1-\omega}$ in~\eqref{eq:for cases on x} decreases with $\|y\|_{\!X}^{\phantom{p}}$ if $\|x\|_{\!X}^{\phantom{p}}< \|x-y\|_{\!X}^{\phantom{p}}$. Since  $\|y\|_{\!X}^{\phantom{p}}\ge \|x-y\|_{\!X}^{\phantom{p}}-\|x\|_{\!X}^{\phantom{p}}$ is a better lower bound on $\|y\|_{\!X}^{\phantom{p}}$ than our assumption $\|y\|_{\!X}^{\phantom{p}}\ge \|x\|_{\!X}^{\phantom{p}}$ when $\|x\|_{\!X}^{\phantom{p}}\le \frac12\|x-y\|_{\!X}^{\phantom{p}}$, it follows that $\f(\|y\|_{\!X}^{\phantom{p}})\le \f(\|x-y\|_{\!X}^{\phantom{p}}-\|x\|_{\!X}^{\phantom{p}})$ if $\|x\|_{\!X}^{\phantom{p}}\le \frac12\|x-y\|_{\!X}^{\phantom{p}}$ and $\f(\|y\|_{\!X}^{\phantom{p}})\le \f(\|x\|_{\!X}^{\phantom{p}})$ if $\frac12\|x-y\|_{\!X}^{\phantom{p}}\le \|x\|_{\!X}^{\phantom{p}}\le \|x-y\|_{\!X}^{\phantom{p}}$. Next, $\f(\|y\|_{\!X}^{\phantom{p}})$ is nondecreasing in $\|y\|_{\!X}^{\phantom{p}}$ when $\|x\|_{\!X}^{\phantom{p}}\ge \|x-y\|_{\!X}^{\phantom{p}}$, so due to the a priori upper bound $\|y\|_{\!X}^{\phantom{p}}\le \|x-y\|_{\!X}^{\phantom{p}}+\|x\|_{\!X}^{\phantom{p}}$ we have $\f(\|y\|_{\!X}^{\phantom{p}})\le \f(\|x-y\|_{\!X}^{\phantom{p}}+\|x\|_{\!X}^{\phantom{p}})$ in the remaining range $\|x\|_{\!X}^{\phantom{p}}\ge \|x-y\|_{\!X}^{\phantom{p}}$. These observations give
\begin{equation}\label{eq:enter psi}
 \left\|f_\omega(x)-f_\omega(y)\right\|_{\!X}^{\phantom{p}}\stackrel{\eqref{eq:for cases on x}}{\le}\psi_\omega\bigg(\frac{\|x\|_{\!X}^{\phantom{p}}}{\|x-y\|_{\!X}^{\phantom{p}}}\bigg)\|x-y\|_{\!X}^{\omega\phantom{p}},
\end{equation}
where $\psi_\omega:[0,\infty)\to [0,\infty)$ is defined by
\begin{equation}\label{eq:def psi omega}
\forall\, \rho\in [0,\infty),\qquad \psi_\omega(\rho)\eqdef \left\{\begin{array}{ll} (1-\rho)^\omega+\rho^\omega&\mathrm{if}\ 0\le \rho\le \frac12,\\
\frac{1}{\rho^{1-\omega}}&\mathrm{if}\ \frac12\le \rho\le 1,\\
\rho^\omega-\frac{\rho-1}{(1+\rho)^{1-\omega}}&\mathrm{if}\ \rho\ge 1.\end{array}\right.
\end{equation}
For $\rho\in (0,\frac12)$ we have $\psi_\omega'(\rho)=\omega(1/\rho^{1-\omega}-1/(1-\rho)^{1-\omega})> 0$. Also, $\psi'(\rho)=-(1-\omega)/\rho^{2-\omega}<0$ for  $\rho\in (\frac12,1)$. Finally, we claim that $\psi_\omega'(\rho)<0$  if $\rho> 1$. Indeed, for every $\rho>1$,
\begin{multline}\label{eq:analyse last case rho}
 \psi_\omega'(\rho)\stackrel{\eqref{eq:def psi omega}}{=}
\frac{\omega(1+\rho)^{2-\omega}-\omega\rho^{2-\omega}-(2-\omega)\rho^{1-\omega}}{\rho^{1-\omega}(1+\rho)^{2-\omega}}= \frac{(1-\omega)(2-\omega)}{(1+\rho)^{2-\omega}}\int_\rho^{1+\rho}\bigg(\omega\int_1^{\frac{r}{\rho}}\frac{\ud s}{s^\omega}-1\bigg) \ud r \\\le  \frac{(1-\omega)(2-\omega)}{(1+\rho)^{2-\omega}}\int_\rho^{1+\rho}\bigg(\omega\Big(\frac{r}{\rho}-1\Big)-1\bigg) \ud r =-\frac{(1-\omega)(2-\omega)}{(1+\rho)^{2-\omega}}\Big(1-\frac{\omega}{2\rho}\Big)<0,
\end{multline}
where the identities in the second and penultimate steps of~\eqref{eq:analyse last case rho} are straightforward evaluations of the respective definite integrals, the third step of~\eqref{eq:analyse last case rho} uses the fact that $s\ge 1$ in the internal integrand, and the final step of~\eqref{eq:analyse last case rho}  holds because $0<\omega<1<\rho$. We have thus established that $\psi_\omega$ is increasing on $[0,\frac12]$ and decreasing on $[\frac12,\infty)$. Hence, $\psi_\omega$ attains its global maximum at $\rho=\frac12$, where its value is $2^{1-\omega}$ . Due to~\eqref{eq:enter psi}, this justifies the second inequality in~\eqref{eq:holder bounds f omega}.
\end{proof}

\begin{proof}[Proof of Proposition~\ref{prop:average of snowflake back into X}] Fix a Borel probability measure $\mu$ on $X$. The assumption~\eqref{eq:D assumption} implies that
\begin{equation}\label{eq:def beta D}
\beta\eqdef \frac12\left(\frac{D\eta(p,\omega)}{2^{1-\omega}}\right)^{\!\!\frac{p\omega}{1-\omega}}>1.
\end{equation}
Hence, there exists $u\in X$ such that
\begin{equation}\label{eq:translate to 0}
\int_X \|x-u\|_{\!X}^{p\omega}\ud \mu(x)\le \beta \inf_{v\in X}\int_{X}\|x-v\|_{\!X}^{p\omega}\ud\mu(x)\le \beta \iint_{X\times X}\|x-y\|_{\!X}^{p\omega}\ud\mu(x)\ud\mu(y).
\end{equation}

Define $\f:X\to X$ by setting
\begin{equation}\label{eq:normalized phi}
\forall\, z\in X,\qquad \f(z)\eqdef  \bigg(\frac{\iint_{X\times X} \|x-y\|_{\!X}^{p\omega} \ud\mu(x)\ud\mu(y)}{\iint_{X\times X} \|f_\omega(x-u)-f_\omega(y-u)\|_{\!X}^{p} \ud\mu(x)\ud\mu(y)}\bigg)^{\!\!\frac{1}{p}}f_\omega(z-u),
\end{equation}
where $f_\omega$ is the normalization map that is given in~\eqref{eq:def f omega}. Then, by design we have
$$
\iint_{X\times X} \|\f(x)-\f(y)\|_{\!X}^{p} \ud\mu(x)\ud\mu(y)=\iint_{X\times X} \|x-y\|_{\!X}^{p\omega} \ud\mu(x)\ud\mu(y).
$$
Proposition~\ref{prop:average of snowflake back into X} will therefore be proven if we demonstrate that the $\omega$-H\"older constant of $\f:X\to X$ satisfies $\|\f\|_{\!\Lip_\omega(X,X)}^{\phantom{p}}\le D$. Indeed,
\begin{align}
\begin{split}
\|\f&\|_{\!\Lip_\omega(X,X)}^{\phantom{p}}
\stackrel{\ \eqref{eq:normalized phi}\wedge \eqref{eq:holder bounds f omega}}{\ =} 2^{1-\omega}\bigg(\frac{\iint_{X\times X} \|x-y\|_{\!X}^{p\omega} \ud\mu(x)\ud\mu(y)}{\iint_{X\times X} \|f_\omega(x-u)-f_\omega(y-u)\|_{\!X}^{p} \ud\mu(x)\ud\mu(y)}\bigg)^{\!\!\frac{1}{p}}\label{eq:compute L}
\end{split}\\ \nonumber&\stackrel{\eqref{eq:holder bounds f omega}}{\le} \frac{2^{1-\omega}}{\eta(p,\omega)^\omega} \bigg(\frac{\iint_{X\times X} \left(\|x-u\|_{\!X}^{p\omega}+\|y-u\|_{\!X}^{p\omega}\right)^{1-\omega}\|f_\omega(x-u)-f_\omega(y-u)\|_{\!X}^{p\omega}
\ud\mu(x)\ud\mu(y)}{\iint_{X\times X} \|f_\omega(x-u)-f_\omega(y-u)\|_{\!X}^{p} \ud\mu(x)\ud\mu(y)}\bigg)^{\!\!\frac{1}{p}}\\
&\le \label{eq:here we use holder} \frac{2^{1-\omega}}{\eta(p,\omega)^\omega}\bigg(\frac{\iint_{X\times X} \left(\|x-u\|_X^{p\omega}+\|y-u\|_{\!X}^{p\omega}\right)
\ud\mu(x)\ud\mu(y)}{\iint_{X\times X} \|f_\omega(x-u)-f_\omega(y-u)\|_{\!X}^{p} \ud\mu(x)\ud\mu(y)}\bigg)^{\!\!\!\frac{1-\omega}{p}}\\
&=\nonumber \frac{2^{1-\omega}}{\eta(p,\omega)^\omega} \bigg(\frac{2\int_{X} \|x-u\|_{\!X}^{p\omega}
\ud\mu(x)}{\iint_{X\times X} \|f_\omega(x-u)-f_\omega(y-u)\|_{\!X}^{p} \ud\mu(x)\ud\mu(y)}\bigg)^{\!\!\!\frac{1-\omega}{p}}
\\ \nonumber &\!\!\stackrel{\eqref{eq:translate to 0}}\le \frac{2^{1-\omega}(2\beta)^{\!\frac{1-\omega}{p}}}{\eta(p,\omega)^\omega} \bigg(\frac{\iint_{X\times X} \|x-y\|_{\!X}^{p\omega} \ud\mu(x)\ud\mu(y)}{\iint_{X\times X} \|f_\omega(x-u)-f_\omega(y-u)\|_{\!X}^{p} \ud\mu(x)\ud\mu(y)}\bigg)^{\!\!\!\frac{1-\omega}{p}}\\ &\!\!\stackrel{\eqref{eq:compute L}}{=} \frac{2^{1-\omega}(2\beta)^{\!\frac{1-\omega}{p}}}{\eta(p,\omega)^\omega} \left(\frac{\|\f\|_{\!\Lip_\omega(X,X)}^{\phantom{p}}}{2^{1-\omega}}\right)^{\!\!1-\omega},\label{eq:for botstrap quote}
\end{align}
where~\eqref{eq:here we use holder}  is an application of Jensen's inequality for the probability measure on $X\times X$ whose Radon--Nikodym derivative with respect to $\mu\times \mu$ is proportional to $(x,y)\mapsto \|f_\omega(x-u)-f_\omega(y-u)\|_{\!X}^{p}$. Now, the bound~\eqref{eq:for botstrap quote}  simplifies to give the desired estimate
\begin{equation*}
\|\f\|_{\!\Lip_\omega(X,X)}^{\phantom{p}}\le \frac{2^{1-\omega}(2\beta)^{\!\frac{1-\omega}{p\omega}}}{\eta(p,\omega)}\stackrel{\eqref{eq:def beta D}}{=}D.\tag*{\qedhere}
\end{equation*}
\end{proof}

\subsection{Deduction of Theorem~\ref{thm:lp version Kp} from~\eqref{eq:C2}}\label{sec:pass to larger p}

As we stated in the Introduction, the matrix-dimension inequality~\eqref{eq:matrix dim p sharp} of Theorem~\ref{thm:lp version Kp}  is a formal consequence of the $\ell_1$ matrix-dimension inequality~\eqref{eq:C2} that we deduced there from Theorem~\ref{thm:average john}. We also explained in the Introduction that if one settles for a matrix-dimension inequality as in~\eqref{eq:beta p version} with a worse asymptotic dependence on $p$ as $p\to \infty$ than that of~\eqref{eq:matrix dim p sharp}, which we expect to be sharp (recall Conjecture~\ref{conj:p depend}), then this could be done using reductions from~\cite[Section~7.4]{Nao14} between notions of $q$-average distortion as $q$ varies over $[1,\infty)$, or using the better bounds of Proposition~\ref{prop:other exponents}. However, it seems that neither the literature nor Proposition~\ref{prop:other exponents}  suffice for deducing Theorem~\ref{thm:lp version Kp} from~\eqref{eq:C2}. We rectify this here using the elementary bounds that we derived in Section~\ref{sec:normalization} and basic input from topological degree theory. Those who are not concerned with obtaining the conjecturally sharp dependence on $p$ can therefore skip the present section and instead mimic the argument of the Introduction that led to~\eqref{eq:C2}.

\subsubsection{A nonlinear Rayleigh quotient inequality}\label{sec:degree} The following lemma relates quantities that are naturally viewed as nonlinear versions of classical Rayleigh quotients. The need for estimates of this type first arose due to algorithmic concerns in~\cite{ANNRW18,ANNRW-FOCS18}; see also the survey~\cite[Section~5.1.1]{Nao18}.

\begin{lemma}\label{lem:ray} Fix $n\in \N$ and $p,q\in [1,\infty)$ with $p\le q$. Suppose that $\pi=(\pi_1,\ldots,\pi_n)\in \bigtriangleup^{\!n-1}$ and that $\A=(a_{ij})\in \M_n(\R)$ is a stochastic and $\pi$-reversible matrix. Let $(X,\|\cdot\|_{\!X}^{\phantom{p}})$ be a Banach space. For any $x_1,\ldots,x_n\in X$ there are $y_1=y_1(\frac{p}{q},\pi,x_1,\ldots,x_n),\ldots, y_n=y_n(\frac{p}{q},\pi,x_1,\ldots,x_n)\in X$ with
\begin{equation}\label{eq:p/q version}
\bigg(\frac{\sum_{i=1}^n\sum_{j=1}^n \pi_i\pi_j \|y_i-y_j\|_{\!X}^p}{\sum_{i=1}^n\sum_{j=1}^n \pi_ia_{ij} \|y_i-y_j\|_{\!X}^p}\bigg)^{\!\!\frac{1}{p}}\ge \frac{p}{2q}\bigg(\frac{\sum_{i=1}^n\sum_{j=1}^n \pi_i\pi_j \|x_i-x_j\|_{\!X}^q}{\sum_{i=1}^n\sum_{j=1}^n \pi_ia_{ij} \|x_i-x_j\|_{\!X}^q}\bigg)^{\!\!\frac{1}{q}}.
\end{equation}
\end{lemma}

Prior to proving Lemma~\ref{lem:ray}, we will proceed to see how, in combination with~\eqref{eq:C2}, it quickly implies  Theorem~\ref{thm:lp version Kp}.  Indeed, let $(X,\|\cdot\|_X)$ be a finite-dimensional normed space, $n\in \N$ and $x_1,\ldots,x_n\in X$. Fix $\pi \in \nsimplex$, a stochastic $\pi$-reversible matrix $\A=(a_{ij})\in \M_n(\R)$ and $p\ge 1$. Apply Lemma~\ref{lem:ray} to get new vectors $y_1,\ldots,y_n\in X$ that satisfy the inequality.
\begin{equation}\label{eq:2p ratio}
\frac{\sum_{i=1}^n\sum_{j=1}^n \pi_i\pi_j \|y_i-y_j\|_{\!X}^{\phantom{p}}}{\sum_{i=1}^n\sum_{j=1}^n \pi_ia_{ij} \|y_i-y_j\|_{\!X}^{\phantom{p}}}\ge \frac{1}{2p}\bigg(\frac{\sum_{i=1}^n\sum_{j=1}^n \pi_i\pi_j \|x_i-x_j\|_{\!X}^p}{\sum_{i=1}^n\sum_{j=1}^n \pi_ia_{ij} \|x_i-x_j\|_{\!X}^p}\bigg)^{\!\!\frac{1}{p}}.
\end{equation}

An application of~\eqref{eq:C2} to the new configuration $\{y_1,\ldots,y_n\}\subset X$ of points in $X$ yields the following lower bound on the dimension of $X$, in which $\mathsf{C}\ge 1$ is the universal constant of Theorem~\ref{thm:average john}.
\begin{equation}\label{eq:use dim bound for yi}
\dim(X)\ge \exp\!\bigg(\frac{1-\lambda_2(\A)}{\mathsf{C}^2}\cdot \frac{\sum_{i=1}^n\sum_{j=1}^n \pi_i\pi_j \|y_i-y_j\|_{\!X}^{\phantom{p}}}{\sum_{i=1}^n\sum_{j=1}^n \pi_ia_{ij}\|y_i-y_j\|_{\!X}^{\phantom{p}}}\bigg),
\end{equation}
Upon substitution of~\eqref{eq:2p ratio} into~\eqref{eq:use dim bound for yi} we get the desired bound~\eqref{eq:matrix dim p sharp} with $\mathsf{K}=2\mathsf{C}^2$.\qed

Towards the proof of Lemma~\ref{lem:ray}, a variant of the following preparatory lemma also played a key role in~\cite{ANNRW18,ANNRW-FOCS18}, for similar purposes. Its short proof relies on considerations from algebraic topology.

\begin{lemma}\label{lem:onto} Suppose that $(X,\|\cdot\|^{\phantom{p}}_{\!X})$ is a finite-dimensional normed space and that $f:X\to X$ is a  continuous function that satisfies
\begin{equation}\label{eq:at infty assumption}
\lim_{R\to \infty} \inf_{\substack{x\in X\\ \|x\|^{\phantom{p}}_{\!X}\ge R}}\big(\|x\|^{\phantom{\frac{p}{q}}}_{\!X}-\|x-f(x)\|^{\phantom{\frac{p}{q}}}_{\!X}\big)=\infty.
\end{equation}
Then $f$ is surjective.
\end{lemma}

\begin{proof} Write $\dim(X)=k$. Fix an arbitrary point $z\in \mathbb{S}^k$ in the Euclidean sphere $\mathbb{S}^k$ of $\R^{k+1}$, and fix also any homeomorphism $h:\mathbb{S}^{k}\setminus \{z\}\to X$ between the punctured sphere $\mathbb{S}^{k}\setminus \{z\}$ and $X$. Define $g:\mathbb{S}^k\to \mathbb{S}^k$ by $g(w)= h^{-1}\circ f\circ h(w)$ for $w\in \mathbb{S}^{k}\setminus \{z\}$, and $g(z)=z$. We claim that $g$ is continuous at $z$, and hence it is continuous on all of  $\mathbb{S}^k$. Indeed, if $\{w_n\}_{n=1}^\infty \subset \mathbb{S}^k$ and $\lim_{n\to \infty} w_n=z$, then $\lim_{n\to \infty} \|h(w_n)\|_X=\infty$. Consequently, $$\big\|f\big(h(w_n)\big)\big\|_{\!X}^{\phantom{p}}\ge \|h(w_n)\|_{\!X}^{\phantom{p}}-\big\|h(w_n)-f\big(h(w_n)\big)\big\|_{\!X}^{\phantom{p}}\xrightarrow[n\to \infty]{}\infty,$$
where we used~\eqref{eq:at infty assumption}.  Therefore $\lim_{n\to \infty} g(w_n)=z$, as required.

We next claim that $g$ is homotopic to the identity mapping $\mathsf{Id}_{\mathbb{S}^k}:\mathbb{S}^k\to \mathbb{S}^k$. Indeed, denote
$$
\forall (t,w)\in [0,1]\times (\mathbb{S}^k\setminus\{z\}),\qquad H(t,w)\eqdef h^{-1}\Big(th(w)+(1-t)f\big(h(w)\big)\Big),
$$
and $H(t,z)=z$ for all $t\in [0,1]$. If we will check that  $H:[0,1]\times \mathbb{S}^k\to \mathbb{S}^k$ is continuous at every point of $[0,1]\times \{z\}$, then it would follow that it is continuous on all of $[0,1]\times \mathbb{S}^k$, thus yielding the desired homotopy. To see this, take any $\{t_n\}_{n=1}^\infty\subset [0,1]$ such that $\lim_{n\to \infty} t_n=t$ exists, and any $\{w_n\}_{n=1}^\infty \subset \mathbb{S}^k$ with $\lim_{n\to \infty} w_n=z$. We then have $\lim_{n\to \infty} \|h(w_n)\|_{\!X}^{\phantom{p}}=\infty$, and therefore

\begin{align*}
\big\|t_nh(w_n)+(1-t_n)f\big(h(w_n)\big)\big\|_{\!X}^{\phantom{p}}&\ge \|h(w_n)\|_{\!X}^{\phantom{p}}-(1-t_n)\big\|h(w_n)-f\big(h(w_n)\big)\big\|_{\!X}^{\phantom{p}}\\&\ge \|h(w_n)\|_{\!X}^{\phantom{p}}-\big\|h(w_n)-f\big(h(w_n)\big)\big\|_{\!X}^{\phantom{p}}\xrightarrow[n\to \infty]{}\infty,
\end{align*}
where we used~\eqref{eq:at infty assumption} once more. Hence $\lim_{n\to \infty} H(t_n,w_n)=z=H(t,z)$, as required.

Because we showed that $g$ is homotopic to the identity on $\mathbb{S}^k$, it has  degree $1$, and therefore it is surjective; see e.g.~\cite[page~134]{Hat02}. Hence $h^{-1}\circ f\circ h(\mathbb{S}^n\setminus\{z\})=\mathbb{S}^n\setminus\{z\}$, i.e., $f(X)=X$.
\end{proof}

\begin{lemma}\label{lem:new vectors 1} Fix $n\in \N$, $\pi\in \bigtriangleup^{\!n-1}$, $\omega\in (0,1)$, a Banach space $(X,\|\cdot\|_{\!X}^{\phantom{p}})$, and $x_1,\ldots,x_n\in X$. Then,  there exist new vectors $y_1=y_1(\omega,\pi,x_1,\ldots,x_n),\ldots, y_n=y_n(\omega,\pi,x_1,\ldots,x_n)\in X$ that satisfy $\sum_{i=1}^n \pi_i y_i=0$, and for every $q\in (0,\infty)$ we have
\begin{equation}\label{holder after centering omega}
\forall\, i,j\in \n,\qquad 2^{1-\frac{1}{\omega}}\|x_i-x_j\|_{\!X}^{\!\frac{1}{\omega}}\le \|y_i-y_j\|_{\!X}^{\phantom{p}}\le \frac{(\|y_i\|_{\!X}^{ q\omega}+\|y_j\|_{\!X}^{q\omega })^{\!\frac{1-\omega}{q\omega}}}{\eta(q,\omega)}\|x_i-x_j\|_{\!X}^{\phantom{p}}.
\end{equation}
\end{lemma}

\begin{proof} Our eventual goal is to apply Lemma~\ref{lem:onto} to the mapping
$$
f=f_{\omega,\pi,x_1,\ldots,x_n}:\spn(\{x_1,\ldots,x_n\})\to \spn(\{x_1,\ldots,x_n\})
$$
that is defined by setting for every $x\in \spn(\{x_1,\ldots,x_n\})$,
\begin{equation}\label{eq:new f composition}
f(x)\eqdef \sum_{i=1}^n \pi_i f_\omega^{-1}(f_\omega(x)-x_i)\stackrel{\eqref{eq:def f omega}}{=}\frac{1}{\|x\|_{\!X}^{\!\frac{1}{\omega}-1}}\sum_{i=1}^n\pi_i \big\|x-\|x\|_{\!X}^{1-\omega}x_i\big\|_{\!X}^{\!\frac{1}{\omega}-1}\big(x-\|x\|_{\!X}^{1-\omega}x_i\big).
\end{equation}
Suppose for the moment that we checked that $f$ satisfies the assumption~\eqref{eq:at infty assumption} of Lemma~\ref{lem:onto}. It would follow that $f$ is surjective, and in particular there exists $x=x(\omega,\pi,x_1,\ldots,x_n)\in \spn(\{x_1,\ldots,x_n\})$ such that $f(x)=0$. Thus, if we choose  $y_i=f_\omega^{-1}(f_\omega(x)-x_i)$ for  $i\in \n$, then $\sum_{i=1}^n\pi_iy_i=0$. Because $f_\omega(y_1)+x_1=\ldots=f_\omega(y_n)+x_n=f_\omega(x)$, we have $\|f_\omega(y_i)-f_\omega(y_j)||_{\!X}^{\phantom{p}}=\|x_i-x_j\|_{\!X}^{\phantom{p}}$ for all $i,j\in \n$. The desired bounds~\eqref{holder after centering omega} now follows from Lemma~\ref{lem:sharp holder bounds on normalization}.

Both  $f_\omega$ and $f_\omega^{-1}$ are continuous, so  $f$ is also continuous. Write $\xi_z(x)=f_\omega^{-1}(f_\omega(x)-z)$ for every  $x,z\in X$. Then $f(x)=\sum_{i=}^n \pi_i\xi_{x_i}(x)$, and therefore by the convexity of $\|\cdot\|_X:X\to X$ we have
$$
\|x\|_{\!X}^{\phantom{p}}-\|x-f(x)\|_{\!X}^{\phantom{p}}=\|x\|_{\!X}^{\phantom{p}}-\Big\|\sum_{i=1}^n \pi_i(\xi_{x_i}(x)-x)\Big\|_{\!X}^{\phantom{p}}\ge \sum_{i=1}^n \pi_i \big(\|x\|_{\!X}^{\phantom{p}}-\|x-\xi_{x_i}(x)\|_{\!X}^{\phantom{p}}\big).
$$
This shows that the assumption~\eqref{eq:at infty assumption} of Lemma~\ref{lem:onto} would hold true if $\xi_z$ satisfied it for every fixed $z\in X$.  Since  $f_\omega(x)-f_\omega(\xi_z(x))=z$ by the definition of $\xi_z$, the case $p=\frac{1}{\omega}$ of Lemma~\ref{lem:sharp holder bounds on normalization} gives
\begin{equation}\label{eq:xi1}
\|z\|_{\!X}^{\phantom{p}}\ge \frac{2^{1-\omega}\omega}{(\|x\|_{\!X}^{\phantom{p}}+\|\xi_z(x)\|_{\!X}^{\phantom{p}})^{1-\omega}}\|x-\xi_z(x)\|_{\!X}^{\phantom{p}},
\end{equation}
where we also used the fact that $\eta(\frac{1}{\omega},\omega)=\omega2^{1-\omega}$, by Lemma~\ref{lem bounds on eta}. Note that
\begin{equation}\label{eq:xi2}
\|\xi_z(x)\|_{\!X}^{\phantom{p}}=\|f_\omega(x)-z\|_{\!X}^{\!\frac{1}{\omega}}\le \big(\|f_\omega(x)\|_{\!X}^{\phantom{p}}+\|z\|_{\!X}^{\phantom{p}}\big)^{\!\frac{1}{\omega}}=\big(\|x\|_{\!X}^{\omega\phantom{p}}+
\|z\|_{\!X}^{\phantom{p}}\big)^{\!\frac{1}{\omega}}.
\end{equation}
By combining~\eqref{eq:xi1} and~\eqref{eq:xi2} we conclude that
\begin{equation*}
\|x\|_{\!X}^{\phantom{p}}-\|x-\xi_z(x)\|_{\!X}^{\phantom{p}}\ge \|x\|_{\!X}^{\phantom{p}}-\frac{\|z\|_{\!X}^{\phantom{p}}}{\omega}\big(\|x\|_{\!X}^\omega+\|z\|_{\!X}^{\phantom{p}}\big)^{\!\frac{1-\omega}{\omega}} \xrightarrow[\|x\|_{\!X}^{\phantom{p}}\to \infty]{}\infty.\tag*{\qedhere}
\end{equation*}

\end{proof}

\begin{proof}[Completion of the proof of Lemma~\ref{lem:ray}] The ensuing reasoning is inspired by an idea of Matou\v{s}ek~\cite{Mat97}. Apply Lemma~\ref{lem:new vectors 1}  with $\omega=\frac{p}{q}$ to get $y_1,\ldots,y_n\in X$ (depending on $\frac{p}{q},\pi,x_1,\ldots,x_n$) such that
\begin{equation}\label{eq:centered}
\sum_{i=1}^n \pi_i y_i=0,
\end{equation}
and for every $i,j\in \n$,
\begin{equation}\label{holder after centering}
 2^{1-\frac{q}{p}}\|x_i-x_j\|_{\!X}^{\!\frac{q}{p}}\le \|y_i-y_j\|_{\!X}^{\phantom{\!\frac{q}{p}}}\le \frac{q}{p}\|x_i-x_j\|_{\!X}^{\phantom{\!\frac{q}{p}}}\left(\frac{\|y_i\|_{\!X}^{p}+\|y_j\|_{\!X}^p}{2}\right)^{\!\!\frac{1}{p}-\frac{1}{q}},
\end{equation}
where we also used the fact that $\eta(q,\frac{p}{q})=\frac{p}{q}2^{\frac{1}{p}-\frac{1}{q}}$, by Lemma~\ref{lem bounds on eta}.  Note that
\begin{align}\label{eq:rank 1}
\begin{split}
\sum_{i=1}^n\sum_{j=1}^n \pi_i a_{ij} \frac{\|y_i\|_{\!X}^p+\|y_j\|_{\!X}^p}{2}&= \frac12\sum_{i=1}^n\sum_{j=1}^n \pi_i a_{ij}\|y_i\|_{\!X}^p +\frac12\sum_{j=1}^n \sum_{i=1}^n \pi_ja_{ji}\|y_j\|_{\!X}^p \\&=\sum_{i=1}^n \pi_i \Big\|y_i-\sum_{s=1}^n \pi_s y_s\Big\|_{\!X}^p\le \sum_{i=1}^n \sum_{j=1}^n \pi_i\pi_j\|y_i-y_j\|_{\!X}^p,
\end{split}
\end{align}
where the first step uses $\pi$-reversibility, the second step uses stochasticity and the centering condition~\eqref{eq:centered}, and the final step follows from Jensen's inequality (since $p\ge 1$).  Hence,
\begin{align*}
\sum_{i=1}^n\sum_{j=1}^n \pi_ia_{ij} \|y_i-y_j\|_{\!X}^p&\le \Big(\frac{q}{p}\Big)^p\sum_{i=1}^n\sum_{j=1}^n \pi_ia_{ij}\|x_i-x_j\|_{\!X}^p\left(\frac{\|y_i\|_{\!X}^{p}+\|y_j\|_{\!X}^p}{2}\right)^{1-\frac{p}{q}}\\&
\le \Big(\frac{q}{p}\Big)^p\bigg(\sum_{i=1}^n\sum_{j=1}^n \pi_ia_{ij}\|x_i-x_j\|_{\!X}^q\bigg)^{\frac{p}{q}}\bigg(\sum_{i=1}^n\sum_{j=1}^n \pi_i a_{ij} \frac{\|y_i\|_{\!X}^p+\|y_j\|_{\!X}^p}{2}\bigg)^{1-\frac{p}{q}}\\
&\le \Big(\frac{q}{p}\Big)^p\bigg(\sum_{i=1}^n\sum_{j=1}^n \pi_ia_{ij}\|x_i-x_j\|_{\!X}^q\bigg)^{\frac{p}{q}}\bigg(\sum_{i=1}^n\sum_{j=1}^n \pi_i\pi_j\|y_i-y_j\|_{\!X}^{p}\bigg)^{1-\frac{p}{q}},
\end{align*}
where the first step is the second inequality in~\eqref{holder after centering},  the second step is H\"older's inequality, and the final step is~\eqref{eq:rank 1}.  This simplifies to give
\begin{align*}
\bigg(\frac{\sum_{i=1}^n\sum_{j=1}^n \pi_i\pi_j \|y_i-y_j\|_{\!X}^p}{\sum_{i=1}^n\sum_{j=1}^n \pi_ia_{ij} \|y_i-y_j\|_{\!X}^p}\bigg)^{\frac{1}{p}}&\ge \frac{q}{p}\bigg(\frac{\sum_{i=1}^n\sum_{j=1}^n \pi_i\pi_j \|y_i-y_j\|_{\!X}^p}{\sum_{i=1}^n\sum_{j=1}^n \pi_ia_{ij} \|x_i-x_j\|_{\!X}^q}\bigg)^{\frac{1}{q}}\\&\ge \frac{p}{2^{1-\frac{p}{q}}q}\bigg(\frac{\sum_{i=1}^n\sum_{j=1}^n \pi_i\pi_j \|x_i-x_j\|_{\!X}^q}{\sum_{i=1}^n\sum_{j=1}^n \pi_ia_{ij} \|x_i-x_j\|_{\!X}^q}\bigg)^{\frac{1}{q}},
\end{align*}
where the final step is the first inequality in~\eqref{holder after centering}.
\end{proof}

\begin{remark}\label{rem:LS}  In~\cite[Proposition~3.9]{LS17} de Laat and de la Salle proved that for every Banach space $(X,\|\cdot\|_X)$, every $n\in \N$, every $\pi\in \bigtriangleup^{\!n-1}$ and every $\pi$-reversible stochastic matrix $\A=(a_{ij})\in \M_n(\R)$,
\begin{equation}\label{eq:extrapolation implicit}
\forall\, 1\le p\le q<\infty,\qquad \gamma(\A,\|\cdot\|_{\!X}^q)^{\!\frac{p}{q}}\lesssim_{p,q} \gamma(\A,\|\cdot\|_{\!X}^p)\lesssim_{p,q} \gamma(\A,\|\cdot\|_{\!X}^q).
\end{equation}
This is a Banach space-valued generalization of a useful extrapolation result for Poincar\'e inequalities that Matou\v{s}ek proved in~\cite{Mat97} for real-valued functions (see also~\cite[Lemma~5.5]{BLMN05} or~\cite[Lemma~4.4]{NS11}). Direct precursors of~\eqref{eq:extrapolation implicit} are those of~\cite{Mim15,Che16}, but they treat the case of graphs (relying on their representation as Schreier coset graphs due to~\cite{Gro77}, as well as ideas of~\cite{BFGM07}) with the resulting bound depending on the maximum degree; as such, these earlier versions are not suitable for applications that use arbitrary stochastic matrices (e.g.~when using duality as we do here).

Using Theorem~\ref{thm:full duality}, it follows from the rightmost inequality in~\eqref{eq:extrapolation implicit} that for every $\omega\in (0,1]$ the $\omega$-snowflake of $X$ embeds with $(1/\omega)$-average distortion $D_\omega\ge 1$ into an ultrapower of $\ell_{1/\omega}(X)$, where $D_\omega$ may depend only on $\omega$ (for this, we are considering~\eqref{eq:extrapolation implicit} with $p=1$ and $q=1/\omega$. More generally, \eqref{eq:extrapolation implicit} and Theorem~\ref{thm:full duality} yield an embedding of the $\omega$-snowflake of $X$ into an ultrapower of $\ell_{q}(X)$ with $q$-average distortion $D_{\omega,q}$). Proposition~\ref{prop:average of snowflake back into X}  shows that this is so even for embeddings into $X$ itself. The ingredients of Proposition~\ref{prop:average of snowflake back into X} and~\cite[Proposition~3.9]{LS17} are similar, as~\cite{LS17} considers an $L_p(X)$-valued version of the normalization map that is given in~\eqref{eq:def f omega} as a generalization of the classical Mazur map~\cite{Maz29} (see also~\cite{OS94,Cha95,Dah95,Ray02,Ric15} for earlier variants in special cases, as well  as the subsequent development in~\cite{ANNRW-FOCS18}). We will next show that by incorporating the reasoning of the present section, we obtain the following version of~\eqref{eq:extrapolation implicit} with an explicit dependence on $p,q$.
\begin{equation}\label{eq:lasa good pq}
\forall\, 1\le p\le q<\infty,\qquad \Big(\frac{p}{2q}\Big)^{\! p}\gamma(\A,\|\cdot\|_{\!X}^q)^{\!\frac{p}{q}}\le \gamma(\A,\|\cdot\|_{\!X}^p)\le \Big(\frac{2q}{p}\Big)^{\!q} \gamma(\A,\|\cdot\|_{\!X}^q).
\end{equation}

An inspection of the proof in~\cite{LS17} reveals that the dependence on $p,q$ that it yields is much (exponentially) weaker asymptotically   than that of~\eqref{eq:lasa good pq}, and we believe that this is inherent to the reasoning of~\cite{LS17}.  The first inequality in~\eqref{eq:lasa good pq} is sharp, as already shown in~\cite{Mat97} for real-valued functions. We do not know if the second inequality in~\eqref{eq:lasa good pq} is sharp, and conceivably the rightmost factor $(2q/p)^q$ in~\eqref{eq:lasa good pq} could be replaced by $e^{O(q)}$. If this were indeed possible, then it would be a worthwhile result because it would yield a fully analogous vector-valued generalization of Cheeger's inequality~\cite{Che70,Chu96} and Matou\v{s}ek's extrapolation phenomenon~\cite{Mat97}.

To deduce the first inequality in~\eqref{eq:lasa good pq},  take $x_1,\ldots,x_n\in X$ and use Lemma~\ref{lem:ray} to obtain new vectors $y_1,\ldots,y_n\in X$ such that
\begin{equation}\label{eq:use for yi}
\Big(\frac{p}{2q}\Big)^{\!p}\cdot \bigg(\frac{\sum_{i=1}^n\sum_{j=1}^n \pi_i\pi_j \|x_i-x_j\|_{\!X}^q}{\sum_{i=1}^n\sum_{j=1}^n \pi_ia_{ij} \|x_i-x_j\|_{\!X}^q}\bigg)^{\!\!\frac{p}{q}}\le \frac{\sum_{i=1}^n\sum_{j=1}^n \pi_i\pi_j \|y_i-y_j\|_{\!X}^p}{\sum_{i=1}^n\sum_{j=1}^n \pi_ia_{ij} \|y_i-y_j\|_{\!X}^p}\le \gamma(\A,\|\cdot\|_{\!X}^p),
\end{equation}
where the last step of~\eqref{eq:use for yi} is the definition of $\gamma(\A,\|\cdot\|_{\!X}^p)$ applied to the new configuration of vectors $\{y_1,\ldots,y_n\}\subset X$. It remains to note that by the definition of $\gamma(\A,\|\cdot\|_{\!X}^q)$, the supremum of the left hand side of~\eqref{eq:use for yi} over all possible $x_1,\ldots,x_n\in X$ equals the left hand side of~\eqref{eq:lasa good pq}.

To deduce the second inequality in~\eqref{eq:lasa good pq}, use Proposition~\ref{prop:average of snowflake back into X}  and Corollary~\ref{cor:specific bounds} with $\omega=p/q$ and $(\mu(x_1),\ldots,\mu(x_n))=\pi$ to get new vectors $z_1,\ldots,z_n\in X$ satisfying $\|z_i-z_j\|_X^q\le (p/2q)^q\|x_i-x_j\|_X^{p}$ for all $i,j\in \n$, and also $\sum_{i=1}^n\sum_{j=1}^n \pi_ia_{ij} \|z_i-z_j\|_{\!X}^q\ge \sum_{i=1}^n\sum_{j=1}^n \pi_i\pi_j \|x_i-x_j\|_{\!X}^p$. Thus,
\begin{equation*}
\frac{\sum_{i=1}^n\sum_{j=1}^n \pi_i\pi_j \|x_i-x_j\|_{\!X}^p}{\sum_{i=1}^n\sum_{j=1}^n \pi_ia_{ij} \|x_i-x_j\|_{\!X}^p}\le \Big(\frac{2q}{p}\Big)^{\!q}\cdot \frac{\sum_{i=1}^n\sum_{j=1}^n \pi_i\pi_j \|z_i-z_j\|_{\!X}^q}{\sum_{i=1}^n\sum_{j=1}^n \pi_ia_{ij} \|z_i-z_j\|_{\!X}^q}\le \Big(\frac{2q}{p}\Big)^{\!q}\gamma(\A,\|\cdot\|_{\!X}^q).
\end{equation*}
\end{remark}
\section{Impossibility results}\label{sec:impossibility}
The main purpose of this section is to prove Lemma~\ref{lem:snowflake john}, Lemma~\ref{prop:p/2 sharp} and to refine the lower bound~\eqref{eq:quad dist for 1/2 snow} on the Hilbertian average distortion of snowflakes of regular graphs with a spectral gap.

The assertion of Lemma~\ref{lem:snowflake john} that the $\omega$-snowflake of any $k$-dimensional normed space $(X,\|\cdot\|_X)$ embeds into a Hilbert space with bi-Lipschitz distortion $k^{\frac{\omega}{2}}$ is a quick consequence of John's theorem~\cite{Joh48}, combined with Schoenberg's result~\cite{Sch38} that the $\omega$-snowflake of an infinite dimensional Hilbert space $(H,\|\cdot\|_H)$ embeds isometrically into $H$. Indeed, by John's theorem there exists a mapping $f:X\to H$ such that $\|x-y\|_X\le \|f(x)-f(y)\|_H\le \sqrt{k}\|x-y\|_X$ for all $x,y\in X$. By Schoenberg's theorem there exists  a mapping $g:H\to H$ such that $\|g(u)-g(v)\|_H=\|u-v\|_H^\omega$ for all $u,v\in H$. Hence, $\|x-y\|_X^\omega\le ||g\circ f(x)-g\circ f(y)\|_H\le k^{\frac{\omega}{2}}\|x-y\|_X^\omega$ for all $x,y\in X$.

The more substantial  novelty of  Lemma~\ref{lem:snowflake john} is the assertion that the above composition of Schoenberg's embedding and John's embedding yields the correct worst-case asymptotic behavior (up to universal constant factors) of the Hilbertian bi-Lipschitz distortion of $\omega$-snowflakes of $k$-dimensional normed spaces. The proof is a quick application of the invariant {\em metric cotype $2$ with sharp scaling parameter} that was introduced in~\cite{MN08}, but this has not been previously noted in the literature. Note that the endpoint case $\omega=1$ here is classical, by a reduction to the linear theory through differentiation, but this approach is inherently unsuitable for treating H\"older functions.

\begin{proof}[Proof of Lemma~\ref{lem:snowflake john}] Fix $k\in \N$ and consider the normed space $\ell_\infty^{k^2}(\C)\cong \ell_\infty^{k^2}(\ell_2^2)$ whose dimension over $\R$ equals $2k^2$. Suppose that the $\omega$-snowflake of $\ell_\infty^{k^2}(\C)$ embeds into a Hilbert space $(H,\|\cdot\|_H)$ with bi-Lipschitz distortion less than $D$. Thus, there exists an embedding $f:\ell_\infty^{k^2}(\C)\to H$ such that
\begin{equation}\label{eq:torus}
\forall\, x,y\in \ell_\infty^{k^2}(\C),
\qquad \|x-y\|_{\ell_\infty^{k^2}(\C)}^\omega\le \|f(x)-f(y)\|_H\le D\|x-y\|_{\ell_\infty^{k^2}(\C)}^\omega.
\end{equation}

By~\cite[Section~3]{MN08}, the following inequality holds true for any $f:\{1,\ldots, 4k\}^{k^2}\to H$.
\begin{align}\label{eq:use cotype}
\begin{split}
\frac{1}{(4k)^{k^2}}\sum_{j=1}^{k^2} &\sum_{x\in \{1,\ldots,4k\}^{k^2}}\bigg\|f \Big(-e^{\frac{\pi \mathsf{i}}{2k}x_j}e_j+\sum_{r\in \{1,\ldots,k^2\}\setminus\{j\}} e^{\frac{\pi \mathsf{i}}{2k}x_r}e_r\Big)-f \Big(\sum_{r=1}^{k^2} e^{\frac{\pi \mathsf{i}}{2k}x_r}e_r\Big)\bigg\|_{\!H}^2\\ &\lesssim \frac{k^2}{(12k)^{k^2}}\sum_{\e\in \{-1,0,1\}^{k^2}}\sum_{x\in \{1,\ldots,4k\}^{k^2}}\bigg\|f \Big(\sum_{r=1}^{k^2} e^{\frac{\pi \mathsf{i}}{2k}(x_r+\e_r)}e_r\Big)-f \Big(\sum_{r=1}^{k^2} e^{\frac{\pi \mathsf{i}}{2k}x_r}e_r\Big)\bigg\|_{\!H}^2,
\end{split}
\end{align}
where $e_1,\ldots,e_{k^2}$ is the standard basis of $\C^{k^2}$. By combining~\eqref{eq:torus} with~\eqref{eq:use cotype}, we conclude that
\begin{equation*}
2^{2\omega}k^2\lesssim k^2 D^2\Big|e^{\frac{\pi \mathsf{i}}{2k}}-1\Big|^{2\omega}\asymp k^{2(1-\omega)}D^2\iff D\gtrsim k^\omega\asymp \dim \!\big( \ell_\infty^{k^2}(\C)\big)^{\!\frac{\omega}{2}}.\tag*{\qedhere}
\end{equation*}
\end{proof}

The above proof of Lemma~\ref{lem:snowflake john} works also for embeddings of $\omega$-snowflakes of $k$-dimensional normed spaces into $L_p$ when $p\in [1,2]$, yielding the same conclusion. However, for $p>2$ the  upper and lower bounds that it yields (using sharp metric cotype $p$) do not match. We therefore ask

\begin{question} Suppose that $\omega\in (0,1)$ and $p\in (2,\infty)$. What is
 the infimum over those $\beta\in (0,1]$ for which there exists $\alpha_\beta\in (0,\infty)$ such that the $\omega$-snowflake of any finite-dimensional normed space $X$ embeds into $L_p$ with bi-Lipschitz distortion at most $\alpha_\beta\dim(X)^\beta$?
\end{question}

Next, the optimality of Theorem~\ref{thm:really main} for Hilbertian targets in the regime of higher H\"older regularity,  as exhibited by  Lemma~\eqref{prop:p/2 sharp}, is a quick application of the classical invariant {\em Enflo type}~\cite{Enf76,BMW86}.

\begin{proof}[Proof of Lemma~\ref{prop:p/2 sharp}]  Fix $k\in \N$ and denote $\cc\eqdef k^{\frac{1}{p}-\frac12}$, so that $\cc_{\ell_2}(\ell_p^k)=\cc$; see e.g.~\cite[Section~8]{JL01}. Choose $X=\ell_p^k$, so that it has modulus of uniform smoothness of power type $p$. Suppose that $(Z,\|\cdot\|_Z)$ is a normed space whose modulus of uniform smoothness has power type $2$. We will show that if the $(\frac{p}{2}+\e)$-snowflake of $\ell_p^k$ embeds into $Z$ with $\alpha$-average distortion $D$, then necessarily
\begin{equation}\label{eq:p/2+eps refined}
D\gtrsim \frac{\cc^{\frac{2\e}{2-p}}}{\sqrt{\alpha}+\mathscr{S}_2(Z)}.
\end{equation}
The desired lower bound~\eqref{eq:alpha plus beta} in Lemma~\eqref{prop:p/2 sharp} is the special case of~\eqref{eq:p/2+eps refined} corresponding to $Z=\ell_\beta(\ell_2)$, because $\ell_\beta(\ell_2)$ is isometric to a subspace of $L_\beta$, and $\mathscr{S}_2(L_\beta)\asymp \sqrt{\beta}$ by~\cite{Fig76,BCL94}.

Let $\mathbb{F}_2$ be the field of two elements. Identify $\mathbb{F}_2^k$ with the $2^k$ vertices of the hypercube $\{0,1\}^k\subset \ell_p^k$. We also let $e_1,\ldots,e_k$ denote the standard coordinate basis and $\mu$ denote the uniform probability measure on $\mathbb{F}_2^k\subset \ell_p^k$, respectively. By~\cite[equation~(6.32)]{Nao14}, any $f:\mathbb{F}_2^k\to Z$ satisfies the following bound,  in which additions that occur in the argument of $f$ are in $\mathbb{F}_2^k$, i.e., modulo $2$ coordinate-wise.
\begin{align}\label{eq:enflo type}
\begin{split}
\bigg(\iint_{\mathbb{F}_2^k\times \mathbb{F}_2^k}\|f(x)-f(y)\|_{\!Z}^{\alpha\phantom{p}}\!&\ud\mu(x)\ud\mu(y)\bigg)^{\!\!\frac{1}{\alpha}}\\&\lesssim \big(\mathscr{S}_2(Z)+\sqrt{\alpha}\big)\sqrt{k} \bigg(\frac{1}{k}\sum_{i=1}^k \int_{\mathbb{F}_2^k} \|f(x+e_i)-f(x)\|_{\!Z}^{\alpha\phantom{p}}\!\ud\mu(x)\bigg)^{\!\!\frac{1}{\alpha}}.
\end{split}
\end{align}

 Suppose that $\|f(x+e_i)-f(x)\|_{Z}\le D$ for every $x\in \mathbb{F}_2^k$ and $i\in \{1,\ldots,k\}$. This would  follow if $f$ were $(\frac{p}{2}+\e)$-H\"older with constant $D$, which is what is relevant in the present context, but we are in fact assuming significantly  less here, namely that $f$ is $D$-Lipschitz in the metric that is induced by $\ell_1^k$ on $\mathbb{F}_2^k$. Under this assumption, the right hand side of~\eqref{eq:enflo type} is at most $D\big(\mathscr{S}_2(Z)+\sqrt{\alpha}\big)\sqrt{k}$.

 If we also have
$$
\bigg(\iint_{\mathbb{F}_2^k\times \mathbb{F}_2^k}\|f(x)-f(y)\|_{Z}^{\alpha\phantom{p}}\!\ud\mu(x)\ud\mu(y)\bigg)^{\!\!\frac{1}{\alpha}}\ge\bigg(\iint_{\mathbb{F}_2^k\times \mathbb{F}_2^k}\|x-y\|_{\ell_p^k}^{\alpha(\frac{p}{2}+\e)}\ud\mu(x)\ud\mu(y)\bigg)^{\!\!\frac{1}{\alpha}}\gtrsim k^{\frac12+\frac{\e}{p}},
$$
where the last step holds because $\|x-y\|_{\ell_p^k}\gtrsim k^{\frac{1}{p}}$ for a constant fraction of $(x,y)\in \mathbb{F}_2^k\times \mathbb{F}_2^k$, then by contrasting this with~\eqref{eq:enflo type} we conclude that
\begin{equation*}
D\gtrsim \frac{k^{\frac{\e}{p}}}{\sqrt{\alpha}+\mathscr{S}_2(Z)} =\frac{\cc^{\frac{2\e}{2-p}}}{\sqrt{\alpha}+\mathscr{S}_2(Z)}.\tag*{\qedhere}
\end{equation*}
\end{proof}

Note that the case $p=1$ of the above proof of  Lemma~\ref{prop:p/2 sharp} gives that  if the $(\frac12+\e)$-snowflake of $\ell_1^k$ embeds with quadratic average distortion $D\ge 1$ into a Hilbert space, then necessarily $D\gtrsim k^\e$. A slightly more careful examination (using Enflo's original ``diagonal versus edge'' inequality in~\cite{Enf69} in place of~\eqref{eq:enflo type}) of the constant factors in this special case reveals that we actually get the sharp bound $D\ge k^\e$ (for the uniform measure on $\{0,1\}^k$), as stated in the Introduction.

We will next revisit the estimate~\eqref{eq:quad dist for 1/2 snow} that was derived in the Introduction. Recalling the relevant setting, we are given $n\in \N$ and a connected regular graph $\G=(\n,E_\G)$.  In the course of the deduction of~\eqref{eq:quad dist for 1/2 snow}, see specifically the penultimate step in~\eqref{eq:penultimate}, we actually showed that if the  $\frac12$-snowflake of the shortest-path metric  $(\n,d_\G)$ embeds into a Hilbert space $(H,\|\cdot\|_H)$ with quadratic average distortion $D\ge 1$, then necessarily
$$
D\ge \sqrt{1-\lambda_2(\G)} \bigg(\frac{1}{n^2}\sum_{i=1}^n\sum_{j=1}^nd_\G(i,j)\bigg)^{\!\!\frac12}.
$$
Replacing the H\"older exponent $\frac12$ by an arbitrary $\omega\in (0,1]$, the same reasoning shows mutatis mutandis that if the  $\omega$-snowflake of the shortest-path metric  $(\n,d_\G)$ embeds into a Hilbert space $(H,\|\cdot\|_H)$ with quadratic average distortion $D\ge 1$, then necessarily
\begin{equation}\label{before improving the power of gap}
D\ge \sqrt{1-\lambda_2(\G)} \bigg(\frac{1}{n^2}\sum_{i=1}^n\sum_{j=1}^n d_\G(i,j)^{2\omega}\bigg)^{\!\!\frac12}.
\end{equation}
In particular, this implies that if $\G$ is  an expander, then $D\gtrsim (\log n)^\omega$. Therefore, by considering the uniform measure on an isometric copy of $(\n,d_\G)$ in $\ell_\infty^n$ we see that  Conjecture~\ref{conj:power of log} asks for the optimal asymptotic dependence on the dimension for fixed $\omega\in (0,\frac12)$.

As we explained  in the Introduction, the above proof of the ``vanilla'' spectral bound~\eqref{before improving the power of gap} goes back to~\cite{LLR95,Mat97,Gro00}; further examples of implementations of this (by now standard) useful idea can be found in~\cite{LM00,LMN02,Oza04,BLMN05,NRS05,KN06,Laf08,NR09,Pis10,NS11,GN12,MN14,Nao14,JV14,NR17,Nao18}. An especially important special case of~\eqref{before improving the power of gap} is when $\G$ is a vertex-transitive graph, e.g.~when it is the Cayley graph of a group of order $n$. In this case, by equation~(4.24) in~\cite{Nao14} (see also~\cite{NR09,HP15}) we have
\begin{equation}\label{eq:diam for transitive}
\forall\, p\ge 1,\ \forall\,\omega\in (0,1],\qquad \bigg(\frac{1}{n^2}\sum_{i=1}^n\sum_{j=1}^n d_\G(i,j)^{p\omega}\bigg)^{\!\!\frac1{p}}\asymp \diam(\G)^\omega,
\end{equation}
where $\diam(\G)$ is the diameter of $(\n,d_\G)$. So, \eqref{before improving the power of gap} for a vertex transitive graph becomes
\begin{equation}\label{eq:root gap diameter omega}
D\gtrsim \sqrt{1-\lambda_2(\G)} \diam(\G)^\omega.
\end{equation}
The following lemma improves the dependence on the spectral gap in~\eqref{before improving the power of gap} and~\eqref{eq:root gap diameter omega} for small $\omega$.

\begin{lemma}\label{lem:better power of gap} Fix $n\in \N$, $\omega\in (0,1]$ and $p,q,D\in [1,\infty)$. Let  $\G=(\n,E_\G)$ be a connected regular graph such that the $\omega$-snowflake of the metric space $(\n,d_\G)$ embeds with $q$-average distortion less than $D$ into $\ell_p$. Then necessarily
\begin{equation}\label{eq:lower D min}
D\gtrsim \frac{\big(1-\lambda_2(\G)\big)^{\min\left\{\omega,\frac{1}{\min\{p,2\}}\right\}}}{(p^2+q^2)^{\frac{1}{\min\{p,2\}}}} \bigg(\frac{1}{n^2}\sum_{i=1}^n\sum_{j=1}^n d_\G(i,j)^{ q\omega}\bigg)^{\!\!\frac{1}{q}}.
\end{equation}
In particular, if $\G$ is a vertex-transitive graph, then
\begin{equation}\label{eq:diam to omega}
D\gtrsim \frac{\big(1-\lambda_2(\G)\big)^{\min\left\{\omega,\frac{1}{\min\{p,2\}}\right\}}}{(p^2+q^2)^{\frac{1}{\min\{p,2\}}}} \diam(\G)^\omega.
\end{equation}
\end{lemma}

Prior to proving Lemma~\ref{lem:better power of gap}, we will discuss some of its consequences.

\begin{example}\label{example:SL}
Contrast Lemma~\ref{lem:better power of gap} with~\eqref{eq:root gap diameter omega}  in the following illustrative classical example. Fix $\mathfrak{q},k\in \N$ such that $\mathfrak{q}$ is  a power of a prime. Consider the Cayley graph of $\mathsf{SL}_k(\mathbb{F}_\mathfrak{q})$ that is induced by the symmetric generating set $\{\I_k\pm \mathsf{E}_k(i,j): (i,j)\in \k^2\ \wedge\ i\neq j\}$. Here, $\mathbb{F}_\mathfrak{q}$ is the field of size $\mathfrak{q}$ and $\mathsf{E}_k(i,j)\in \M_k(\mathbb{F}_\mathfrak{q})$ is the elementary matrix whose $(i,j)$-entry equals $1$ and the rest of its entries vanish. In what follows, $\mathsf{SL}_k(\mathbb{F}_\mathfrak{q})$ will always be assumed to be equipped with the word metric that corresponds to this (standard) generating set. We then have
\begin{equation}\label{eq:Kassabo and Alon}
1-\lambda_2\big(\mathsf{SL}_k(\mathbb{F}_\mathfrak{q})\big)\asymp \frac{1}{k}\qquad\mathrm{and}\qquad \diam\big(\mathsf{SL}_k(\mathbb{F}_\mathfrak{q})\big)\asymp \frac{k^2\log \mathfrak{q}}{\log k}.
\end{equation}
The first assertion in~\eqref{eq:Kassabo and Alon} is due to Kassabov~\cite{Kas05}. The second assertion  in~\eqref{eq:Kassabo and Alon} was obtained  by Alon~\cite{Alo19} who extended\footnote{Alon obtained the second assertion in~\eqref{eq:Kassabo and Alon} independently, before we learned of the earlier work~\cite{AHM07}.}    a similar algorithm  of  Andr\'en, Hellstr\"om and Markstr\"om~\cite{AHM07} that proves it for $\mathfrak{q}=O(1)$; see also~\cite{Ril05} for prior diameter bounds. If $\omega\in (0,1]$ and the $\omega$-snowflake of $\mathsf{SL}_k(\mathbb{F}_\mathfrak{q})$ embeds with quadratic average distortion $D\ge 1$ into a Hilbert space, then by~\eqref{eq:Kassabo and Alon} and~\eqref{eq:root gap diameter omega},
$$
D\gtrsim  \frac{k^{2\omega-\frac12}(\log\mathfrak{q})^\omega}{(\log k)^\omega}.
$$
This bound is vacuous if $\omega<\frac14$. However, if we use  Lemma~\ref{lem:better power of gap}  in place of~\eqref{eq:root gap diameter omega} we get the following lower bound on $D$ which tends to $\infty$ as $k\to\infty$ in the entire range $\omega\in (0,1]$.
$$
D\gtrsim  \frac{k^{\max\left\{2\omega-\frac12,\omega\right\}}(\log\mathfrak{q})^\omega}{(\log k)^\omega}.
$$
\end{example}

\begin{remark}\label{eq"log n 1/q}
 Lemma~\ref{lem:better power of gap} in the case when $\G$ is an expander, namely it is both $O(1)$-regular and $1/(1-\lambda_2(\G))=O(1)$, shows that the case $p=1$ of the first assertion~\eqref{eq:rough D upper bound convexity smoothness} of Theorem~\ref{thm:really main} is sharp for every $q\ge 2$, up to a multiplicative factor which may depend on only $q$. Indeed, take $Y=\ell_q$ and $X=\ell_\infty^n$. Then, the modulus of uniform convexity of $Y$ has power type $q$ and $\cc_Y(X)=n^{1/q}$. By considering the uniform distribution over the image of an isometric embedding of  $(\n,d_\G)$ into $X$, we see from Lemma~\ref{lem:better power of gap} that if $X$ embeds with $q$-average distortion $D$ into $Y$, then necessarily \begin{equation}\label{eq:q snowflake sharp}
D\gtrsim \frac{1}{q}(\log n)^{\frac{1}{q}}\asymp \frac{1}{q}\big(\log \cc_Y(X)\big)^{\frac{1}{q}}.
\end{equation}
\end{remark}

Our proof of Lemma~\ref{lem:better power of gap} uses the following lemma, the case $p>q$ of which is a mixed-exponent variant of Matou\v{s}ek's extrapolation phenomenon for Poincar\'e inequalities~\cite{Mat97,BLMN05,NS11}. One could avoid treating mixed exponents and obtain a ``vanilla'' extrapolation inequality by using~\cite{LS17} (recall also Remark~\ref{rem:LS}), but this leads to an asymptotically worse dependence on the spectral gap even when, say, $p=2$ and $1\le q<2$, which is an inherent deficiency: If one considers the variant of~\eqref{eq:max pq} below with $p=2$ and the $q$'th moment on both sides for some $q\in [1,2)$, then the power of $1/(1-\lambda_2(\A))$ becomes $1/q$ rather than the stated $1/2<1/q$, and this is sharp, for example, when $\A$ is the transition matrix of the standard random walk on  the  $k$-dimensional Hamming cube $\{0,1\}^k$.

\begin{lemma}\label{lem:p2q2} Fix $p,q\ge 1$, $n\in \N$ and $\pi=(\pi_1,\ldots,\pi_n)\in \bigtriangleup^{\!n-1}$. Suppose that $\A=(a_{ij})\in \M_n(\R)$ is a stochastic and $\pi$-reversible matrix. Then, every $x_1,\ldots,x_n\in \ell_p$ satisfy the inequality
\begin{equation}\label{eq:max pq}
\bigg(\sum_{i=1}^n\sum_{j=1}^n \pi_i\pi_j \|x_i-x_j\|_{\ell_p}^q\bigg)^{\!\!\frac{1}{q}}\lesssim \bigg(\frac{p^2+q^2}{1-\lambda_2(\A)}\bigg)^{\!\!\frac{1}{\min\{p,2\}}}\bigg(\sum_{i=1}^n\sum_{j=1}^n \pi_i a_{ij} \|x_i-x_j\|_{\ell_p}^{\max\{p,q\}}\bigg)^{\!\!\frac{1}{\max\{p,q\}}}.
\end{equation}
\end{lemma}

\begin{proof} By the case $X=\R$ (and $p=2$) of the first inequality in~\eqref{eq:lasa good pq}, for every $\beta\ge 2$ we have
\begin{equation}\label{eq:scalar s}
\forall\, s_1,\ldots,s_n\in \R,\qquad \sum_{i=1}^n\sum_{j=1}^n \pi_i\pi_j|s_i-s_j|^\beta\le \bigg(\frac{\beta}{\sqrt{1-\lambda_2(\A)}}\bigg)^{\!\!\beta} \sum_{i=1}^n\sum_{j=1}^n \pi_ia_{ij}|s_i-s_j|^\beta.
\end{equation}
The scalar inequality~\eqref{eq:scalar s} with slightly weaker  constant factor appears in~\cite[Lemma~4.4]{NS11}, as a natural quadratic variant (via a similar proof) of Matou\v{s}ek's extrapolation lemma for Poincar\'e inequalities~\cite{Mat97}, which is the analogous $\ell_1$ statement, namely  with ``spectral gap'' replaced by ``Cheeger constant.'' By a point-wise application of~\eqref{eq:scalar s} followed by integration, we see that that
\begin{equation}\label{eq:L beta}
\forall\, f_1,\ldots,f_n\in L_\beta,\qquad \sum_{i=1}^n\sum_{j=1}^n \pi_i\pi_j\|f_i-f_j\|_{\!L_\beta}^\beta\le \bigg(\frac{\beta}{\sqrt{1-\lambda_2(\A)}}\bigg)^{\!\!\beta} \sum_{i=1}^n\sum_{j=1}^n \pi_ia_{ij}\|f_i-f_j\|_{\!L_\beta}^\beta.
\end{equation}
Since $L_2$ is isometric to a subset of $L_\beta$, it follows from~\eqref{eq:L beta} that also
\begin{equation}\label{eq:L 2}
\forall\, f_1,\ldots,f_n\in L_2,\qquad \sum_{i=1}^n\sum_{j=1}^n \pi_i\pi_j\|f_i-f_j\|_{\!L_2}^\beta\le \bigg(\frac{\beta}{\sqrt{1-\lambda_2(\A)}}\bigg)^{\!\!\beta} \sum_{i=1}^n\sum_{j=1}^n \pi_ia_{ij}\|f_i-f_j\|_{\!L_2}^\beta.
\end{equation}

Suppose first that $q\ge p\ge 2$. Then, \eqref{eq:L beta} with $\beta=p$ is the same as the estimate
\begin{equation}\label{eq:gamma of Lp power p}
\gamma\big(\A,\|\cdot\|_{\!L_p}^p\big)\le \bigg(\frac{p}{\sqrt{1-\lambda_2(\A)}}\bigg)^{\!\!p}.
\end{equation}
We therefore obtain the following bound which implies the desired inequality~\eqref{eq:max pq} when  $q\ge p\ge 2$.
$$
\gamma\big(\A,\|\cdot\|_{\!L_p}^q\big)\stackrel{\eqref{eq:lasa good pq}}{\le} \Big(\frac{2q}{p}\Big)^{\!q}\gamma\big(\A,\|\cdot\|_{\!L_p}^p\big)^{\!\frac{q}{p}}\stackrel{\eqref{eq:gamma of Lp power p}}{\le}\bigg(\frac{2q}{\sqrt{1-\lambda_2(\A)}}\bigg)^{\!\!q}.
$$

If $1\le p\le 2$ and $q\ge p$, then by~\cite{Sch38} there exist $f_1,\ldots,f_n\in L_2$ such that
\begin{equation}\label{use schoenberg}
\forall\, i,j\in \n,\qquad \|f_i-f_j\|_{\!L_2}^{\phantom{\frac{p}{2}}}=\|x_i-x_j\|_{\ell_p}^{\frac{p}{2}}.
\end{equation}
An application of~\eqref{eq:L 2} with $\beta=\frac{2q}{p}\ge 2$ now shows that
\begin{align*}
\bigg(\sum_{i=1}^n\sum_{j=1}^n \pi_i\pi_j \|x_i-x_j\|_{\ell_p}^q\bigg)^{\!\!\frac{1}{q}}&\stackrel{\eqref{use schoenberg}}{=}\bigg(\sum_{i=1}^n\sum_{j=1}^n \pi_i\pi_j \|f_i-f_j\|_{\!L_2}^{\!\frac{2q}{p}}\bigg)^{\!\!\frac{1}{q}}\\&\stackrel{\eqref{eq:L 2}}{\le} \bigg(\frac{2q}{p\sqrt{1-\lambda_2(\A)}}\bigg)^{\!\!\frac{2}{p}}\bigg(\sum_{i=1}^n\sum_{j=1}^n \pi_ia_{ij} \|f_i-f_j\|_{\!L_2}^{\!\frac{2q}{p}}\bigg)^{\!\!\frac{1}{q}}\\&\stackrel{\eqref{use schoenberg}}{\le} \bigg(\frac{2q}{\sqrt{1-\lambda_2(\A)}}\bigg)^{\!\!\frac{2}{p}}\bigg(\sum_{i=1}^n\sum_{j=1}^n \pi_ia_{ij} \|x_i-x_j\|_{\ell_p}^{q}\bigg)^{\!\!\frac{1}{q}}.
\end{align*}
This completes the proof of~\eqref{eq:max pq} in the entire range $q\ge p\ge 1$.

Suppose next that $p\ge q\ge 1$ and $p\ge 2$. Writing $x_i=(x_{i1},x_{i2},\ldots)\in \ell_p$ for each $i\in\n$,
\begin{align}
\nonumber \sum_{i=1}^n\sum_{j=1}^n \pi_i\pi_j \|x_i-x_j\|_{\ell_p}^q&= \sum_{i=1}^n\sum_{j=1}^n \pi_i\pi_j \bigg(\sum_{k=1}^\infty |x_{ik}-x_{jk}|^p\bigg)^{\!\!\frac{q}{p}}\\&\le \bigg(\sum_{k=1}^\infty\sum_{i=1}^n\sum_{j=1}^n \pi_i\pi_j  |x_{ik}-x_{jk}|^p\bigg)^{\!\!\frac{q}{p}}\label{eq:use concavity}\\
&\le \bigg(\sum_{k=1}^\infty\bigg(\frac{ p}{\sqrt{1-\lambda_2(\A)}}\bigg)^{\!\!p}\sum_{i=1}^n\sum_{j=1}^n \pi_ia_{ij} |x_{ik}-x_{jk}|^p\bigg)^{\!\!\frac{q}{p}}\label{eq:use scalar mat}\\&=\bigg(\frac{ p}{\sqrt{1-\lambda_2(\A)}}\bigg)^{\!\!q}\bigg(\sum_{i=1}^n\sum_{j=1}^n \pi_ia_{ij} \|x_{i}-x_{j}\|_{\ell_p}^p\bigg)^{\!\!\frac{q}{p}},\label{eq:coordinate wise}
\end{align}
where~\eqref{eq:use concavity} is a consequence of the concavity (since $q\le p$)  of the function $(u\ge 0)\mapsto u^{\frac{q}{p}}$ and~\eqref{eq:use scalar mat} is a coordinate-wise application (with $\beta=p\ge 2$) of the scalar inequality~\eqref{eq:scalar s}. This establishes~\eqref{eq:max pq} when $p\ge q\ge 1$ and $p\ge 2$. If $1\le q\le p\le 2$, then~\eqref{eq:max pq}  follows by using~\eqref{eq:coordinate wise} with $x_1,\ldots,x_n$ replaced by $f_1,\ldots,f_n$ that satisfy~\eqref{use schoenberg}, with $p$ replaced by $2$ and with $q$ replaced by $\frac{2q}{p}\le 2$.
\end{proof}

\begin{proof}[Proof of Lemma~\ref{lem:better power of gap}]  By assumption, there exists an embedding $f:\n\to \ell_p$ such that
\begin{equation}\label{eq:D omega at end}
 \forall\, i,j\in \n,\qquad \|f(i)-f(j)\|_{\ell_p}\le Dd_\G(i,j)^\omega.
\end{equation}
and
\begin{equation}\label{eq:q omega at end}
\frac{1}{n^2}\sum_{i=1}^n\sum_{j=1}^n \|f(i)-f(j)\|_{\ell_p}^q\ge \frac{1}{n^2}\sum_{i=1}^n\sum_{j=1}^n d_\G(i,j)^{q\omega}
\end{equation}
Our task is to bound $D$ from below. Using Lemma~\ref{lem:p2q2} with $\A=\A_\G\in \M_n(\R)$ the adjacency matrix  of $\G$,  $\pi=(\frac{1}{n},\ldots,\frac{1}{n})\in \bigtriangleup^{\!n-1}$ ($\G$ is a regular graph) and $x_i=f(i)$ for all $i\in \n$, we see that
\begin{align*}
\bigg(\frac{1}{n^2}\sum_{i=1}^n\sum_{j=1}^n d_\G(i,j)^{q\omega}\bigg)^{\!\!\frac{1}{q}}&\stackrel{\eqref{eq:q omega at end}}{\le} \bigg(\frac{1}{n^2}\sum_{i=1}^n\sum_{j=1}^n \|f(i)-f(j)\|_{\ell_p}^q\bigg)^{\!\!\frac{1}{q}}\\
&\stackrel{\eqref{eq:max pq}}{\lesssim} \bigg(\frac{p^2+q^2}{1-\lambda_2(\G)}\bigg)^{\!\!\frac{1}{\min\{p,2\}}}\bigg(\frac{1}{|E_\G|}\sum_{\{i,j\}\in E_\G}^n \|f(i)-f(j)\|_{\ell_p}^{\max\{p,q\}}\bigg)^{\!\!\frac{1}{\max\{p,q\}}}\\&\stackrel{\eqref{eq:D omega at end}}{\le} \bigg(\frac{p^2+q^2}{1-\lambda_2(\G)}\bigg)^{\!\!\frac{1}{\min\{p,2\}}} D.
\end{align*}
This is the desired bound~\eqref{eq:lower D min} when $\omega\ge 1/\min\{p,2\}$. Note that thus far we used~\eqref{eq:D omega at end} only for those  $i,j\in \n$ such that $\{i,j\}\in E_\G$. In other words, we derived~\eqref{eq:lower D min} under the assumption that $f$ is $D$-Lipschitz rather than that $f$ is $\omega$-H\"older with constant $D$, which is a more stringent requirement as $d_\G$ takes values in $ [1,\infty)\cup \{0\}$. To prove~\eqref{eq:lower D min} when $\omega\le 1/\min\{p,2\}$ we will probe larger distances  in $(\n,d_\G)$ for which the full H\"older condition~\eqref{eq:D omega at end} gives more information.

Denote
\begin{equation}\label{eq;s choice end}
s\eqdef \left\lceil\frac{1}{1-\lambda_2(\G)}\right\rceil.
\end{equation}
The function $t\mapsto \big(\frac{1+t}{2}\big)^{\frac{1}{1-t}}$ is increasing on $(-1,1)$ and tends to $\frac{1}{\sqrt{e}}$ as $t\to 1^-$. Hence,
\begin{equation}\label{eq:1/sqrt e}
\bigg(\frac{1+\lambda_2(\G)}{2}\bigg)^{\!\!s}\le \bigg(\frac{1+\lambda_2(\G)}{2}\bigg)^{\!\!\frac{1}{1-\lambda_2(\G)}}\le \frac{1}{\sqrt{e}}.
\end{equation}
Using Lemma~\ref{lem:p2q2} with $\A=\Big(\frac12\I_n+\frac12 \A_\G\Big)^{\!s}$, we therefore see that
\begin{align*}
&\bigg(\frac{1}{n^2}\sum_{i=1}^n\sum_{j=1}^n d_\G(i,j)^{q\omega}\bigg)^{\!\!\frac{1}{q}}\\&\quad\,\stackrel{\eqref{eq:q omega at end}}{\le} \bigg(\frac{1}{n^2}\sum_{i=1}^n\sum_{j=1}^n \|f(i)-f(j)\|_{\ell_p}^q\bigg)^{\!\!\frac{1}{q}}\\
&\quad\, \stackrel{\eqref{eq:max pq}}{\lesssim} \bigg(\frac{p^2+q^2}{1-\lambda_2\left(\left(\frac12\I_n+\frac12 \A_\G\right)^{\!s}\right)}\bigg)^{\!\!\frac{1}{\min\{p,2\}}}\bigg(\frac{1}{n}\sum_{i=1}^n \sum_{j=1}^n \Big(\frac12\I_n+\frac12 \A_\G\Big)^{\!s}_{\!ij}\|f(i)-f(j)\|_{\ell_p}^{\max\{p,q\}}\bigg)^{\!\!\frac{1}{\max\{p,q\}}}\\&\stackrel{\eqref{eq:D omega at end}\wedge \eqref{eq:1/sqrt e}}{\lesssim} \big(p^2+q^2\big)^{\!\frac{1}{\min\{p,2\}}} s^\omega D\stackrel{\eqref{eq;s choice end}}{\asymp} \frac{\big(p^2+q^2\big)^{\!\frac{1}{\min\{p,2\}}}}{\big(1-\lambda_2(\G)\big)^{\!\omega}}D,
\end{align*}
where the penultimate step uses  that if for some $i,j\in \n$ the $(i,j)$-entry of $(\frac12\I_n+\frac12 \A_\G)^{s}$ is nonzero, then there is a walk in $\G$ of length at most $s$ from $i$ to $j$, hence $d_\G(i,j)\le s$.
\end{proof}

We end this section with a few further remarks and open questions.

\begin{proposition}\label{prop:av kq} Fix $\mathfrak{q},k\in \N$ such that $\mathfrak{q}$ is  a power of a prime. Let $\mathsf{Av}(k,\mathfrak{q})$ denote the smallest $D\ge 1$ such that $\mathsf{SL}_k(\mathbb{F}_\mathfrak{q})$ embeds into a Hilbert space with average distortion $D$. Then
$$
\mathsf{Av}(k,\mathfrak{q})\asymp (\log\mathfrak{q})\frac{k^{\frac32}}{\log k}\asymp \sqrt[4]{\log \mathfrak{q}}\cdot \frac{(\log |\mathsf{SL}_k(\mathbb{F}_\mathfrak{q})|)^{\frac34}}{\log\log |\mathsf{SL}_k(\mathbb{F}_\mathfrak{q})|-\log\log \mathfrak{q}}.
$$
\end{proposition}

\begin{proof} A substitution of~\eqref{eq:Kassabo and Alon} into~\eqref{eq:diam to omega} when $\omega=q=1$ and $p=2$ gives $\mathsf{Av}(k,\mathfrak{q})\gtrsim  \frac{k^{\frac32}\log \mathfrak{q}}{\log k}$. To prove the matching upper bound, suppose that $\mathfrak{q}=\mathfrak{p}^m$ for some prime $\mathfrak{p}$ and $m\in \N$. Let $v_1,\ldots,v_m$ be a basis of $\mathbb{F}_{\mathfrak{q}}$ over $\mathbb{F}_{\mathfrak{p}}$. Thus, for every $x\in \mathbb{F}_{\mathfrak{q}}$ there are unique $\chi_1(x),\ldots,\chi_m(x)\in \Z/\mathfrak{p}\Z$ such that $x=\chi_1(x)v_1+\ldots+\chi_m(x)v_m$. Define an embedding
$$
f:\mathsf{SL}_k(\mathbb{F}_\mathfrak{q})\to \underbrace{\M_k(\C)\oplus\ldots\oplus \M_k(\C)}_{\text{$m$ times}}\cong \ell_2^{2mk^2}
$$
by setting for some $C>0$,
$$
\forall\, \mathsf{X}=(x_{jk})\in \mathsf{SL}_k(\mathbb{F}_\mathfrak{q}),\qquad  f(\mathsf{X})\eqdef \bigoplus_{s=1}^m \frac{Ck\log q}{\sqrt{m}\log k}\Big(e^{\frac{2\pi \mathsf{i}}{\mathfrak{p}}\chi_s(x_{jk})}\Big)_{(j,k)\in \k^2},
$$
We claim that if  $C$ is a sufficiently large universal  constant, then $f$ exhibits that    $\mathsf{Av}(k,\mathfrak{q})\lesssim  \frac{k^{\frac32}\log \mathfrak{q}}{\log k}$.

Fix distinct indices $\alpha,\beta\in \k$. Then for every  $\mathsf{X}=(x_{jk})\in \mathsf{SL}_k(\mathbb{F}_\mathfrak{q})$ we have $$f\Big(\mathsf{X}\big(\I_k\pm\mathsf{E}_k(\alpha,\beta)\big)\Big) -f(\mathsf{X})=\bigoplus_{s=1}^m \frac{Ck\log q}{\sqrt{m}\log k}\Big(\d_{s\beta}e^{\frac{2\pi \mathsf{i}}{\mathfrak{p}}\chi_s(x_{j\beta})} \Big(e^{\pm\frac{2\pi \mathsf{i}}{\mathfrak{p}}\chi_s(x_{j\alpha})}-1\Big)\Big)_{(j,k)\in \k^2},$$
where $\d_{s\beta}$ is the Kronecker delta. Thus,
$$
\Big\|f\Big(\mathsf{X}\big(\I_k\pm\mathsf{E}_k(\alpha,\beta)\big)\Big) -f(\mathsf{X})\Big\|_{\ell_2^{2mk^2}}\le \frac{Ck\log q}{\sqrt{m}\log k}\cdot 2\sqrt{km}=\frac{2k^{\frac32}\log \mathfrak{q}}{\log k}.
$$
By the definition of the (word) metric on $\mathsf{SL}_k(\mathbb{F}_\mathfrak{q})$, this means that $f$ is $\frac{2k^{\frac32}\log \mathfrak{q}}{\log k}$-Lipschitz.

In the reverse direction, if $\mathsf{X}=(x_{jk}),\mathsf{Y}=(y_{jk})$ are independent and chosen uniformly at random from $\mathsf{SL}_k(\mathbb{F}_\mathfrak{q})$, then with probability that is bounded below by a positive universal constant, we have $|\exp(2\pi \mathsf{i}\chi_s(x_{jk}))/\mathfrak{p}-\exp(2\pi \mathsf{i}\chi_s(y_{jk})/\mathfrak{p})|\gtrsim 1$ for a universal constant fraction of the $mk^2$ triples $(i,j,s)\in \k\times\k\times \{1,\ldots,m\}$. Therefore,
\begin{multline*}
\frac{1}{|\mathsf{SL}_k(\mathbb{F}_\mathfrak{q})|^2}\sum_{\mathsf{X,Y}\in \mathsf{SL}_k(\mathbb{F}_\mathfrak{q})} \|f(\mathsf{X})-f(\mathsf{Y})\|_{\ell_2^{2mk^2}}\gtrsim \frac{Ck\log q}{\sqrt{m}\log k}\cdot \sqrt{mk^2}=\frac{Ck^2\log \mathsf{q}}{\log k}\\\stackrel{\eqref{eq:Kassabo and Alon}}{\gtrsim} C\diam\big(\mathsf{SL}_k(\mathbb{F}_\mathfrak{q})\big)\ge\frac{C}{|\mathsf{SL}_k(\mathbb{F}_\mathfrak{q})|^2}\sum_{\mathsf{X,Y}\in \mathsf{SL}_k(\mathbb{F}_\mathfrak{q})} d_{\mathsf{SL}_k(\mathbb{F}_\mathfrak{q})}(\mathsf{X},\mathsf{Y}).\tag*{\qedhere}
\end{multline*}
\end{proof}

The following conjecture asserts that (at least for fixed $\mathfrak{q}$) the curious-looking but nonetheless sharp asymptotic behavior of Proposition~\ref{prop:av kq}  holds also for bi-Lipschitz embeddings; we suspect that its resolution is tractable, perhaps via the representation-theoretic approach of~\cite{ANV10}.

\begin{conjecture}  For every $k\in \N$ and prime power $\mathfrak{q}$ we have $\cc_{\ell_2}\big(\mathsf{SL}_k(\mathbb{F}_\mathfrak{q})\big)\asymp_\mathfrak{q}\frac{k^{\frac32}}{\log k}$.
\end{conjecture}

\begin{remark}\label{rem:Qcube and others}
By~\eqref{eq:C2} and~\eqref{eq:diam for transitive}  we see that if $\G=(\n,E_\G)$ is a vertex-transitive graph, then
\begin{equation}\label{eq:woith diam}
\forall\, D\ge1,\qquad \dim_D(\G)\ge e^{\frac{c}{D}(1-\lambda_2(\G))\diam(\G)},
 \end{equation}
 where $c>0$ is a universal constant, and the notation $\dim_D(\cdot)$ of~\cite[Definition~2.1]{LLR95}  was recalled in Section~\ref{sec:history}. In fact, this reasoning shows that (Theorem~\ref{thm:average john} implies that) if $(\n,d_\G)$ embeds into a normed space $X$ with average distortion $D$ (rather than the stronger bi-Lipschitz distortion $D$ to which~\eqref{eq:woith diam} alludes), then necessarily $\dim(X)\ge \exp(c(1-\lambda_2(\G))\diam(\G)/D)$. It follows in particular  from~\eqref{eq:woith diam} that if $(\n,d_\G)$ embeds with average distortion $O(1)$ into some normed space of dimension $(\log n)^{O(1)}$, then necessarily $(1-\lambda_2(\G))\diam(\G)\lesssim \log\log n$.

  There are many examples of Cayley graphs $\G=(\n,E_\G)$ for which $\lambda_2(\G)=1-\Omega(1)$ and $\diam(\G)\gtrsim \log n$ (see e.g.~\cite{AR94,NR09}). In all such examples, \eqref{eq:woith diam} asserts that $\dim_D(\G)\gtrsim n^{c/D}$ for some universal constant $c>0$. The Cayley graph that was studied in~\cite{KN06} (a quotient of the Hamming cube by a good code) now shows that there exist arbitrarily large $n$-point metric spaces $\cM_n$ with $\dim_1(\cM_n)\lesssim \log n$  (indeed, $\cM_n$ embeds isometrically into $\ell_1^{k}$ for some $k\lesssim \log n$), yet $\cM_n$ has a $O(1)$-Lipschitz quotient (see~\cite{BJLPS99} for the relevant definition)  that does not embed with distortion $O(1)$ into any normed space of dimension $n^{o(1)}$. To the best of our knowledge, it wasn't previously known that the metric dimension $\dim_D(\cdot)$ can become asymptotically larger (and even increase exponentially)  under Lipschitz quotients, which is yet another major departure from the linear theory, in contrast to what one would normally predict in the context of the Ribe program.
\end{remark}

\begin{remark}\label{rem:must change norm}
Let $\G=(\n,E_\G)$ be a Cayley graph of a finite group with $\lambda_2(\G)=1-\Omega(1)$. The metric space $(\n,d_\G)$  embeds with bi-Lipschitz distortion $\diam(\G)$ into $\ell_2^{n-1}$ by considering any bijection between $\n$ and the vertices of the $n$-simplex. There is therefore no a priori reason why it wouldn't be possible to embed $(\n,d_\G)$ with bi-Lipschitz distortion $O(1)$ into some normed space $(X,\|\cdot\|_X)$ whose bi-Lipschitz distortion from a Hilbert space is at least a sufficiently large constant multiple of $\diam (\G)$. But this is not so if $\diam(\G)$ is large. Indeed, for every $\cc>\cc_{\ell_2}(X)$ and $D>\cc_X(\n,d_\G)$ by Theorem~\ref{thm:really main}  the $\frac12$-snowflake of $(\n,d_\G)$ embeds into $\ell_2$ with quadratic average distortion that is at most a universal constant multiple of $\sqrt{D\log(\cc+1)}$. By contrasting this with the case $\omega=\frac12$ of~\eqref{eq:root gap diameter omega}, it follows that
$$
\cc_X(\n,d_\G)\gtrsim \frac{\diam(\G)}{\log(\cc_{\ell_2}(X)+1)}.
$$
Thus, even if we allow  $\cc_{\ell_2}(X)$ to be as large as $\diam(\G)^{O(1)}$, then any embedding of $(\n,d_\G)$ into $X$ incurs distortion that is at least a positive universal constant multiple of $\diam(\G)/\log\diam(\G)$.
\end{remark}

Substituting~\eqref{eq:Kassabo and Alon} into~\eqref{eq:woith diam}  gives the following noteworthy corollary. It shows that even though elements of $\mathsf{SL}_k(\mathbb{F}_\mathfrak{q})$ have a representation using $k^2$ coordinates (over $\mathbb{F}_\mathfrak{q}$, thus using $\mathfrak{q}k^2$ bits), if one wishes to realize its geometry  with bounded (average) distortion as a subset of the ``commutative'' geometry of a  normed space, then the dimension of that space must be exponentially large.

\begin{corollary}\label{cor:k/log k}
Fix $\mathfrak{q},k\in \N$ such that $\mathfrak{q}$ is  a power of a prime. For $D\ge 1$ let $\dim_D(k,\mathfrak{q})$ denote the smallest $d\in \N$ such that $\mathsf{SL}_k(\mathbb{F}_\mathfrak{q})$ embeds into some $d$-dimensional normed space $X_{k,\mathfrak{q}}$ with bi-Lipschitz distortion $D$. Then, for some universal constant $c>0$ we have
$$
\dim_D(k,\mathfrak{q})\ge \mathfrak{q}^{\frac{ck}{D\log k}}.
$$
This holds even if we only require that the low-dimensional embedding has average distortion $D$.
\end{corollary}

The following conjecture asserts that (for fixed $\mathfrak{q}$) Corollary~\ref{cor:k/log k} is sharp. Given the lower bound that we obtained here, it remains to construct a $O(1)$-distortion embedding of $\mathsf{SL}_k(\mathbb{F}_\mathfrak{q})$  into some  low-dimensional normed space $X$. Here, ``low-dimensional'' means that the dimension of $X$ grows exponentially in $k/\log k$ rather than exponentially in $k^2$ as in Fr\'echet's embedding. We suspect that, beyond its intrinsic interest, such a low-dimensional realization of $\mathsf{SL}_k(\mathbb{F}_\mathfrak{q})$ will be useful elsewhere.

\begin{conjecture}[dimension reduction for $\mathsf{SL}_k(\mathbb{F}_\mathfrak{q})$] For every prime power $\mathfrak{q}$ there exist $D=D(\mathfrak{q})\ge 1$ and $c=c(\mathfrak{q}),C=C(\mathfrak{q})>0$ such that for every integer $k\ge 2$ we have
$$
e^{\frac{ck}{\log k}}\le \dim_D(k,\mathfrak{q})\le e^{\frac{Ck}{\log k}}.
$$
\end{conjecture}

\section{Proof of Theorem~\ref{thm:full duality}}\label{sec:duality}

For a metric space $(\MM,d_\MM)$, a Banach space $(Y,\|\cdot\|_Y)$ and $\omega\in (0,1]$, following the notation of~\cite{Mat90,BS02,LN05,NR17} we consider a quantity $\ee^\omega(\MM,Y)$, called the $\omega$-H\"older extension modulus of the pair $(\MM,Y)$, which is defined as the infimum over those $L\in [1,\infty]$ such that for every subset $S\subset \MM$ and every mapping $\phi:S\to Y$ which is $\omega$-H\"older with constant $1$, i.e., $\|f(x)-f(y)\|_Y\le d_\MM(x,y)^\omega$ for all $x,y\in \MM$, there exists  $\Phi:\MM\to Y$ that extends $\phi$, i.e., $\Phi(s)=\phi(s)$ for all $s\in S$, and $\Phi$ is $\omega$-H\"older with constant $L$. When $\omega=1$ one uses the simpler notation $\ee^1(\MM,T)=\ee(\MM,Y)$. Note that $\ee^\omega(\MM,Y)=\ee(\MM^\omega,Y)$, where henceforth $\MM^\omega$ denotes the $\omega$-snowflake of $(\MM,d_\MM)$.  Thus, one could work throughout (both in the present  context and elsewhere) with the more classical Lipschitz extension modulus $\ee(\cdot,\cdot)$, but  it is beneficial to use the above notation for $\omega$-H\"older extension. Such extension moduli have been studied extensively in the literature; see e.g.~\cite{LN05,BB12,NR17} and the references therein for an indication of the large amount of work that has been done on this topic.

The following powerful extension theorem is a combination of known results. Its special case $p=q=2$  is a combination of Ball's deep work~\cite{Bal92} on Lipschitz extension and our solution~\cite{NPSS06} in collaboration with Peres, Schramm and Sheffield of Ball's Markov type $2$ problem~\cite{Bal92}. Also, its special case when $p=1$ and $Y$ is a Hilbert space is a  theorem of Minty~\cite{Min70}, which relies on Kirszbraun's important theorem~\cite{Kirsz34}. Its statement in full generality follows from the generalization of the above results that appears in our work with Mendel~\cite{MN13-bary}. A special  case of Theorem~\ref{thm:holder extension} was discussed  in~\cite{Nao15}; for ease of later use (below and elsewhere), it is worthwhile to formulate the full statement here and explain  its quick derivation from results in the literature.

\begin{theorem}[Ball's extension phenomenon for H\"older functions]\label{thm:holder extension}Fix  $p,q>0$ with $q\ge \max\{p,2\}$. Write $\omega=p/q$.  Let $(\MM,d_\MM)$ be a metric space that has Markov type $p$ and let $(Y,\|\cdot\|_Y)$ be a Banach space whose modulus of uniform convexity has power type $q$. Then, $\ee^\omega(X,Y)\lesssim \mathbf{M}_p(\MM)^\omega\mathscr{K}_q(Y)$.
\end{theorem}

\begin{proof}  As $Y$ is uniformly convex, it is  reflexive. Hence, $\ee(\NN,Y)\lesssim \mathbf{M}_q(\NN)\mathscr{K}_q(Y)$ for any metric space $(\NN,d_\NN)$, by combining~\cite[Theorem~6.10]{MN14} and~\cite[Theorem~1.11]{MN13-bary} (see the discussion in Section~1.5 of~\cite{MN13-bary}). By definition, we have $\ee^\omega(\MM,Y)=\ee(\MM^\omega,Y)$ and $\mathbf{M}_q(\MM^\omega)\le \mathbf{M}_p(\MM)^\omega$. Consequently,  \begin{equation*}
\ee^\omega(\MM,Y)\lesssim \mathbf{M}_q(\MM^\omega)\mathscr{K}_q(Y)\le \mathbf{M}_p(\MM)^\omega\mathscr{K}_q(Y).\tag*{\qedhere}
\end{equation*}
\end{proof}

Note that due to Theorem~\ref{thm:quote NPSS}, if we are in the setting of Theorem~\ref{thm:holder extension} and  $(X,\|\cdot\|_X)$ is a Banach space whose modulus of uniform smoothness has power type $p$, then $\ee^\omega(X,Y)\lesssim \mathscr{S}_p(X)^\omega\mathscr{K}_q(Y)$.

The proof of Lemma~\ref{lem:reduction to finite rational} below is a natural (a bit tedious) discretization/dominated convergence argument; we include it for the sake of completeness, but it could be skipped and left as a technical exercise. In what follows, it is convenient to use the (ad hoc) terminology that a measure $\mu$ on a set $\Omega$ is rational if it is finitely supported and $\mu(\{x\})\in \Q$ for every $x\in \Omega$.

\begin{lemma}[compactness]\label{lem:reduction to finite rational} Fix $p\ge 1$ and $D,\alpha>1$. Let $(\MM,d_\MM)$ be a separable infinite metric space and let $(Y,\|\cdot\|_Y)$ be a Banach space. Suppose that  for any rational probability measure $\rho$ on $\MM$ the metric probability space $(\supp(\rho),d_\MM,\rho)$ embeds with $p$-average distortion $D$ into $Z$. Then, $(\MM,d_\MM)$ embeds with $p$-average distortion $\alpha D$ into the ultrapower $Y^\mathscr{U}$ for any non-principal ultrafilter $\mathscr{U}$ on $\N$. Also, $(\MM,d_\MM)$ embeds with $p$-average distortion $\alpha\ee(X,Y) D$ into $Y$.
\end{lemma}

The proof of Theorem~\ref{thm:full duality} is a direct application of Lemma~\ref{lem:reduction to finite rational} to Theorem~\ref{thm:quote duality} below, which is (the nontrivial direction of a) the duality result of~\cite[Theorem~1.3]{Nao14}, while using Theorem~\ref{thm:holder extension} to justify the second assertion of Theorem~\ref{thm:full duality}. The term ``duality'' here indicates that the existence of the embedding that Theorem~\ref{thm:quote duality} asserts is proved in~\cite{Nao14} by a separation (Hahn--Banach) argument.

\begin{theorem}\label{thm:quote duality} Let $(\MM,d_\MM)$ be a metric space and $(Y,\|\cdot\|_Y)$ a Banach space. Suppose that there exist $p,K\ge 1$ such that $\gamma(\A,d_\MM^p)\le K\gamma(\A,\|\cdot\|_Y^p)$ for every $n\in \N$ and any symmetric stochastic matrix $\A\in \M_n(\R)$. Then, for every $D>K$ and every rational probability measure $\rho$ on $\MM$, the (finite) metric measure space $(\supp(\rho),d_\MM,\rho)$ embeds with $p$-average distortion $D$ into $\ell_p(Y)$.
\end{theorem}

\begin{proof}[Proof of Lemma~\ref{lem:reduction to finite rational}] For each $\d>0$ fix an arbitrary $\d$-net $\{x_i^\d\}_{i=1}^\infty$ of $(\MM,d_\MM)$, i.e., $d_\MM(x_i^\d,x_j^\d)\ge \d$ for all distinct $i,j\in \N$, and also $\bigcup_{i=1}^\infty B_\MM(x_i^\d,\d)=\MM$. Define inductively $V_1^\d=B_\MM(x_1^\d,\d)$ and $V_{j+1}^\d=B_\MM(x_{j+1}^\d,\d)\setminus \bigcup_{i=1}^{j}B_\MM(x_j^\d,\d)$ for  $j\in \N$, namely $\{V_{j}^\d\}_{j=1}^\infty$ is the disjoint Voronoi tessellation of $\MM$ that is induced by the (ordered) $\d$-net $\{x_i^\d\}_{i=1}^\infty$.

Fix from now on a Borel probability measure $\mu$ on $\MM$. Let $(W,\|\cdot\|_W)$ be any Banach space (we will eventually take $W$ to be either $Y^\mathscr{U}$ or $Y$). Suppose that there is some $\lambda>0$ such that for every $\d>0$ and $n\in \N$ there exists a $\lambda$-Lipschitz mapping $f_n^\d:\MM\to W$ that satisfies
\begin{equation}\label{eq:W unified}
\sum_{i=1}^n \sum_{j=1}^n \|f_n^\d(x_i^\d)-f_n^\d(x_j^\d)\|_{W}^p\mu(V_i^\d)\mu(V_j^\d)\ge \sum_{i=1}^n \sum_{j=1}^n d_\MM(x_i^\d,x_j^\d)^p\mu(V_i^\d)\mu(V_j^\d).
\end{equation}
We will next show that this assumption formally implies that for any $\Lambda>\lambda$, the metric probability space $(\MM,d_\MM,\mu)$ embeds with $p$-average distortion $\Lambda$ into $(W,\|\cdot\|_W)$.

As justified in the beginning of Section~\ref{sec:aux em}, we may assume that $\iint_{\MM\times \MM} d_\MM(x,y)^p\ud\mu(x)\ud\mu(y)<\infty$. If we could prove that
\begin{equation}\label{eq:sup lower fnd}
\sup_{\substack{\d>0\\n\in \N}}\iint_{\MM\times \MM} \|f_n^\d(x)-f_n^\d(y)\|_W^p\ud\mu(x)\ud\mu(y)\ge \iint_{\MM\times \MM} d_\MM(x,y)^p\ud\mu(x)\ud\mu(y),
\end{equation}
then for some $\d>0$ and $n\in \N$ the normalized mapping $g_n^\d\eqdef \frac{\Lambda}{\lambda}f_n^\d: \MM\to W$
would be $\Lambda$-Lipschitz and satisfy $\iint_{\MM\times \MM} \|g_n^\d(x)-g_n^\d(y)\|_W^p\ud\mu(x)\ud\mu(y)\ge\iint_{\MM\times \MM} d_\MM(x,y)^p\ud\mu(x)\ud\mu(y)$, as required.

To prove~\eqref{eq:sup lower fnd}, note that because $\{x_s^\d\}_{s=1}^\infty$ is $\d$-dense in $\MM$ we have
\begin{equation}\label{eq:eta density}
\sup_{\substack{i,j\in \N\\ (x,y)\in V_i^\d\times V_j^\d}} |d_\MM(x_i^\d,x_j^\d)- d_\MM(x,y)|\le \sup_{\substack{i,j\in \N\\ (x,y)\in V_i^\d\times V_j^\d}} \big(d_\MM(x_i^\d,x)+ d_\MM(x_j^\d,y)\big)\le 2\d.
\end{equation}
Hence, the $\R$-valued function on $\MM\times \MM$ that is equal to $d_\MM(x_i^\d,x_j^\d)^p$ on $V_i\times V_j$ for each $i,j\in \N$ tends point-wise to $d_\MM^p:\MM\times \MM\to \R$ as $\d\to 0$, and it is bounded from above by the $(\mu\times \mu)$-integrable function $(x,y)\mapsto (d_\MM(x,y)+2\d)^p$. By the dominated convergence theorem we therefore have
\begin{align*}
\begin{split}
\iint_{\MM\times \MM} d_\MM(x,y)^p\ud\mu(x)\ud\mu(y)&=\lim_{\d\to 0}\sum_{i=1}^\infty\sum_{j=1}^\infty d_\MM(x_i^\d,x_j^\d)^p\mu(V_i^\d)\mu(V_j^\d)
\\&\!\!\stackrel{\eqref{eq:W unified}}{\le}\limsup_{\d\to 0}\limsup_{n\to\infty}\sum_{i=1}^n \sum_{j=1}^n \|f_n^\d(x_i^\d)-f_n^\d(x_j^\d)\|_{W}^p\mu(V_i^\d)\mu(V_j^\d)\\
&\le \sup_{\substack{\d>0\\n\in \N}}\sum_{i=1}^\infty \sum_{j=1}^\infty \|f_n^\d(x_i^\d)-f_n^\d(x_j^\d)\|_{W}^p\mu(V_i^\d)\mu(V_j^\d).
\end{split}
\end{align*}
The desired statement~\eqref{eq:sup lower fnd} would therefore follow if we could show that for every fixed $n\in \N$,
\begin{equation}\label{for second dominATED}
\lim_{\d\to 0}\bigg(\iint_{\MM\times \MM} \|f_n^\d(x)-f_n^\d(y)\|_W^p\ud\mu(x)\ud\mu(y)-\sum_{i=1}^\infty \sum_{j=1}^\infty \|f_n^\d(x_i^\d)-f_n^\d(x_j^\d)\|_{W}^p\mu(V_i^\d)\mu(V_j^\d)\bigg)=0.
\end{equation}

To justify~\eqref{for second dominATED}, for  $n\in \N$ and $\d>0$  consider the function $h_n^\d:\MM\times \MM\to \R$ that is defined by
$$
\forall\, i,j\in \N,\ \forall (x,y)\in V_i^\d\times V_j^\d,\qquad h_n^\d(x,y)\eqdef\|f_n^\d(x)-f_n^\d(y)\|_W^p-\|f_n^\d(x_i^\d)-f_n^\d(x_j^\d)\|_{W}^p.
$$
Under this
notation, the assertion of~\eqref{for second dominATED} is the same as $\lim_{\d\to 0}\iint_{\MM\times\MM} h_n^\d(x,y)\ud\mu(x)\ud\mu(y)=0$. Since $f_n^\d$ is assumed to be $\lambda$-Lipschitz,  if  $(x,y)\in V_i^\d\times V_j^\d$ for some  $i,j\in \N$, then we have
\begin{multline*}
|h_n^\d(x,y)|\le \|f_n^\d(x)-f_n^\d(y)\|_W^p+\|f_n^\d(x_i^\d)-f_n^\d(x_j^\d)\|_{W}^p\\\le \lambda^p \big(d_\MM(x,y)^p+d_\MM(x_i^\d,x_j^\d)^p\big)
\stackrel{\eqref{eq:eta density}}{\le} \lambda^p \big(d_\MM(x,y)^p+(d_\MM(x,y)+2\d)^p\big).
\end{multline*}
By the dominated convergence theorem, it therefore suffices to show that $\lim_{\d\to 0} h_n^\d=0$ point-wise. This is so since if $\d\le 1$ and $(x,y)\in V_i^\d\times V_j^\d$ for $i,j\in \N$, then $\|f_n^\d(x)-f_n^\d(y)\|_W\le \lambda d_\MM(x,y)$ and
$$
\big|\|f_n^\d(x)-f_n^\d(y)\|_W-\|f_n^\d(x_i^\d)-f_n^\d(x_j^\d)\|_{W}\big|\le \|f_n^\d(x)-f_n^\d(x_i^\d)\|_W+\|f_n^\d(y)-f_n^\d(x_j^\d)\|_W\le 2\lambda\d\le 2\lambda.
$$
Hence, both of the numbers  $\|f_n^\d(x)-f_n^\d(y)\|_W$ and $\|f_n^\d(x_i^\d)-f_n^\d(x_j^\d)\|_{W}$ belong to the bounded interval $[0,\lambda d_\MM(x,y)+2\lambda]$ and are within $2\lambda\d$ of each other. By the uniform continuity of $t\mapsto t^p$ on $[0,\lambda d_\MM(x,y)+2\lambda]$, it follows that $\lim_{\d\to 0} (\|f_n^\d(x)-f_n^\d(y)\|_W^p-\|f_n^\d(x_i^\d)-f_n^\d(x_j^\d)\|_{W}^p)=0$, as required.

By the above considerations, it remains to establish, for fixed $\d>0$ and $n\in \N$, the existence of $f_n^\d$ when $W=Y^\mathscr{U}$ or $W=Y$, and $\lambda$ is less than $\alpha D$ or $\ee(X,Y)\alpha D$, respectively. The case $W=Y$ is a direct consequence of the definition of $\ee(X,Y)$ and the assumption of Lemma~\ref{lem:reduction to finite rational}. Indeed, recalling that $\alpha>1$, for each $i\in \n$ choose $\rho_i\in \Q$ satisfying  $\mu(V_i^\d)\le \rho_i\le \alpha^{\frac{p}{4}}\mu(V_i^\d)$.  By assumption, there is a $D$-Lipschitz mapping $\f: \{x_i^\d:\ i\in \n\ \wedge \ \rho_i>0\}\to Y$ that satisfies
$
\sum_{i=1}^n \sum_{j=1}^n \|\f(x_i^\d)-\f(x_j^\d)\|_{\!Y}^p\rho_i\rho_j\ge \sum_{i=1}^n \sum_{j=1}^n d_\MM(x_i^\d,x_j^\d)^p\rho_i\rho_j.
$
Therefore,
\begin{equation}\label{eq:alpha 23}
\alpha^{\frac{p}{2}}\sum_{i=1}^n \sum_{j=1}^n \|\f(x_i^\d)-\f(x_j^\d)\|_{\!Y}^p\mu(V_i^\d)\mu(V_j^\d)\ge \sum_{i=1}^n \sum_{j=1}^n d_\MM(x_i^\d,x_j^\d)^p\mu(V_i^\d)\mu(V_j^\d).
\end{equation}
Extend $\f$ to a function $\Phi:\MM\to Y$ which is $\ee(X,Y)\sqrt[4]{\alpha}D$-Lipschitz. Then, $f_n^\d= \sqrt{\alpha}\Phi$ has Lipschitz constant less than $\ee(X,Y)\alpha D$ and, by virtue of~\eqref{eq:alpha 23}, it satisfies the desired estimate~\eqref{eq:W unified}.

For the remaining case, namely when $W=Y^\mathscr{U}$ and $\lambda<\alpha D$,
 fix $n\in \N$ and $\{y_j\}_{j=1}^\infty\subset \MM\setminus \{x_i^\d\}_{i=1}^n$ such that $\{x_i^\d\}_{i=1}^n\cup \{y_j\}_{j=1}^\infty$ is dense in $\MM$. Fix also $k\in \N$ and $\eta\in (0,1)\cap \Q$, and define a measure $\nu$ on $\mathcal{F}=\{x_i^\d\}_{i=1}^n\cup\{y_j\}_{j=1}^k$ by setting $\nu(x_i^\d)\in \Q$ to be any rational number satisfying $(1-\eta)\mu(V_i^\d)\le \nu(x_i^\d)< \mu(V_i^\d)$ if $\mu(V_i^\d)>0$, and $\nu$ assigns mass $\eta$ to all the other points in $\mathcal{F}$. By assumption, there is a $D$-Lipschitz mapping $\psi:\mathcal{F}\to Y$ such that
\begin{equation*}\label{eq:this1}
\sum_{u\in \mathcal{F}}\sum_{v\in \mathcal{F}}\|\psi(u)-\psi(v)\|_{\!Y}^p\nu(u)\nu(v)\ge \sum_{u\in \mathcal{F}}\sum_{v\in \mathcal{F}}d_\MM(u,v)^p\nu(u)\nu(v).
\end{equation*}
Then, since $\psi$ is $D$-Lipschitz and $\nu(w)=\eta$ for all $w\in \mathcal{F}\setminus \{x_i^\d\}_{i=1}^n=\{y_j\}_{j=1}^k$, we have
\begin{multline*}
\sum_{i=1}^n\sum_{j=1}^n \|\psi(x_i^\d)-\psi(x_j^\d)\|_{\!Y}^p\mu(V_i^\d)\mu(V_j^\d)\\ \ge (1-\eta)^2\sum_{i=1}^n\sum_{j=1}^n d_\MM(x_i^\d,x_j^\d)^p\mu(V_i^\d)\mu(V_j^\d)-2\eta D^p\sum_{u\in \mathcal{F}}\sum_{j=1}^k d_\MM(u,y_j)^p.
\end{multline*}
Since $\alpha>1$, by choosing small enough $\eta$ it follows from this that for each $k\in \N$ there exists a $D$-Lipschitz mapping $\psi_{k}:\{x_i^\d\}_{i=1}^n\cup\{y_j\}_{j=1}^k\to Y$ that satisfies
\begin{equation}\label{psik n sum}
\sum_{i=1}^n\sum_{j=1}^n \|\psi_{k}(x_i^\d)-\psi_k(x_j^\d)\|_{\!Y}^p\mu(V_i^\d)\mu(V_j^\d)\ge \alpha^{-\frac{p}{2}}\sum_{i=1}^n\sum_{j=1}^n d_\MM(x_i^\d,x_j^\d)^p\mu(V_i^\d)\mu(V_j^\d).
\end{equation}
We extend $\psi_k$ to all of $\MM$ by setting it to be identically equal to $\psi_k(x_1^\d)$ on $\MM\setminus (\{x_i^\d\}_{i=1}^n\cup\{y_j\}_{j=1}^k)$.

 Since $\psi_k$ is $D$-Lipschitz on $\{x_i^\d\}_{i=1}^n\cup\{y_j\}_{j=1}^k$, for every $z\in \MM$ the sequence $\{\psi_k(z)-\psi_k(x_1^\d)\}_{k=1}^\infty$ is bounded. We can therefore   define $f_n^\d:\MM\to Y^\mathscr{U}$  by setting  $f_n^\d(z)=\sqrt{\alpha}(\psi_k(z)-\psi_k(x_1^\d))_{k=1}^\infty/\mathscr{U}$ for all $z\in  \MM$. It follows directly from the definitions that $f_n^\d$ has Lipschitz constant $\sqrt{\alpha} D<\alpha D$ on the dense subset $\{x_i^\d\}_{i=1}^n\cup \{y_j\}_{j=1}^\infty$, so its Lipschitz constant is less than $\alpha D$ on all of $\MM$. Also, the desired bound~\eqref{eq:W unified} follows by passing to the ultralimit of~\eqref{psik n sum} as $k\to \infty$, since the number of pairwise distances that appear in the left hand side of~\eqref{psik n sum} is independent of $k$.
\end{proof}

\bibliographystyle{abbrv}
\bibliography{almost-ext}

\def\cprime{$'$} \def\cprime{$'$} \def\cprime{$'$} \def\cprime{$'$}
  \def\cprime{$'$} \def\cprime{$'$} \def\cprime{$'$} \def\cprime{$'$}
  \def\cprime{$'$} \def\cprime{$'$} \def\cprime{$'$} \def\cprime{$'$}
\begin{thebibliography}{100}

\bibitem{ABN11}
I.~Abraham, Y.~Bartal, and O.~Neiman.
\newblock Advances in metric embedding theory.
\newblock {\em Adv. Math.}, 228(6):3026--3126, 2011.

\bibitem{AB15}
F.~Albiac and F.~Baudier.
\newblock Embeddability of snowflaked metrics with applications to the
  nonlinear geometry of the spaces {$L_p$} and {$\ell_p$} for {$0<p<\infty$}.
\newblock {\em J. Geom. Anal.}, 25(1):1--24, 2015.

\bibitem{Alo19}
N.~Alon.
\newblock The diameter of $\mathsf{SL}_n(\mathbb{F}_q)$.
\newblock Preprint, 2019.

\bibitem{AFR85}
N.~Alon, P.~Frankl, and V.~R{\"{o}}dl.
\newblock Geometrical realization of set systems and probabilistic
  communication complexity.
\newblock In {\em 26th Annual Symposium on Foundations of Computer Science},
  pages 277--280. {IEEE} Computer Society, 1985.

\bibitem{AR94}
N.~Alon and Y.~Roichman.
\newblock Random {C}ayley graphs and expanders.
\newblock {\em Random Structures Algorithms}, 5(2):271--284, 1994.

\bibitem{AS16}
N.~Alon and J.~H. Spencer.
\newblock {\em The probabilistic method}.
\newblock Wiley Series in Discrete Mathematics and Optimization. John Wiley \&
  Sons, Inc., Hoboken, NJ, fourth edition, 2016.

\bibitem{AIR18}
A.~Andoni, P.~Indyk, and I.~Razenshteyn.
\newblock Approximate nearest neighbor search in high dimensions.
\newblock To appear in {\em Proceedings of the 2018 International Congress of
  Mathematicians}, preprint available at
  \url{https://arxiv.org/abs/1806.09823}, 2018.

\bibitem{ANN18}
A.~Andoni, A.~Naor, and O.~Neiman.
\newblock Snowflake universality of {W}asserstein spaces.
\newblock {\em Ann. Sci. \'{E}c. Norm. Sup\'{e}r. (4)}, 51(3):657--700, 2018.

\bibitem{ANNRW18}
A.~Andoni, A.~Naor, A.~Nikolov, I.~Razenshteyn, and E.~Waingarten.
\newblock Data-dependent hashing via nonlinear spectral gaps.
\newblock In {\em Proceedings of the 50th Annual {ACM} {SIGACT} Symposium on
  Theory of Computing, {STOC} 2018, Los Angeles, CA, USA, June 25-29, 2018},
  pages 787--800, 2018.

\bibitem{ANNRW-FOCS18}
A.~Andoni, A.~Naor, A.~Nikolov, I.~Razenshteyn, and E.~Waingarten.
\newblock H{\"{o}}lder homeomorphisms and approximate nearest neighbors.
\newblock In {\em 59th {IEEE} Annual Symposium on Foundations of Computer
  Science}, pages 159--169, 2018.

\bibitem{ANNRW18-interpolation}
A.~Andoni, A.~Naor, A.~Nikolov, I.~Razenshteyn, and E.~Waingarten.
\newblock Complex interpolation, {H}\"older homeomorphisms, and algorithmic
  applications.
\newblock Forthcoming manuscript, 2019.

\bibitem{ANNRW18-general}
A.~Andoni, A.~Naor, A.~Nikolov, I.~Razenshteyn, and E.~Waingarten.
\newblock Spectral partitioning of metric spaces.
\newblock Forthcoming manuscript, 2019.

\bibitem{ANRW16}
A.~Andoni, H.~Nguyen, A.~Nikolov, I.~Razenshteyn, and E.~Waingarten.
\newblock Approximate near neighbors for general symmetric norms.
\newblock Extended abstract in STOC'17, available at
  \url{https://arxiv.org/abs/1611.06222}, 2016.

\bibitem{AHM07}
D.~Andr\'{e}n, L.~Hellstr\"{o}m, and K.~Markstr\"{o}m.
\newblock On the complexity of matrix reduction over finite fields.
\newblock {\em Adv. in Appl. Math.}, 39(4):428--452, 2007.

\bibitem{AR92}
J.~Arias-de Reyna and L.~Rodr{\'{\i}}guez-Piazza.
\newblock Finite metric spaces needing high dimension for {L}ipschitz
  embeddings in {B}anach spaces.
\newblock {\em Israel J. Math.}, 79(1):103--111, 1992.

\bibitem{Aro76}
N.~Aronszajn.
\newblock Differentiability of {L}ipschitzian mappings between {B}anach spaces.
\newblock {\em Studia Math.}, 57(2):147--190, 1976.

\bibitem{ANV10}
T.~Austin, A.~Naor, and A.~Valette.
\newblock The {E}uclidean distortion of the lamplighter group.
\newblock {\em Discrete Comput. Geom.}, 44(1):55--74, 2010.

\bibitem{BFGM07}
U.~Bader, A.~Furman, T.~Gelander, and N.~Monod.
\newblock Property ({T}) and rigidity for actions on {B}anach spaces.
\newblock {\em Acta Math.}, 198(1):57--105, 2007.

\bibitem{Bal92}
K.~Ball.
\newblock Markov chains, {R}iesz transforms and {L}ipschitz maps.
\newblock {\em Geom. Funct. Anal.}, 2(2):137--172, 1992.

\bibitem{Bal13}
K.~Ball.
\newblock The {R}ibe programme.
\newblock {\em Ast\'erisque}, (352):Exp. No. 1047, viii, 147--159, 2013.
\newblock S{\'e}minaire Bourbaki. Vol. 2011/2012. Expos{\'e}s 1043--1058.

\bibitem{BCL94}
K.~Ball, E.~A. Carlen, and E.~H. Lieb.
\newblock Sharp uniform convexity and smoothness inequalities for trace norms.
\newblock {\em Invent. Math.}, 115(3):463--482, 1994.

\bibitem{Ban32}
S.~Banach.
\newblock {\em Th\'eorie des op\'erations lin\'eaires}.
\newblock \'Editions Jacques Gabay, Sceaux, 1993.
\newblock Reprint of the 1932 original.

\bibitem{BLMN05}
Y.~Bartal, N.~Linial, M.~Mendel, and A.~Naor.
\newblock On metric {R}amsey-type phenomena.
\newblock {\em Ann. of Math. (2)}, 162(2):643--709, 2005.

\bibitem{BJLPS99}
S.~Bates, W.~B. Johnson, J.~Lindenstrauss, D.~Preiss, and G.~Schechtman.
\newblock Affine approximation of {L}ipschitz functions and nonlinear
  quotients.
\newblock {\em Geom. Funct. Anal.}, 9(6):1092--1127, 1999.

\bibitem{Bau16}
F.~P. Baudier.
\newblock Quantitative nonlinear embeddings into {L}ebesgue sequence spaces.
\newblock {\em J. Topol. Anal.}, 8(1):117--150, 2016.

\bibitem{BL00}
Y.~Benyamini and J.~Lindenstrauss.
\newblock {\em Geometric nonlinear functional analysis. {V}ol. 1}, volume~48 of
  {\em American Mathematical Society Colloquium Publications}.
\newblock American Mathematical Society, Providence, RI, 2000.

\bibitem{BL76}
J.~Bergh and J.~L{\"o}fstr{\"o}m.
\newblock {\em Interpolation spaces. {A}n introduction}.
\newblock Springer-Verlag, Berlin-New York, 1976.
\newblock Grundlehren der Mathematischen Wissenschaften, No. 223.

\bibitem{Bou85}
J.~Bourgain.
\newblock On {L}ipschitz embedding of finite metric spaces in {H}ilbert space.
\newblock {\em Israel J. Math.}, 52(1-2):46--52, 1985.

\bibitem{Bou86}
J.~Bourgain.
\newblock The metrical interpretation of superreflexivity in {B}anach spaces.
\newblock {\em Israel J. Math.}, 56(2):222--230, 1986.

\bibitem{BMW86}
J.~Bourgain, V.~Milman, and H.~Wolfson.
\newblock On type of metric spaces.
\newblock {\em Trans. Amer. Math. Soc.}, 294(1):295--317, 1986.

\bibitem{BDK65}
J.~Bretagnolle, D.~Dacunha-Castelle, and J.-L. Krivine.
\newblock Fonctions de type positif sur les espaces {$L^{p}$}.
\newblock {\em C. R. Acad. Sci. Paris}, 261:2153--2156, 1965.

\bibitem{BB12}
A.~Brudnyi and Y.~Brudnyi.
\newblock {\em Methods of geometric analysis in extension and trace problems.
  {V}olume 2}, volume 103 of {\em Monographs in Mathematics}.
\newblock Birkh\"auser/Springer Basel AG, Basel, 2012.

\bibitem{BS02}
Y.~Brudnyi and P.~Shvartsman.
\newblock Stability of the {L}ipschitz extension property under metric
  transforms.
\newblock {\em Geom. Funct. Anal.}, 12(1):73--79, 2002.

\bibitem{Cal64}
A.-P. Calder{\'o}n.
\newblock Intermediate spaces and interpolation, the complex method.
\newblock {\em Studia Math.}, 24:113--190, 1964.

\bibitem{Cha95}
F.~Chaatit.
\newblock On uniform homeomorphisms of the unit spheres of certain {B}anach
  lattices.
\newblock {\em Pacific J. Math.}, 168(1):11--31, 1995.

\bibitem{Che70}
J.~Cheeger.
\newblock A lower bound for the smallest eigenvalue of the {L}aplacian.
\newblock In {\em Problems in analysis ({P}apers dedicated to {S}alomon
  {B}ochner, 1969)}, pages 195--199. Princeton Univ. Press, Princeton, N. J.,
  1970.

\bibitem{Che16}
Q.~Cheng.
\newblock Sphere equivalence, property {H}, and {B}anach expanders.
\newblock {\em Studia Math.}, 233(1):67--83, 2016.

\bibitem{Chr73}
J.~P.~R. Christensen.
\newblock Measure theoretic zero sets in infinite dimensional spaces and
  applications to differentiability of {L}ipschitz mappings.
\newblock {\em Publ. D\'{e}p. Math. (Lyon)}, 10(2):29--39, 1973.
\newblock Actes du Deuxi\`eme Colloque d'Analyse Fonctionnelle de Bordeaux
  (Univ. Bordeaux, 1973), I, pp. 29--39.

\bibitem{Chu96}
F.~R.~K. Chung.
\newblock Laplacians of graphs and {C}heeger's inequalities.
\newblock In {\em Combinatorics, {P}aul {E}rd\H{o}s is eighty, {V}ol. 2
  ({K}eszthely, 1993)}, volume~2 of {\em Bolyai Soc. Math. Stud.}, pages
  157--172. J\'{a}nos Bolyai Math. Soc., Budapest, 1996.

\bibitem{Cla36}
J.~A. Clarkson.
\newblock Uniformly convex spaces.
\newblock {\em Trans. Amer. Math. Soc.}, 40(3):396--414, 1936.

\bibitem{Cwi78}
M.~Cwikel.
\newblock Complex interpolation spaces, a discrete definition and reiteration.
\newblock {\em Indiana Univ. Math. J.}, 27(6):1005--1009, 1978.

\bibitem{CR82}
M.~Cwikel and S.~Reisner.
\newblock Interpolation of uniformly convex {B}anach spaces.
\newblock {\em Proc. Amer. Math. Soc.}, 84(4):555--559, 1982.

\bibitem{Dah95}
M.~Daher.
\newblock Hom\'{e}omorphismes uniformes entre les sph\`eres unit\'{e} des
  espaces d'interpolation.
\newblock {\em Canad. Math. Bull.}, 38(3):286--294, 1995.

\bibitem{DS97}
G.~David and S.~Semmes.
\newblock {\em Fractured fractals and broken dreams}, volume~7 of {\em Oxford
  Lecture Series in Mathematics and its Applications}.
\newblock The Clarendon Press, Oxford University Press, New York, 1997.
\newblock Self-similar geometry through metric and measure.

\bibitem{Day44}
M.~M. Day.
\newblock Uniform convexity in factor and conjugate spaces.
\newblock {\em Ann. of Math. (2)}, 45:375--385, 1944.

\bibitem{LS17}
T.~de~Laat and M.~de~la Salle.
\newblock Banach space actions and ${L}^2$-spectral gap.
\newblock Preprint, available at \url{https://arxiv.org/abs/1705.03296}, 2017.

\bibitem{Enf69}
P.~Enflo.
\newblock On the nonexistence of uniform homeomorphisms between
  {$L_{p}$}-spaces.
\newblock {\em Ark. Mat.}, 8:103--105, 1969.

\bibitem{Enf72}
P.~Enflo.
\newblock Banach spaces which can be given an equivalent uniformly convex norm.
\newblock {\em Israel J. Math.}, 13:281--288 (1973), 1972.

\bibitem{Enf76}
P.~Enflo.
\newblock Uniform homeomorphisms between {B}anach spaces.
\newblock In {\em S\'eminaire {M}aurey-{S}chwartz (1975--1976), {E}spaces,
  {$L^{p}$}, applications radonifiantes et g\'eom\'etrie des espaces de
  {B}anach, {E}xp. {N}o. 18}, page~7. Centre Math., \'Ecole Polytech.,
  Palaiseau, 1976.

\bibitem{EN18}
A.~Eskenazis and A.~Naor.
\newblock On coarse and uniform embeddings into ${L}_p$.
\newblock Forthcoming manuscript, 2018.

\bibitem{Fig76}
T.~Figiel.
\newblock On the moduli of convexity and smoothness.
\newblock {\em Studia Math.}, 56(2):121--155, 1976.

\bibitem{FP74}
T.~Figiel and G.~Pisier.
\newblock S\'eries al\'eatoires dans les espaces uniform\'ement convexes ou
  uniform\'ement lisses.
\newblock {\em C. R. Acad. Sci. Paris S\'er. A}, 279:611--614, 1974.

\bibitem{Fre06}
M.~Fr{\'e}chet.
\newblock Sur quelques points du calcul fonctionnel.
\newblock {\em Rend. Circ. Mat. Palermo}, 22:1--74, 1906.

\bibitem{Gel38}
I.~Gelfand.
\newblock {Abstrakte Funktionen und lineare Operatoren.}
\newblock {\em Matematicheski\u{\i} Sbornik}, 4(2):235--286, 1938.

\bibitem{GN12}
R.~I. Grigorchuk and P.~W. Nowak.
\newblock Diameters, distortion, and eigenvalues.
\newblock {\em European J. Combin.}, 33(7):1574--1587, 2012.

\bibitem{Gro00}
M.~Gromov.
\newblock Spaces and questions.
\newblock {\em Geom. Funct. Anal.}, (Special Volume, Part I):118--161, 2000.
\newblock GAFA 2000 (Tel Aviv, 1999).

\bibitem{Gro03}
M.~Gromov.
\newblock Random walk in random groups.
\newblock {\em Geom. Funct. Anal.}, 13(1):73--146, 2003.

\bibitem{Gro77}
J.~L. Gross.
\newblock Every connected regular graph of even degree is a {S}chreier coset
  graph.
\newblock {\em J. Combinatorial Theory Ser. B}, 22(3):227--232, 1977.

\bibitem{Han56}
O.~Hanner.
\newblock On the uniform convexity of {$L^p$} and {$l^p$}.
\newblock {\em Ark. Mat.}, 3:239--244, 1956.

\bibitem{Hat02}
A.~Hatcher.
\newblock {\em Algebraic topology}.
\newblock Cambridge University Press, Cambridge, 2002.

\bibitem{Hei01}
J.~Heinonen.
\newblock {\em Lectures on analysis on metric spaces}.
\newblock Universitext. Springer-Verlag, New York, 2001.

\bibitem{Hei80}
S.~Heinrich.
\newblock Ultraproducts in {B}anach space theory.
\newblock {\em J. Reine Angew. Math.}, 313:72--104, 1980.

\bibitem{HP15}
M.~Herman and J.~Pakianathan.
\newblock On the distribution of distances in homogeneous compact metric
  spaces.
\newblock {\em Topology Appl.}, 193:97--99, 2015.

\bibitem{HLW06}
S.~Hoory, N.~Linial, and A.~Wigderson.
\newblock Expander graphs and their applications.
\newblock {\em Bull. Amer. Math. Soc. (N.S.)}, 43(4):439--561 (electronic),
  2006.

\bibitem{IN05}
H.~Izeki and S.~Nayatani.
\newblock Combinatorial harmonic maps and discrete-group actions on {H}adamard
  spaces.
\newblock {\em Geom. Dedicata}, 114:147--188, 2005.

\bibitem{Jam72}
R.~C. James.
\newblock Super-reflexive {B}anach spaces.
\newblock {\em Canad. J. Math.}, 24:896--904, 1972.

\bibitem{Joh48}
F.~John.
\newblock Extremum problems with inequalities as subsidiary conditions.
\newblock In {\em Studies and {E}ssays {P}resented to {R}. {C}ourant on his
  60th {B}irthday, {J}anuary 8, 1948}, pages 187--204. Interscience Publishers,
  Inc., New York, N. Y., 1948.

\bibitem{JL84}
W.~B. Johnson and J.~Lindenstrauss.
\newblock Extensions of {L}ipschitz mappings into a {H}ilbert space.
\newblock In {\em Conference in modern analysis and probability ({N}ew {H}aven,
  {C}onn., 1982)}, volume~26 of {\em Contemp. Math.}, pages 189--206. Amer.
  Math. Soc., Providence, RI, 1984.

\bibitem{JL01}
W.~B. Johnson and J.~Lindenstrauss.
\newblock Basic concepts in the geometry of {B}anach spaces.
\newblock In {\em Handbook of the geometry of {B}anach spaces, {V}ol. {I}},
  pages 1--84. North-Holland, Amsterdam, 2001.

\bibitem{JLS87}
W.~B. Johnson, J.~Lindenstrauss, and G.~Schechtman.
\newblock On {L}ipschitz embedding of finite metric spaces in low-dimensional
  normed spaces.
\newblock In {\em Geometrical aspects of functional analysis (1985/86)}, volume
  1267 of {\em Lecture Notes in Math.}, pages 177--184. Springer, Berlin, 1987.

\bibitem{JV14}
P.-N. Jolissaint and A.~Valette.
\newblock {$L^p$}-distortion and {$p$}-spectral gap of finite graphs.
\newblock {\em Bull. Lond. Math. Soc.}, 46(2):329--341, 2014.

\bibitem{Kal08-survey}
N.~J. Kalton.
\newblock The nonlinear geometry of {B}anach spaces.
\newblock {\em Rev. Mat. Complut.}, 21(1):7--60, 2008.

\bibitem{Kas05}
M.~Kassabov.
\newblock Kazhdan constants for {${\rm SL}_n({\Bbb Z})$}.
\newblock {\em Internat. J. Algebra Comput.}, 15(5-6):971--995, 2005.

\bibitem{KN06}
S.~Khot and A.~Naor.
\newblock Nonembeddability theorems via {F}ourier analysis.
\newblock {\em Math. Ann.}, 334(4):821--852, 2006.

\bibitem{Kirsz34}
M.~D. Kirszbraun.
\newblock \"{U}ber die zusammenziehenden und {L}ipschitzchen
  {T}ransformationen.
\newblock {\em Fundam. Math.}, 22:77--108, 1934.

\bibitem{Kon12}
T.~Kondo.
\newblock {${\rm CAT}(0)$} spaces and expanders.
\newblock {\em Math. Z.}, 271(1-2):343--355, 2012.

\bibitem{Laa02}
T.~J. Laakso.
\newblock Plane with {$A_\infty$}-weighted metric not bi-{L}ipschitz embeddable
  to {${\Bbb R}^N$}.
\newblock {\em Bull. London Math. Soc.}, 34(6):667--676, 2002.

\bibitem{Laf08}
V.~Lafforgue.
\newblock Un renforcement de la propri\'{e}t\'{e} ({T}).
\newblock {\em Duke Math. J.}, 143(3):559--602, 2008.

\bibitem{Laf09}
V.~Lafforgue.
\newblock Propri\'et\'e ({T}) renforc\'ee banachique et transformation de
  {F}ourier rapide.
\newblock {\em J. Topol. Anal.}, 1(3):191--206, 2009.

\bibitem{LN05}
J.~R. Lee and A.~Naor.
\newblock Extending {L}ipschitz functions via random metric partitions.
\newblock {\em Invent. Math.}, 160(1):59--95, 2005.

\bibitem{Lin63}
J.~Lindenstrauss.
\newblock On the modulus of smoothness and divergent series in {B}anach spaces.
\newblock {\em Michigan Math. J.}, 10:241--252, 1963.

\bibitem{LT77}
J.~Lindenstrauss and L.~Tzafriri.
\newblock {\em Classical {B}anach spaces. {I}}.
\newblock Springer-Verlag, Berlin-New York, 1977.
\newblock Sequence spaces, Ergebnisse der Mathematik und ihrer Grenzgebiete,
  Vol. 92.

\bibitem{LT79}
J.~Lindenstrauss and L.~Tzafriri.
\newblock {\em Classical {B}anach spaces. {II}}, volume~97 of {\em Ergebnisse
  der Mathematik und ihrer Grenzgebiete [Results in Mathematics and Related
  Areas]}.
\newblock Springer-Verlag, Berlin-New York, 1979.
\newblock Function spaces.

\bibitem{LLR95}
N.~Linial, E.~London, and Y.~Rabinovich.
\newblock The geometry of graphs and some of its algorithmic applications.
\newblock {\em Combinatorica}, 15(2):215--245, 1995.

\bibitem{LM00}
N.~Linial and A.~Magen.
\newblock Least-distortion {E}uclidean embeddings of graphs: products of cycles
  and expanders.
\newblock {\em J. Combin. Theory Ser. B}, 79(2):157--171, 2000.

\bibitem{LMN02}
N.~Linial, A.~Magen, and A.~Naor.
\newblock Girth and {E}uclidean distortion.
\newblock {\em Geom. Funct. Anal.}, 12(2):380--394, 2002.

\bibitem{Lio60}
J.-L. Lions.
\newblock Une construction d'espaces d'interpolation.
\newblock {\em C. R. Acad. Sci. Paris}, 251:1853--1855, 1960.

\bibitem{Man72}
P.~Mankiewicz.
\newblock On {L}ipschitz mappings between {F}r\'{e}chet spaces.
\newblock {\em Studia Math.}, 41:225--241, 1972.

\bibitem{Mat90}
J.~Matou{\v{s}}ek.
\newblock Extension of {L}ipschitz mappings on metric trees.
\newblock {\em Comment. Math. Univ. Carolin.}, 31(1):99--104, 1990.

\bibitem{Mat92}
J.~Matou{\v{s}}ek.
\newblock Note on bi-{L}ipschitz embeddings into normed spaces.
\newblock {\em Comment. Math. Univ. Carolin.}, 33(1):51--55, 1992.

\bibitem{Mat96}
J.~Matou{\v{s}}ek.
\newblock On the distortion required for embedding finite metric spaces into
  normed spaces.
\newblock {\em Israel J. Math.}, 93:333--344, 1996.

\bibitem{Mat97}
J.~Matou{\v{s}}ek.
\newblock On embedding expanders into {$l_p$} spaces.
\newblock {\em Israel J. Math.}, 102:189--197, 1997.

\bibitem{Mat02}
J.~Matou{\v{s}}ek.
\newblock {\em Lectures on discrete geometry}, volume 212 of {\em Graduate
  Texts in Mathematics}.
\newblock Springer-Verlag, New York, 2002.

\bibitem{Maz29}
S.~Mazur.
\newblock Une remarque sur l'hom\'eomorphie des champs fonctionels.
\newblock {\em Studia Math.}, 1:83--85, 1929.

\bibitem{MN04}
M.~Mendel and A.~Naor.
\newblock Euclidean quotients of finite metric spaces.
\newblock {\em Adv. Math.}, 189(2):451--494, 2004.

\bibitem{MN08}
M.~Mendel and A.~Naor.
\newblock Metric cotype.
\newblock {\em Ann. of Math. (2)}, 168(1):247--298, 2008.

\bibitem{MN13-convexity}
M.~Mendel and A.~Naor.
\newblock Markov convexity and local rigidity of distorted metrics.
\newblock {\em J. Eur. Math. Soc. (JEMS)}, 15(1):287--337, 2013.

\bibitem{MN13-bary}
M.~Mendel and A.~Naor.
\newblock Spectral calculus and {L}ipschitz extension for barycentric metric
  spaces.
\newblock {\em Anal. Geom. Metr. Spaces}, 1:163--199, 2013.

\bibitem{MN14}
M.~Mendel and A.~Naor.
\newblock Nonlinear spectral calculus and super-expanders.
\newblock {\em Publ. Math. Inst. Hautes \'Etudes Sci.}, 119:1--95, 2014.

\bibitem{MN15}
M.~Mendel and A.~Naor.
\newblock Expanders with respect to {H}adamard spaces and random graphs.
\newblock {\em Duke Math. J.}, 164(8):1471--1548, 2015.

\bibitem{Mil38}
D.~{Milman}.
\newblock {On some criteria for the regularity of spaces of the type $(B)$.}
\newblock {\em {C. R. (Dokl.) Acad. Sci. URSS, n. Ser.}}, 20:243--246, 1938.

\bibitem{Mil64}
J.~Milnor.
\newblock On the {B}etti numbers of real varieties.
\newblock {\em Proc. Amer. Math. Soc.}, 15:275--280, 1964.

\bibitem{Mim15}
M.~Mimura.
\newblock Sphere equivalence, {B}anach expanders, and extrapolation.
\newblock {\em Int. Math. Res. Not. IMRN}, (12):4372--4391, 2015.

\bibitem{Min70}
G.~J. Minty.
\newblock On the extension of {L}ipschitz, {L}ipschitz-{H}\"older continuous,
  and monotone functions.
\newblock {\em Bull. Amer. Math. Soc.}, 76:334--339, 1970.

\bibitem{Nao12}
A.~Naor.
\newblock An introduction to the {R}ibe program.
\newblock {\em Jpn. J. Math.}, 7(2):167--233, 2012.

\bibitem{Nao12-azuma}
A.~Naor.
\newblock On the {B}anach-space-valued {A}zuma inequality and small-set
  isoperimetry of {A}lon-{R}oichman graphs.
\newblock {\em Combin. Probab. Comput.}, 21(4):623--634, 2012.

\bibitem{Nao14}
A.~Naor.
\newblock Comparison of metric spectral gaps.
\newblock {\em Anal. Geom. Metr. Spaces}, 2:1--52, 2014.

\bibitem{Nao15}
A.~Naor.
\newblock Uniform nonextendability from nets.
\newblock {\em C. R. Math. Acad. Sci. Paris}, 353(11):991--994, 2015.

\bibitem{Nao16-riesz}
A.~Naor.
\newblock Discrete {R}iesz transforms and sharp metric {$X_p$} inequalities.
\newblock {\em Ann. of Math. (2)}, 184(3):991--1016, 2016.

\bibitem{Nao17}
A.~Naor.
\newblock A spectral gap precludes low-dimensional embeddings.
\newblock In {\em 33rd {I}nternational {S}ymposium on {C}omputational
  {G}eometry}, volume~77 of {\em LIPIcs. Leibniz Int. Proc. Inform.}, pages
  Art. No. 50, 16. Schloss Dagstuhl. Leibniz-Zent. Inform., Wadern, 2017.
\newblock Preprint (containing additional material that did not appear
  elsewhere) available at \url{https://arxiv.org/abs/1611.08861}.

\bibitem{Nao18}
A.~Naor.
\newblock Metric dimension reduction: {A} snapshot of the {R}ibe program.
\newblock In {\em Proceedings of the 2018 {I}nternational {C}ongress of
  {M}athematicians, Rio de Janeiro. {V}olume {I}}, pages 767--846, 2018.

\bibitem{NPSS06}
A.~Naor, Y.~Peres, O.~Schramm, and S.~Sheffield.
\newblock Markov chains in smooth {B}anach spaces and {G}romov-hyperbolic
  metric spaces.
\newblock {\em Duke Math. J.}, 134(1):165--197, 2006.

\bibitem{NR17}
A.~Naor and Y.~Rabani.
\newblock On {L}ipschitz extension from finite subsets.
\newblock {\em Israel J. Math.}, 219(1):115--161, 2017.

\bibitem{NRS05}
A.~Naor, Y.~Rabani, and A.~Sinclair.
\newblock Quasisymmetric embeddings, the observable diameter, and expansion
  properties of graphs.
\newblock {\em J. Funct. Anal.}, 227(2):273--303, 2005.

\bibitem{NS16-Xp}
A.~Naor and G.~Schechtman.
\newblock Metric {$X_p$} inequalities.
\newblock {\em Forum Math. Pi}, 4:e3, 81 pp., 2016.

\bibitem{NS11}
A.~Naor and L.~Silberman.
\newblock Poincar\'e inequalities, embeddings, and wild groups.
\newblock {\em Compos. Math.}, 147(5):1546--1572, 2011.

\bibitem{NR09}
I.~Newman and Y.~Rabinovich.
\newblock Hard metrics from {C}ayley graphs of abelian groups.
\newblock {\em Theory Comput.}, 5:125--134, 2009.

\bibitem{OS94}
E.~Odell and T.~Schlumprecht.
\newblock The distortion problem.
\newblock {\em Acta Math.}, 173(2):259--281, 1994.

\bibitem{Ost13}
M.~I. Ostrovskii.
\newblock {\em Metric embeddings. Bilipschitz and coarse embeddings into Banach
  spaces}, volume~49 of {\em De Gruyter Studies in Mathematics}.
\newblock De Gruyter, Berlin, 2013.

\bibitem{Oza04}
N.~Ozawa.
\newblock A note on non-amenability of {$\mathscr{B}(l_p)$} for {$p=1,2$}.
\newblock {\em Internat. J. Math.}, 15(6):557--565, 2004.

\bibitem{Pet39}
B.~J. Pettis.
\newblock A proof that every uniformly convex space is reflexive.
\newblock {\em Duke Math. J.}, 5(2):249--253, 1939.

\bibitem{Pis75}
G.~Pisier.
\newblock Martingales with values in uniformly convex spaces.
\newblock {\em Israel J. Math.}, 20(3-4):326--350, 1975.

\bibitem{Pis79-lattices-interpolation}
G.~Pisier.
\newblock Some applications of the complex interpolation method to {B}anach
  lattices.
\newblock {\em J. Analyse Math.}, 35:264--281, 1979.

\bibitem{Pis10}
G.~Pisier.
\newblock Complex interpolation between {H}ilbert, {B}anach and operator
  spaces.
\newblock {\em Mem. Amer. Math. Soc.}, 208(978):vi+78, 2010.

\bibitem{Rab08}
Y.~Rabinovich.
\newblock On average distortion of embedding metrics into the line.
\newblock {\em Discrete Comput. Geom.}, 39(4):720--733, 2008.

\bibitem{Ray02}
Y.~Raynaud.
\newblock On ultrapowers of non commutative {$L_p$} spaces.
\newblock {\em J. Operator Theory}, 48(1):41--68, 2002.

\bibitem{Rib76}
M.~Ribe.
\newblock On uniformly homeomorphic normed spaces.
\newblock {\em Ark. Mat.}, 14(2):237--244, 1976.

\bibitem{Ric15}
E.~Ricard.
\newblock H\"{o}lder estimates for the noncommutative {M}azur maps.
\newblock {\em Arch. Math. (Basel)}, 104(1):37--45, 2015.

\bibitem{Rie27}
M.~Riesz.
\newblock Sur les maxima des formes bilin\'eaires et sur les fonctionnelles
  lin\'eaires.
\newblock {\em Acta Math.}, 49(3-4):465--497, 1927.

\bibitem{Ril05}
T.~R. Riley.
\newblock Navigating in the {C}ayley graphs of {${\rm SL}_N(\Bbb Z)$} and
  {${\rm SL}_N(\Bbb F_p)$}.
\newblock {\em Geom. Dedicata}, 113:215--229, 2005.

\bibitem{Sch38}
I.~J. Schoenberg.
\newblock Metric spaces and positive definite functions.
\newblock {\em Trans. Amer. Math. Soc.}, 44(3):522--536, 1938.

\bibitem{Ste56}
E.~M. Stein.
\newblock Interpolation of linear operators.
\newblock {\em Trans. Amer. Math. Soc.}, 83:482--492, 1956.

\bibitem{Tho65}
R.~Thom.
\newblock Sur l'homologie des vari\'et\'es alg\'ebriques r\'eelles.
\newblock In {\em Differential and {C}ombinatorial {T}opology ({A} {S}ymposium
  in {H}onor of {M}arston {M}orse)}, pages 255--265. Princeton Univ. Press,
  Princeton, N.J., 1965.

\bibitem{Tho48}
G.~O. Thorin.
\newblock Convexity theorems generalizing those of {M}. {R}iesz and {H}adamard
  with some applications.
\newblock {\em Comm. Sem. Math. Univ. Lund [Medd. Lunds Univ. Mat. Sem.]},
  9:1--58, 1948.

\bibitem{WW75}
J.~H. Wells and L.~R. Williams.
\newblock {\em Embeddings and extensions in analysis}.
\newblock Springer-Verlag, New York, 1975.
\newblock Ergebnisse der Mathematik und ihrer Grenzgebiete, Band 84.

\end{thebibliography}

 \end{document}